\documentclass[]{article}
\usepackage{geometry}
\usepackage{amsmath}
\usepackage{amssymb}

\usepackage{bigfoot}

\usepackage{enumerate}

\usepackage{fancyhdr}
\usepackage[hidelinks]{hyperref} 

\usepackage{amsthm} 
\newtheorem{theorem}{Theorem}
\numberwithin{theorem}{section}
\newtheorem*{theorem*}{Theorem}
\newtheorem{definition}[theorem]{Definition}
\newtheorem{cor}[theorem]{Corollary}
\newtheorem{proposition}[theorem]{Proposition}
\newtheorem{lemma}[theorem]{Lemma}
\newtheorem{claim}[theorem]{Claim}
\newtheorem{remark}[theorem]{Remark}

\newtheorem{example}[theorem]{Example}

\newtheorem*{example*}{Example}

\newenvironment{subproof}[1][\proofname]{%
  \begin{proof}[#1]%
}{%
  \end{proof}%
}

\newtheorem{innercustomgeneric}{\customgenericname}
\providecommand{\customgenericname}{}
\newcommand{\newcustomtheorem}[2]{%
	\newenvironment{#1}[1]
	{%
		\renewcommand\customgenericname{#2}%
		\renewcommand\theinnercustomgeneric{##1}%
		\innercustomgeneric
	}
	{\endinnercustomgeneric}
}

\newcustomtheorem{customthm}{Theorem}
\newcustomtheorem{customlemma}{Lemma}
\newtheorem{assumption}[theorem]{Assumption}

\usepackage{mathrsfs}

\usepackage{xcolor} 

\usepackage{caption}
\usepackage{subcaption}
\usepackage{graphicx}

\usepackage{algorithm}
\usepackage{algpseudocode} 

\usepackage{bbm}

\usepackage{bm}

\usepackage{mathtools}

\usepackage[toc,page]{appendix}

\usepackage{caption}
\usepackage{subcaption}

\usepackage{tikz}
\usetikzlibrary{fit,shapes,arrows,positioning}

\tikzset{decision/.style={diamond, draw, text width=4.5em, text badly centered, inner sep=0pt}}
\tikzset{inout/.style={ellipse, draw, text width=7em, text centered, rounded corners,
 minimum width=3.5cm}}
\tikzset{block/.style={rectangle, draw, text width=12em, text centered, rounded corners,
 minimum width=3.5cm}}
 \tikzset{block1/.style={rectangle, draw,fill=gray!20, text width=10em, text centered, rounded corners,
  minimum width=3.5cm}}
\tikzset{line/.style={draw, -latex}}

\title{Bridging the gap: symplecticity and low regularity in Runge--Kutta resonance-based schemes}
\author{Georg Maierhofer\thanks{Corresponding author; georg.maierhofer@maths.ox.ac.uk} \and Katharina Schratz\thanks{katharina.schratz@sorbonne-universite.fr}}
\date{\footnotemark[1]\ Mathematical Institute, University of Oxford, OX2 6GG United Kingdom\\\ \vspace{-0.2cm}\\
\footnotemark[2]\ Laboratoire Jacques-Louis Lions (UMR 7598), Sorbonne Universit\'e, UPMC,\\ 4 place Jussieu, 75005 France\\\ \vspace{0.2cm}\\\today}
\newcommand{\dd}{\mathrm{d}}
\newcommand{\e}{\mathrm{e}}
\newcommand{\re}{\mathrm{Re}}

\makeatletter
\newcommand{\doublewidetilde}[1]{{%
  \mathpalette\double@widetilde{#1}%
}}
\newcommand{\double@widetilde}[2]{%
  \sbox\z@{$\m@th#1\widetilde{#2}$}%
  \ht\z@=.9\ht\z@
  \widetilde{\box\z@}%
}
\makeatother

\newcommand\blfootnote[1]{%
  \begingroup
  \renewcommand\thefootnote{}\footnote{#1}%
  \addtocounter{footnote}{-1}%
  \endgroup
}

\begin{document}

\pagenumbering{arabic}
	
\maketitle
\vspace{-1cm}
\begin{abstract}
Recent years have seen an increasing amount of research devoted to the development of so-called resonance-based methods for dispersive nonlinear partial differential equations. In many situations, this new class of methods allows for approximations in a much more general setting (e.g. for rough data) than, for instance, classical splitting or exponential integrator methods. However, they lack one important property:
the \emph{preservation of geometric properties of the flow}. This is particularly drastic in the case of the Korteweg--de Vries (KdV) equation and the nonlinear Schr\"odinger equation (NLSE) which are fundamental models in the broad field of dispersive infinite-dimensional Hamiltonian systems, possessing infinitely many conserved quantities, an important property which we wish to capture - at least up to some degree - also on the discrete level. Nowadays, a wide range of structure preserving integrators for Hamiltonian systems are available, however, typically these existing algorithms can only approximate highly regular solutions efficiently. State-of-the-art low-regularity integrators, on the other hand,  poorly preserve the geometric structure of the underlying PDE. {In this work we introduce a novel framework, so-called Runge--Kutta resonance-based methods, {for a large class of dispersive nonlinear equations} which} {incorporate a much larger amount of degrees of freedom than prior resonance-based schemes while featuring similarly favourable low-regularity convergence properties}. {In particular, for the KdV and NLSE case, we} are able to bridge the gap between low regularity and structure preservation {by characterising} a large class of symplectic (in the Hamiltonian picture) resonance-based methods for both equations that allow for low-regularity approximations to the solution while preserving the underlying geometric structure of the continuous problem on the discrete level. 
\end{abstract}

\tableofcontents
\section{Introduction}
In this work we focus on the numerical approximation of solutions to {the following class of dispersive nonlinear partial differential equations.} {\begin{align}\label{eqn:general_dispersive_nonlinear_eqn}
\begin{cases}
i \partial_t u(t,x) +   \mathcal{L}\left(\nabla\right) u(t,x) =\vert \nabla\vert^\alpha \rho\left(u(t,x), \overline u(t,x)\right), &\quad (t,x)\in [0,T]\times \mathbb{T}^d,\\
u(0,x) = u_0(x),&\quad x\in\mathbb{T}^d,
\end{cases}    
\end{align}
where $\mathcal{L}$ is a dispersive linear operator (possibly oscillatory in the small parameter $\varepsilon$), $\rho$ is a polynomial nonlinearity which may depend on both $u$ and $\overline{u}$, $d\in\mathbb{N}$ and $\mathbb{T}^d=\left(\mathbb{R}/(2\pi\mathbb{Z})\right)^d$ is the $d$-dimensional torus. Two classical examples of such equations are the Korteweg--De Vries (KdV) equation and the nonlinear Schr\"odinger equation (NLSE), which are described in further detail in Section~\ref{sec:construction_of_rk_res_schemes} below.} We seek to construct numerical algorithms which can
\begin{itemize}
    \item[(I)] approximate the time dynamics of the partial differential equation under low regularity assumptions, i.e., allowing for rough data, and at the same time, 
    \item[(II)]  preserve the underlying geometric structure of the continuous problem.
\end{itemize}The numerical solution of   nonlinear dispersive equations with low-regularity data is thereby an ongoing challenge of its own right: Classical numerical {time-integrators} are developed with analytic solutions in mind \cite{iserles_2008}. For this reason, classical integrators require a significant amount of regularity of the solution to converge reliably. 
The  necessity for smooth solutions is not just a theoretical technicality: The severe order reduction of classical methods in the low-regularity setting is indeed  observed in practice \cite{bruned_schratz_2022,hofmanova2017exponential,ostermann2018low} (see also Section~\ref{sec:numerical_experiments}) leading to instability, loss of convergence and huge computational costs. Over the recent decade, this challenge has motivated the idea of tailored low-regularity integrators which are able to provide reliable  convergence rates in a more general setting allowing for, e.g., rough initial data \cite{li2022unfiltered,wu2021embedded}. A particular class of integrators which has proven successful in a range of applications are so-called {resonance-based} methods \cite{Bronsardbruned2022,bruned_schratz_2022,knoeller2019fourier,ostermann2018low,limaschratz2022,ostermann2021error,roussetschratz2021,rousset2021general}.

While in many situations this new class of integrators allows for approximations of much rougher solutions than, for instance, classical splitting methods \cite{holden2011operator,holden2013operator}, previous resonance-based approaches lack one important property: the preservation of geometric structures. This is particularly drastic in case of the KdV equation and the {one-dimensional NLSE} which are completely integrable, possessing infinitely many conserved quantities \cite{dunajski2010solitons,iandoli2022}, an important property which we wish to  capture - at least up to some degree - also on the level of the discretisation. A revolutionary step in this direction was taken by the theory of {geometric numerical integration}  \cite{engquist2009highly,hairer2013geometric,leimkuhler_reich_2005,sanz2018numerical,shen2019geometric} resulting in the development of a wide range of structure-preserving algorithms firstly for dynamical systems and later also for partial differential equations with conservation laws \cite{ascher2005symplectic,celledoni2008symmetric,faou2012geometric,ROUHI199518,ROUHI1996209,shen2019geometric}. 
However, in general, these methods rely heavily on the treatment of highly regular solutions to achieve guaranteed convergence. 
State-of-the-art low-regularity integrators,  whilst allowing for approximations for rougher data,  on the other hand, come with the  major drawback of poor preservation of geometric structure of the underlying PDE (cf. \cite{bruned_schratz_2022,ostermann2018low} and also Section~\ref{sec:numerical_experiments}).

 Albeit some recent work has started to look at symmetric low-regularity integrators for specific equations \cite{fengmaierhoferschratz23,banicamaierhoferschratz22,AlamaBronsard23}, so far no low-regularity integrators exist which can provably preserve first integrals of the underlying equation. Thus, generally speaking, until now structure preservation seemed out of reach for low-regularity integrators, and the low regularity regime was out of reach for structure-preserving algorithms. With this work, we aim to take a central step towards bridging this gap, {by introducing the first symplectic low-regularity integrators for the NLSE and the KdV equation} which exactly preserve the quadratic first integrals of these equations.

In particular, we introduce a novel point-of-view on the construction of resonance-based schemes, which is motivated by the design of methods for highly oscillatory quadrature \cite{deano2017computing,iserles2004,iserles2005,iserlesnorsett2004,maierhofer2020extended} and classical Runge--Kutta schemes. {This point-of-view highlights the central idea of resonance-based methods which is to use a suitable averaging process based on Duhamel's formula to mitigate regularity requirements in the numerical approximation. In this way, resonance-based methods can be related to the average vector field method \cite{mclachlan1999geometric,quispel2008new,celledoni2009energy} and it is natural to try and incorporate more degrees of freedom as compared to explicit schemes in prior work on low-regularity integration.} This novel construction leads to a broad class of low-regularity integrators, which we call Runge--Kutta resonance-based schemes {(RK resonance-based schemes)}, that incorporate a much larger amount of degrees of freedom than previous resonance-based schemes and as a result also lend themselves to the inclusion of structure preserving properties, without breaking the low-regularity approximation properties.

In contrast to classical resonance-based methods \cite{bruned_schratz_2022,ostermann2018low,roussetschratz2021,rousset2021general}, which are all explicit, RK resonance-based schemes allow for an implicit nature of the numerical methods, which is shown in Section~\ref{sec:examples_of_symplectic_integrators} to be a necessary condition in our characterisation of symplectic RK resonance-based schemes. This is very much in the spirit of classical Runge--Kutta methods which are necessarily implicit if they are symmetric or symplectic (cf. \cite{Kulikov2003,hairer2013geometric}). {Our construction of these RK resonance-based schemes (exhibited in {Section~\ref{sec:construction_of_rk_res_schemes})} {involves} two central steps which differ significantly from prior constructions: {firstly,} we replace classical left-endpoint approximations in these expressions by {interpolating polynomials and ideas from highly oscillatory quadrature, which results in an analogue of the construction of classical Runge--Kutta methods that are based on classical quadrature}}; {secondly,} we revisit the the low-regularity kernel approximation in Duhamel's formula for the construction of resonance-based schemes to ensure that our new approximations respect the symplectic structure of the original equation {(this is necessary only in the case of the NLSE)}.

The remainder of this manuscript is structured as follows. In {Section}~\ref{sec:construction_of_rk_res_schemes} {we recall some relevant background on prior work in resonance-based schemes before introducing our novel construction of RK resonance-based methods firstly for the KdV equation and then for the general class \eqref{eqn:general_dispersive_nonlinear_eqn}. We demonstrate how these novel integrators connect to existing work in this field, in particular demonstrating how certain prior low-regularity integrators can be seen as RK resonance-based methods. This is followed in Section~\ref{sec:symplectic_RK_res_based_schemes} by a more specific construction of a subclass of symplectic resonance-based methods for the KdV equation and the NLSE which} requires an additional layer in the construction {namely a symplectic kernel approximation in Duhamel's formula}. {The structure preservation properties of these schemes are analysed in further detail in {Section}~\ref{sec:structure_preservation}, and a characterisation of all symplectic and quadratic invariant preserving schemes in this class is provided for both equations.} We provide several examples of schemes in this class, and in particular study the low-regularity convergence properties of the so-called resonance-based midpoint rule in further detail for both equations in Sections~\ref{sec:convergence_analysis_KdV} \& \ref{sec:convergence_analysis_NLSE}. {An outlook of a more general framework for convergence analysis of RK resonance-based schemes is provided in Section~\ref{sec:general_idea_convergence_analysis}.} Finally, our theoretical findings are {underlined} in computational experiments which are described in Section~\ref{sec:numerical_experiments} and concluding remarks indicating possible future directions for this research are provided in Section~\ref{sec:conclusions}.


\section{Construction of Runge--Kutta resonance-based schemes}\label{sec:construction_of_rk_res_schemes}
In the following we introduce a new class of integrators, which extends previous work on resonance-based schemes (\cite{hofmanova2017exponential,ostermann2018low,bruned_schratz_2022}) {by taking the following two novel steps in the construction:
\begin{enumerate}
\item {in the present section,} we follow ideas from highly oscillatory quadrature and construct multilevel schemes by considering interpolating polynomials of the twisted variable contributions to the variations of constants expression of the solution;
\item {in Section~\ref{sec:symplectic_RK_res_based_schemes},} we introduce a novel low-regularity kernel approximation in Duhamel's formula which respects the symplectic structure of the original equation {for two examples of Hamiltonian dispersive nonlinear PDEs, the KdV equation and the NLSE}.
\end{enumerate}}
These features allow us to incorporate more degrees of freedom and therefore facilitate structure preservation into our low-regularity integrators. In particular{, in Section~\ref{sec:structure_preservation},} we characterise a subclass of these low-regularity schemes which is able to exactly preserve the quadratic first integrals {and the symplectic forms of the KdV and the NLS equations.}
\subsection{Preliminaries}\label{sec:preliminaries}
Before diving into the construction of this novel class of resonance-based schemes, let us set the scene by {being more specific about the types of equations we consider and by} recalling some central properties and tools that will be used in our later analysis. To begin with we will conduct most of the convergence analysis on periodic Sobolev spaces $H^r=H^r(\mathbb{T}),r\geq 0,$ with norm
\begin{align*}
	\|u\|_{H^{r}}^2:=\sum_{m\in\mathbb{Z}}\langle m\rangle^{2r}|\hat{u}_m|^2,\ \text{where}\ \langle m\rangle=\begin{cases}
		|m|,& m\neq 0,\\
		1, & m={0,}
	\end{cases}
\end{align*}
where the Fourier coefficient $\hat{u}_m$ is given by
\begin{align*}
	\hat{u}_m=\frac{1}{2\pi}\int_{\mathbb{T}} e^{-i m x} u(x) \dd x.
\end{align*}
{Throughout this manuscript the following well-known bilinear estimates will prove to be a useful tool:
	\begin{lemma}\label{lem:bilinear_estimates}
		For any $r>1/2$ there is a constant $C_r>0$ such that for all $f,g\in H^r$ we have
		\begin{align*}
			\|fg\|_{H^r}\leq C_r\|f\|_{H^r}\|g\|_r.
		\end{align*}
	\end{lemma}
	For further details and a proof of Lemma~\ref{lem:bilinear_estimates} see for instance \cite[Eqs.~(10)-(11)]{bronsard2022error}.} 
{The types of equations we consider are of the form \eqref{eqn:general_dispersive_nonlinear_eqn} with polynomial nonlinearities $p$ and linear operators $\mathcal{L},|\nabla|^\alpha$ which satisfy the following assumptions (cf. \cite{bruned_schratz_2022}).
\begin{assumption}\label{assump:fourier_form_of_relevant_operators} The linear operators $\mathcal{L}, |\nabla|^\alpha$ have the following representation in Fourier coordinates:
   \begin{align}\label{eqn:form_L_in_fourier}
   \mathscr{L}\left(\nabla\right)(k) =a_\lambda k^\lambda + \sum_{\gamma : |\gamma| < \lambda} a_{\gamma} \prod_{j} k_j^{\gamma_j} ,\quad \vert \nabla\vert^\alpha(k) =  \sum_{\gamma : |\gamma| < \lambda} b_{\gamma}\prod_{j=1}^{d} k_j^{\gamma_j},
\end{align}
for some $\lambda \in \mathbb{N}$, $ \lambda>\alpha$, $a_\gamma,b_\gamma\in\mathbb{R}$, and where we use the notation $ |\gamma| = \sum_i \gamma_i $ for $ \gamma \in \mathbb{Z}^d $ and
\begin{align*}
k^\lambda  = k_1^\lambda + \ldots + k_d^\lambda,\quad\text{for\ \ }k = (k_1,\ldots,k_d)\in \mathbb{Z}^d.
\end{align*}
\end{assumption}}
{\begin{remark}
    Note that while we focus on the form shown above, similarly to the ideas in \cite{bruned_schratz_2022} we can also treat the highly oscillatory case $\mathcal{L}=\mathcal{L}(\nabla,\dfrac{1}{\varepsilon})$ for some $\varepsilon\ll 1$, by considering operators of the form
$$
\mathscr{L}\left(\nabla, \frac{1}{\varepsilon}\right)=\frac{1}{\varepsilon^\lambda}+\mathscr{B}\left(\nabla, \frac{1}{\varepsilon}\right), \quad|\nabla|^\alpha=\mathscr{C}\left(\nabla, \frac{1}{\varepsilon}\right),
$$
for some differential operators $\mathscr{B}\left(\nabla, \frac{1}{\varepsilon}\right)$ and $\mathscr{C}\left(\nabla, \frac{1}{\varepsilon}\right)$ which can be bounded uniformly in $|\varepsilon|$ and are relatively bounded by differential operators of degree $\lambda$ and degree $\alpha<\lambda$, respectively. Following similar ideas to the ones described in this section would allow us to construct RK resonance-based schemes for Klein--Gordon-type equations.
\end{remark}}
{The following are two classical examples of dispersive nonlinear systems that fall in the class described by Assumption~\ref{assump:fourier_form_of_relevant_operators}.
\begin{example}
The periodic Korteweg--de Vries (KdV) equation \cite{drazin_johnson_1989} is given by
\begin{align}\label{eqn:standard_IVP_KdV}
	\begin{cases}
		\partial_t u(t,x)+\partial_{x}^3u(t,x)=\frac{1}{2}\partial_x\left(u(t,x)\right)^2,&\quad (t,x)\in [0,T]\times \mathbb{T},\\
		u(0,x)=u_0(x),&\quad x\in\mathbb{T},
	\end{cases}
\end{align}
and is of the form \eqref{eqn:general_dispersive_nonlinear_eqn} with $d=1,\mathcal{L}(\nabla)=i\partial_x^3,\alpha=1,\rho(u,\overline{u})=\frac{i}{2}u^2$.
\end{example}}
{\begin{example}\label{ex:NLSE}
The periodic nonlinear Schr\"odinger equation (NLSE) is given by
\begin{align}\label{eqn:Cauchy_problem_NLS}
	\begin{cases}
		i\partial_t u(t,x)=-\Delta^2 u(t,x)+\mu |u(t,x)|^{2p}u(t,x),&\quad (t,x)\in [0,T]\times \mathbb{T}^d,\\
		u(0,x)=u_0(x),&\quad x\in\mathbb{T}^d,		
	\end{cases}
\end{align}
where $p\in\mathbb{N}$, $\mu\in\mathbb{R}$. The NLSE is of the form  \eqref{eqn:general_dispersive_nonlinear_eqn} with $\mathcal{L}(\nabla)=\Delta,\alpha=0,\rho(u,\overline{u})=\mu(u\overline{u})^p$.
\end{example}}
\subsection{Introductory example: the KdV equation}\label{sec:RK_res_methods_for_KdV}
Having introduced these preliminaries let us know consider the construction of RK resonance-based methods. We begin our discussion with the KdV setting, because here we are able to resolve the nonlinear frequency interactions exactly {and the process of constructing RK resonance-based schemes simplifies because no kernel approximations are required. The extension of this construction to more general equations based on low-regularity kernel approximations is discussed in subsequent sections.} {To begin with we note that the mass $\int_{\mathbb{T}} u(\cdot,x)\dd x$ is conserved in \eqref{eqn:standard_IVP_KdV}, and that we may (by considering $u_0\mapsto u_0-\int_{\mathbb{T}}u_0\dd x$) therefore impose without loss of generality the following assumption to simplify subsequent constructions:
	\begin{assumption}\label{assumption:zero_mass}
		We assume throughout that our solution to the KdV equation has zero mass, i.e. that $\int_{\mathbb{T}}u_0(x)\dd x=0$.
	\end{assumption}
	In order to derive RK resonance-based methods it will be helpful to consider the twisted variable $v(t,x)=\exp(\partial_x^3 t)u(t,x).$ This change of variable is widely known to provide a useful tool both for the analysis of dispersive nonlinear equations \cite{bourgain1993,tao2006nonlinear} and the construction of tailored numerical schemes. The twisted variable $v$ satisfies the following initial value problem which is equivalent to \eqref{eqn:standard_IVP_KdV}:
	\begin{align}\begin{split}\label{eqn:interation_picture_kdv} 
			\begin{cases}
				\partial_tv(t,x)=\frac{1}{2}\e^{t\partial_x^3}\partial_x\left(\e^{-t\partial_x^3}v(t,x)\right)^2,&(t,x)\in\mathbb{R}_{+}\times \mathbb{T},\\
				v(0,x)=u_0(x),&x\in\mathbb{T}.
			\end{cases}
		\end{split}
	\end{align}	
	It is easy to see that under Assumption~\ref{assumption:zero_mass} the twisted variable will also satisfy $\int_{\mathbb{T}} v(t,x)\dd x=0$ for all times $t\geq 0$.} {We then consider Duhamel's formula for the twisted system \eqref{eqn:interation_picture_kdv}} where, for {notational simplicity}, we suppress the $x$-dependence of the unknown functions in the following notation
\begin{align}\label{eqn:Duhamel_for_v}
v(t_n+\tau)=v(t_n)+\frac{1}{2}\int_{0}^\tau \e^{(t_n+s)\partial_x^3}\partial_x\left(\e^{-(t_n+s)\partial_x^3}v(t_n+s)\right)^2\dd s.
\end{align}
{In the above we denoted by $\tau>0$ the time step and by $t_n=n\tau, n\in\mathbb{N},$ the time grid.} We can reformulate \eqref{eqn:Duhamel_for_v} in terms of the Fourier coefficients of $v$ as follows:
\begin{align}\nonumber
\hat{v}_m(t_n+\tau)&=\hat{v}_m(t_n)+\sum_{a+b=m} \frac{im}{2}\e^{-it_n(m^3-a^3-b^3)}\int_0^\tau \e^{-is(m^3-a^3-b^3)}\hat{v}_a(t_n+s)\hat{v}_b(t_n+s)\dd s\\\label{eqn:duhamel_in_fourier_for_v}
&=\hat{v}_m(t_n)+\sum_{a+b=m} \frac{im}{2}\e^{-it_n3mab}I^\tau_{a,b} {[v]},
\end{align}
where we defined the oscillatory integral
{\begin{align}\label{oscint}
    I^\tau_{a,b}[v]:=\int_0^\tau \e^{-is3mab}\hat{v}_a(t_n+s)\hat{v}_b(t_n+s)\dd s,
\end{align}}
and used the algebraic relation $(a+b)^3-a^3-b^3=3(a+b)ab$. The central observation is that the nonlinear frequency interactions in the KdV system are now captured by the oscillatory terms 
\begin{align}\label{cosc}
\exp(-is3mab).
\end{align} 

Our key idea in the novel construction lies now in embedding these nonlinear frequency interactions exactly into our numerical discretisation (in the spirit of resonance-based schemes \cite{Bronsardbruned2022,bruned_schratz_2022,hofmanova2017exponential}), {while approximating the non-oscillatory parts in a more general way than prior work on resonance-based schemes.} In the discretisation of our oscillatory integral \eqref{oscint} this idea translates to treating these central oscillations \eqref{cosc} exactly and to henceforth only approximate numerically the corresponding non-oscillatory parts
\begin{align}\label{eq:svpa}
\hat{v}_a(t_n+s)\quad \text{and}\quad \hat{v}_b(t_n+s)
\end{align}
in \eqref{oscint}. Note that $\hat{v}_\sigma(t_n+s)$ ($\sigma =a,b$) are indeed slowly varying as, thanks to 
 \eqref{eqn:interation_picture_kdv}, we have for any $s>d/2$
\begin{align}\label{tayK}
\left\|\partial_t v\right\|_{s}\leq c\left\|\partial_x v\right\|_{s}\left\|v\right\|_{s}
\end{align}
for some constant $c>0$ independent of $v,u_0$. Here we relied on the bilinear estimates from Lemma~\ref{lem:bilinear_estimates} and the fact that $v\mapsto \exp(\pm t\partial_x^3)v$ is an isometry on $H^s$, for all $s,t\geq0$. 

In {the} prior work \cite{hofmanova2017exponential} the central idea in the discretisation of the oscillatory integral \eqref{oscint} lies in a simple Taylor series expansion of the non-oscillatory parts \eqref{eq:svpa} in the spirit of
\begin{align}\label{eq:svK}
\hat{v}_\sigma(t_n+s)\approx \hat{v}_\sigma(t_n), \quad \text{for any }\sigma\in\mathbb{Z}.
\end{align}
Together with the observation in \eqref{tayK} this leads to a local error structure at low regularity of  the form
\begin{align}\label{lowE}
 \mathcal{O}\left( s \partial_t v \right) =    \mathcal{O}\left( s \partial_x v^2 \right).
\end{align}
We call the above error  of \emph{low regularity} as a classical direct approximation of the KdV equation (for example a first order exponential integrator) would introduce a local error  at order 
\begin{align*}
 \mathcal{O}\left( s \partial_t u  \right) =    \mathcal{O}\left( s \partial_x^3 u\right)
\end{align*}
which involves higher {spatial} derivatives (and thus higher regularity assumptions on the solution) than \eqref{lowE}.

Most resonance-based schemes proposed in the literature so far follow exactly this construction \cite{bruned_schratz_2022,hofmanova2017exponential,ostermann2018low}. Due to the favourable local error structure of this approach, in general one obtains better approximations at low regularity than classical numerical schemes (e.g., splitting, exponential integrator or Lawson-type methods). However, a major drawback lies in the fact that the quite brutal approximation \eqref{eq:svK} destroys the symplectic structure of the KdV flow
\begin{align}\nonumber
v(0) \mapsto \phi_{0,t}(v(0)) = v(0)+\frac{1}{2}\int_0^t \e^{s\partial_x^3}\partial_x\left(\e^{-s\partial_x^3}\phi_{0,s}(v(0))\right)^2\dd s.
\end{align}

In order to overcome this, our new idea lies in the fact that \textit{qualitatively} the size of the local error in the numerical scheme would remain the same if we used for $\sigma =a,b$ the implicit approximation
\begin{align*}
\hat{v}_\sigma(t_n+s)\approx \hat{v}_\sigma(t_n+\tau) 
\end{align*}
or indeed if we took, more generally, a polynomial-type interpolant for the term involving the unknown: 
\begin{align}\label{eqn:interpolation_unknown}
	\hat{v}_\sigma(t_n+s)\approx P(s):=\sum_{p=0}^{S}\frac{s^p}{\tau^p} \sum_{q=0}^Sa_{p,q}\hat{v}_\sigma(t_n+c_q\tau)
\end{align}
for some $S\in\mathbb{N}, 0\leq c_0<c_1<\cdots<c_{S}\leq 1, a_{p,q}\in\mathbb{C}, p=0\dots S$. {Note the factors $\frac{s^p}{\tau^p}$ can be justified by mapping the interpolation problem to a unit reference interval $[0,1]$ instead of $[0,\tau]$.} The central observation is that the resulting integrals can still be given an exact representation in physical space, thus allowing for fast FFT-based computations: Indeed let us define the maps
\begin{align}\label{eqn:central_integrals_in_KdV_RK_construction}
   v\mapsto \mathcal{F}^{[KdV]}_p(\tau; c_q;v):=\frac{1}{2}\sum_{m\in\mathbb{Z}}e^{ix m}\frac{1}{\tau^{p+1}}\sum_{a+b=m}\int_{0}^{c_q\tau} ime^{-is3mab} s^{p}\dd s\hat{v}_a\hat{v}_b.
\end{align}
This can be expressed as follows:
\begin{align*}
    \mathcal{F}^{[KdV]}_p(\tau; c_q;v)=\frac{1}{2}\sum_{m\in\mathbb{Z}}e^{ix m}\sum_{a+b=m}c_q^{p+1} im\varphi_{p+1}(-ic_q\tau 3mab)\hat{v}_a\hat{v}_b,
\end{align*}
where we have the recurrence:
\begin{align*}
    \varphi_{p+1}(z)=\frac{e^{z}-p\varphi_{p}(z)}{z},\quad \varphi_1(z)={\frac{e^{z}-1}{z},}
\end{align*}
which allows us to express the functions in a simple way in physical space, for example (using assumption~\ref{assumption:zero_mass}): \begin{align}\begin{split}\label{eqn:physical_coords_KdV_integrals}
\mathcal{F}^{[KdV]}_0(\tau;c_q;v)&=\frac{1}{6\tau}e^{\tau c_q\partial_x^3}\left(e^{-\tau c_q\partial_x^3}\partial_x^{-1}v\right)^2-\frac{1}{6\tau}\left(\partial_x^{-1}v\right)^2,\\
\mathcal{F}^{[KdV]}_1(\tau;c_q;v)&=\frac{1}{18\tau^2}\partial_x^{-1}\left(\partial_x^{-2}v\right)^2-\frac{1}{18\tau^2}\partial_x^{-1}e^{\tau c_q\partial_x^3}\left(e^{-\tau c_q\partial_x^3}\partial_x^{-2}v\right)^2+\frac{c_q}{6\tau}e^{\tau c_q\partial_x^3}\left(e^{-\tau c_q\partial_x^3}\partial_x^{-1}v\right)^2{,}
\end{split}
\end{align}
where we defined
\begin{align}\label{eqn:definition_dx_inv}
    \widehat{\left(\partial_x^{-1}v\right)}_{k}:=\begin{cases}0,& k=0,\\
        \frac{1}{ik}\hat{v}_k,&k\neq0.
    \end{cases}
\end{align}
Following the above remarks we define RK resonance-based schemes, motivated by Duhamel's formula \eqref{eqn:duhamel_in_fourier_for_v}, as follows:
\begin{align}\begin{split}\label{eqn:RK_res_based_schemes_KdV}
	u^{n+1}&=e^{-\tau\partial_x^3}u^{n}+\tau\sum_{p,q,r=0}^{S} b^{p,q,r}e^{-\tau\partial_x^3}K_{p,q,r},\\
	K_{p,q,r}&=\mathcal{F}^{[KdV]}_p(\tau; c_q; u^{n}+\tau \sum_{\tilde{p},\tilde{q},\tilde{r}=0}^{S} a_{p,q,r}^{\tilde{p},\tilde{q},\tilde{r}}K_{\tilde{p},\tilde{q},\tilde{r}}),\end{split}
\end{align}
for some constants $c_q\in[0,1], a_{p,q,r}^{\tilde{p},\tilde{q},\tilde{r}},b^{p,q,r}\in\mathbb{R}, 0\leq p,q,r,\tilde{p},\tilde{q},\tilde{r}\leq S$ and a given $S\in\mathbb{N}$. Note that the form \eqref{eqn:RK_res_based_schemes_KdV} is similar to classical Runge--Kutta methods but incorporates additional degrees of freedom to allow for a `highly-oscillatory quadrature' resolution of the integrals in Duhamel's formula. A comparable formulation is known to describe exponential Runge--Kutta methods as introduced by \cite{hochbruck2005exponential} (see also \cite[Section~2.3]{hochbruck_ostermann_2010}).
\begin{remark}
In similar vein to the proof of stability in {Section~{\ref{sec:convergence_analysis_KdV}}} it can be shown that the methods found in the class \eqref{eqn:RK_res_based_schemes_KdV} are in fact all unconditionally stable for solutions in $H^s,s>2$, which is in itself a significant advantage over other existent methods for the KdV equation.
\end{remark}
\begin{example}
Perhaps one of the simplest examples in this category of RK resonance-based integrators is the explicit first order method introduced by \cite{hofmanova2017exponential}, which takes the form
\begin{align*}
    u^{n+1}=e^{-\tau\partial_x^3}u^n+\frac{1}{6}\left(e^{-\tau\partial_x^3}\partial_x^{-1}u^n\right)^2-\frac{1}{6}e^{-\tau\partial_x^3}\left(\partial_x^{-1}u^n\right)^2   .
\end{align*}
This method can be found from the above expression \eqref{eqn:RK_res_based_schemes_KdV} by taking $S=0$ and the coefficients $c_0=b^{0,0,0}=1,a_{0,0,0}^{0,0,0}=0$.
\end{example}
The advantage of \eqref{eqn:RK_res_based_schemes_KdV} over classical explicit resonance-based methods (\cite{hofmanova2017exponential,bruned_schratz_2022}) is that this novel formulation incorporates many more degrees of freedom allowing for schemes with structure preserving properties (cf. {Section}~{\ref{sec:symplectic_RK_res_based_schemes}}) at the same time as good low-regularity convergence properties (cf. Section~{\ref{sec:convergence_analysis}}). 
\subsection{{Construction for general dispersive nonlinear equations}}\label{sec:RK_res_schemes_for_general_eqns}
{Let us now describe an extension of this construction to the general class of dispersive nonlinear equations described in \eqref{eqn:general_dispersive_nonlinear_eqn} and Assumption~\ref{assump:fourier_form_of_relevant_operators}. The construction will consist of the following three steps:
\begin{enumerate}[(I)]
    \item Change to the twisted variable in \eqref{eqn:general_dispersive_nonlinear_eqn};
    \item Determine the resonances in Duhamel's formula for the twisted variable and approximate the corresponding convolution kernel with low-regularity error;
    \item Define piecewise polynomial approximants of the unknowns in the interaction picture and collect terms to arrive at the RK resonance-based scheme.
\end{enumerate}
Let us now provide more details on each individual step of the aforementioned procedure.
\paragraph{(I) Twisted variable:} Analogously  to Section \ref{sec:RK_res_methods_for_KdV} we consider the twisted variable $$v(t,x)=\exp(-it\mathcal{L}(\nabla))u(t,x),$$ which satisfies
\begin{align}\label{eqn:twisted_formulation_of_general_eqn}
\begin{cases}
	i \partial_t v(t,x) =e^{-it\mathcal{L}(\nabla)}\vert \nabla\vert^\alpha \rho\left(e^{it\mathcal{L}(\nabla)}v(t,x), e^{-it\mathcal{L}(\nabla)}\overline v(t,x)\right), &\quad (t,x)\in [0,T]\times \mathbb{T}^d,\\
	v(0,x) = u_0(x),&\quad x\in\mathbb{T}^d.
\end{cases} 
\end{align}
The solution for \eqref{eqn:twisted_formulation_of_general_eqn} can then be written in the mild form
\begin{align}\label{eqn:Duhamel_forumla_general_dispersive}
	v(t_n+\tau)=v(t_n)-i\int_{0}^{\tau}e^{-i(t_n+s)\mathcal{L}(\nabla)}\vert \nabla\vert^\alpha \rho\!\left(e^{i(t_n+s)\mathcal{L}(\nabla)}v(t_n+s,x), e^{-i(t_n+s)\mathcal{L}(\nabla)}\overline v(t_n+s,x)\right)\dd s.
\end{align}
\paragraph{(II) Low-regularity kernel approximation:} In order to arrive at practical schemes (i.e. methods which can be efficiently implemented) while retaining good low-regularity approximation properties we switch to Fourier coordinates where \eqref{eqn:Duhamel_forumla_general_dispersive} becomes (for all $m\in\mathbb{Z}^d$ and $\tau>0$)
\begin{align}\label{eqn:Duhamel_forumla_general_dispersive_Fourier}
\hat{v}_{m}(t_n+\tau)=\hat{v}_m(t_n)-i|\nabla|^\alpha(m) \int_{0}^{\tau}e^{-i(t_n+s)\mathcal{L}(\nabla)(m)}\rho_m\left(e^{it\mathcal{L}(\nabla)}v(t_n+s,x), e^{-it\mathcal{L}(\nabla)}\overline v(t_n+s,x)\right)\dd s,
\end{align}
where $\mathcal{L}(\nabla)(m), \vert \nabla\vert^\alpha(m)$ are as in Assumption~\ref{assump:fourier_form_of_relevant_operators} and the polynomial $\rho$ takes the following form in Fourier coordinates 
\begin{align}\nonumber
    \rho(u,\bar{u})&=\sum_{m\in\mathbb{Z}^d}e^{im\cdot x}\rho_m(u,\bar{u}),\\\label{eqn:form_of_p_Fourier}
    \rho_m(u,\bar{u})&=\sum_{N+M\leq \deg p}a_{M,N}\sum_{m=\sum_ik_i-\sum_j\overline{k}_j} \prod_{i=1}^N \hat{v}_{k_i}(t_n+s)\prod_{j=1}^M\overline{\hat{u}}_{\bar{k}_j}(t).
\end{align}
\begin{example}[Cubic nonlinear Schr\"odinger equation]\label{ex:fourier_duhamel_cubic_NLSE}
	For the defocusing cubic nonlinear Schr\"odinger equation (cf. Example~\ref{ex:NLSE} with $\mu=1,p=1$) the formula \eqref{eqn:Duhamel_forumla_general_dispersive_Fourier} simplifies to the following expression
	\begin{align*}
		\hat{v}_m(t_n+\tau)=\hat{v}_m(t_n)-i\sum_{m+\bar{k}_1={k}_1+{k}_2}e^{it_n(m^2+\bar{k}_1^2-{k}_1^2-{k}_2^2)}\int_0^{\tau}e^{is(m^2+\bar{k}_1^2-{k}_1^2-{k}_2^2)}\overline{\hat{v}_{\bar{k}_1}^n} \hat{v}_{k_1}^n\hat{v}_{k_2}^n\dd s, \quad m\in\mathbb{Z}^d.
	\end{align*}
\end{example}
Using \eqref{eqn:form_of_p_Fourier}, we can then rewrite \eqref{eqn:Duhamel_forumla_general_dispersive_Fourier} in the following form:
\begin{align}\begin{split}\label{eqn:Duhamel_forumla_general_dispersive_Fourier_collected_coeffs}
    \hat{v}_{m}(t_n+\tau)&=\hat{v}_m(t_n)\\
    &\quad-i|\nabla|^\alpha(m)\!\!\sum_{N+M\leq \deg p}\!\!a_{M,N}\\&\quad\ \ \sum_{m=\sum_ik_i-\sum_j\overline{k}_j} \int_{0}^{\tau}e^{-i(t_n+s)\mathcal{P}^{(M,N)}(m,k_1,\dots, k_N,\bar{k}_1,\dots,\bar{k}_M)}\prod_{i=1}^N\hat{v}_{k_i}(t_n+s)\prod_{j=1}^M
    \overline{\hat{v}_{\bar{k}_j}(t_n+s)}\dd s.
    \end{split}
\end{align}
The nonlinear frequency interactions of \eqref{eqn:general_dispersive_nonlinear_eqn} are then given by the following kernels
\begin{align*}
    \mathcal{K}^{(M,N)}(t_n+s;m,k_1,\dots, k_N,\bar{k}_1,\dots,\bar{k}_M)=e^{-i(t_n+s)\mathcal{P}^{(M,N)}(m,k_1,\dots, k_N,\bar{k}_1,\dots,\bar{k}_M)}.
\end{align*}
The central step in the construction of low-regularity integrators (in prior work and in the current setting) is to then find an approximation of $\mathcal{K}^{(M,N)}$ which leads to a low-regularity error, while still allowing for efficient implementation of the nonlinearity in a time-stepping scheme. Based on the order $\lambda$ of the linear differential operator $\mathcal{L}$ (cf. Assumption~\ref{assump:fourier_form_of_relevant_operators}) this means that effectively an approximation $\mathcal{K}^{(M,N)}_{\sigma}\approx \mathcal{K}^{(M,N)}$ of order $\sigma$ is sought such that a uniform approximation of the following form holds
\begin{align}\begin{split}\label{eqn:kernel_approximation_general_dispersive_eqn}
   |\mathcal{K}^{(M,N)}(t_n+s;m,k_1,\dots, k_N,\bar{k}_1,\dots,&\bar{k}_M)-\mathcal{K}_\sigma^{(M,N)}(t_n+s;m,k_1,\dots, k_N,\bar{k}_1,\dots,\bar{k}_M)|\\&\leq Cs^\sigma \sum_{\alpha_j\leq \lambda_\sigma}|k_j|^{\alpha_j}, \quad \forall m, k_i,\bar{k}_j\in\mathbb{Z}^d,\text{\ with\ }m=\sum_{i=1}^N k_i-\sum_{j=1}^M\overline{k}_j,
   \end{split}
\end{align}
where $C$ is independent of $s$ and the Fourier indices, and the integer $\lambda_\sigma$ is such that $\lambda_\sigma<\sigma\lambda$. A further central aspect of this approximation is that the method should lead to efficient computations. In the setting of periodic boundary conditions, and hence pseudospectral methods, it suffices to choose $\mathcal{K}_\sigma$ such that the maps
\begin{align*}
   v\mapsto\sum_{m\in\mathbb{Z}^d}e^{im\cdot x}\!\!\!\!\!\! \sum_{m=\sum_ik_i-\sum_j\overline{k}_j}\!\!\!\!\!\!\!\!\tau^{-(p+1)} \!\int_{0}^{\tau}s^p\mathcal{K}_\sigma^{(M,N)}(t_n+s;m,k_1,\dots, k_N,\bar{k}_1,\dots,\bar{k}_M)\dd s\prod_{i=1}^N\hat{v}_{k_i}\prod_{j=1}^M\overline{\hat{v}_{\bar{k}_j}},
\end{align*}
for $v(x)=\!\sum_{k\in\mathbb{Z}^d}\!e^{ik\cdot x}\hat{v}_k$, have a simple representation in physical space (for further details on this point the reader is referred to \cite{ostermann2018low}. This property is achieved in the KdV equation by taking $\mathcal{K}_{\sigma}=\mathcal{K}$, however requires much more subtle considerations in other cases.
\begin{example}[Low-regularity kernel approximation for cubic NLSE]\label{ex:kernel_approx_nlse_ostermann_schratz} As shown in Example~\ref{ex:fourier_duhamel_cubic_NLSE} for the cubic NLSE we have
\begin{align}\label{eqn:cubic_NLSE_kernel}
    \mathcal{K}(t;k,\bar{k}_1,{k}_1,{k}_2)=e^{it(k^2+\bar{k}_1^2-{k}_1^2-{k}_2^2)}.
\end{align}
A low-regularity kernel approximation introduced in \cite{ostermann2018low} is
\begin{align*}
    \mathcal{K}_1(t;k,\bar{k}_1,{k}_1,{k}_2)=e^{it2\bar{k}_1^2},
\end{align*}
for which we have
\begin{align*}
  |\mathcal{K}(s;k,\bar{k}_1,{k}_1,{k}_2)-\mathcal{K}_1(s;k,\bar{k}_1,{k}_1,{k}_2)|&=2s\left(|\bar{k}_1\cdot {k}_1+\bar{k}_1\cdot {k}_2-{k}_1\cdot {k}_2|\right)|\varphi(2is(\bar{k}_1\cdot {k}_2+\bar{k}_1\cdot {k}_1-{k}_1\cdot {k}_2))|\\
  &\leq 2s\left(|\bar{k}_1||{k}_1|+|{k}_1||{k}_2|+|\bar{k}_1||{k}_2|\right),\\
  &\quad\quad\quad\quad\quad\quad\quad\quad \forall s>0, k,\bar{k}_1,{k}_1,{k}_2\in\mathbb{Z}^d, k={k}_1+{k}_2-\bar{k}_1.
\end{align*}
and the map
\begin{align*}
    v\mapsto \sum_{m\in\mathbb{Z}^d}e^{im\cdot x}\!\!\!\!\!\!\sum_{m=\sum_ik_i-\sum_j\overline{k}_j}\int_{0}^{\tau}\mathcal{K}_1(t_n+s;m,k_1,\bar{k}_1,\bar{k}_2)\dd s\hat{v}_{k_1}\overline{\hat{v}_{\bar{k}_1}}\overline{\hat{v}_{\bar{k}_2}},
\end{align*}
can be represented in physical space by
\begin{align*}
v\mapsto \tau v^2\varphi_1(-2i\tau\Delta)\overline{v},
\end{align*}
thus leading to fast computation in spectral methods (the action of the differential operators can be computed diagonally in Fourier coordinates and the polynomial nonlinearities can be computed as pointwise operations in physical space).
\end{example}
A broad and successful research effort over the recent decade has yielded a broad amount of works that facilitate precisely this construction of low-regularity approximations $\mathcal{K}^{(M,N)}_\sigma$ in a range of settings, in particular the following:
	\begin{itemize}
		\item \textbf{NLSE:} Ostermann \& Schratz \cite{ostermann2018low}, and Feng et al. \cite{fengmaierhoferschratz23};
        \item \textbf{Gross--Pitaevskii equation:} Alama Bronsard \cite{bronsard2022error};
		\item \textbf{KdV equation:} Hofmanov\'a \& Schratz \cite{hofmanova2017exponential}, and Li \& Wu \cite{li2022unfiltered};
		\item \textbf{Klein--Gordon equation:} Wang \& Zhao \cite{wang2022symmetric}, and Cabrera Calvo \cite{calvo2023uniformly};
		\item \textbf{Navier--Stokes equations:} Li et al. \cite{limaschratz2022};
  		\item \textbf{General dispersive nonlinear systems:} Bruned \& Schratz \cite{bruned_schratz_2022}.
\end{itemize}
We can use those kernel approximations as the starting point for our constructions of RK resonance-based schemes. In particular, using \eqref{eqn:Duhamel_forumla_general_dispersive_Fourier_collected_coeffs} we consider the approximation:
\begin{align}\begin{split}\label{eqn:Duhamel_forumla_general_dispersive_Fourier_collected_coeffs_kernel_approx}
    \hat{v}_{m}(t_n+\tau)&=\hat{v}_m(t_n)\\
    &\quad-i|\nabla|^\alpha(m)\!\!\sum_{N+M\leq \deg p}\!\!a_{M,N}\\&\ \sum_{m=\sum_ik_i-\sum_j\overline{k}_j} \int_{0}^{\tau}\mathcal{K}_\sigma^{(M,N)}(t_n+s;m,k_1,\dots, k_N,\bar{k}_1,\dots,\bar{k}_M)\prod_{i=1}^N\hat{v}_{k_i}(t_n+s)\prod_{j=1}^M
    \overline{\hat{v}_{\bar{k}_j}(t_n+s)}\dd s.
    \end{split}
\end{align}
\paragraph{(III) Runge--Kutta resonance-based schemes:} We continue to follow the steps outlined in Section~\ref{sec:RK_res_methods_for_KdV} and consider polynomial-type interpolants for the unknown terms in \eqref{eqn:Duhamel_forumla_general_dispersive_Fourier_collected_coeffs_kernel_approx}. We thus define the following maps
\begin{align}\begin{split}\label{eqn:nonlinear_maps_in_general_RK_schemes}
    v\mapsto \mathcal{F}_{p}^{[\mathcal{L},\rho]}(\tau;c_q;v)&:=-i|\nabla|^\alpha(m)\sum_{N+M\leq \deg p}a_{M,N}\sum_{m\in\mathbb{Z}^d}e^{im \cdot x}\\
    &\  \sum_{m=\sum_ik_i-\sum_j\overline{k}_j}\!\!\!\!\!\!\!\!\tau^{-(p+1)} \!\int_{0}^{\tau}s^p\mathcal{K}_\sigma^{(M,N)}(t_n+s;m,k_1,\dots, k_N,\bar{k}_1,\dots,\bar{k}_M)\dd s\prod_{i=1}^N\hat{v}_{k_i}\prod_{j=1}^M\overline{\hat{v}_{\bar{k}_j}}.
    \end{split}
\end{align}
Following the above remarks and the discussion for the KdV equation we twist back and define RK resonance-based schemes, motivated by Duhamel's formula \eqref{eqn:Duhamel_forumla_general_dispersive}, as follows:
\begin{align}\begin{split}\label{eqn:RK_res_based_schemes_general}
	u^{n+1}&=e^{i\tau \mathcal{L}(\nabla)}u^{n}+\tau\sum_{p,q,r=0}^{S} b^{p,q,r}e^{i\tau\mathcal{L}(\nabla)}K_{p,q,r},\\
	K_{p,q,r}&=\mathcal{F}^{[\mathcal{L},\rho]}_p(\tau; c_q; u^{n}+\tau \sum_{\tilde{p},\tilde{q},\tilde{r}=0}^{S} a_{p,q,r}^{\tilde{p},\tilde{q},\tilde{r}}K_{\tilde{p},\tilde{q},\tilde{r}}).\end{split}
\end{align}
\begin{example}[First order method {introduced in \cite{ostermann2018low}}] 
    Taking the kernel approximation from Example~\ref{ex:kernel_approx_nlse_ostermann_schratz}, i.e. $\mathcal{K}_1(t;k,\bar{k}_1,{k}_1,{k}_2)=e^{it2\bar{k}_1^2}$, we can recover the first order low-regularity integrator for the NLSE introduced in \cite{ostermann2018low} by taking $S=0, c_0=b^{0,0,0}=1,a_{0,0,0}^{0,0,0}=0$, leading to the scheme
    \begin{align*}
        	u^{n+1}&=e^{i\tau{\Delta} }u^{n}+\tau e^{i\tau{\Delta} }\mathcal{F}^{[NLSE]}_0(\tau; 1; u^{n}), 
    \end{align*}
    which corresponds to
    \begin{align*}
        u^{n+1}&=e^{i\tau{\Delta} }u^{n}-i\tau e^{i\tau{\Delta} }\left((u^n)^2\varphi_1(-2i\tau\Delta)\overline{u^n}\right).
    \end{align*}
\end{example}
\begin{example}[Symmetric method from {\cite{AlamaBronsard23,banicamaierhoferschratz22}}] We now consider the piecewise kernel approximation for the cubic NLSE introduced in \cite[(1.8)]{AlamaBronsard23}:
\begin{align*}
    \mathcal{K}(t;k,\bar{k}_1,{k}_1,{k}_2)=e^{it(k^2+\bar{k}_1^2-{k}_1^2-{k}_2^2)}\approx\begin{cases}
        e^{2it\bar{k}_1^2},& s\in[0,\tau/2],\\
        e^{2it\bar{k}_1^2}e^{i\tau(k^2-\bar{k}_1^2-{k}_1^2-{k}_2^2)},& s\in(\tau/2,\tau].
    \end{cases}
\end{align*}
Then, if we take $S=1,c_0=\dfrac{1}{2},c_1=1,b^{0,0,0}=a_{0,0,1}^{0,0,0}=a_{0,1,1}^{0,0,0}=1,b^{0,0,1}=a_{0,0,1}^{0,0,1}=a_{0,0,1}^{0,1,1}=-1,b^{0,1,1}=a_{0,0,1}^{0,1,1}=a_{0,1,1}^{0,1,1}=2$ and all other coefficients $a_{p,q,r}^{\tilde{p},\tilde{q},\tilde{r}},b^{p,q,r}$ to be equal to zero, we obtain the following method
\begin{align*}
    u^{n+1}&=e^{i\tau\Delta}u^n+\tau e^{i\tau\Delta}K_{0,0,0}+2\tau e^{i\tau\Delta}K_{0,1,1}-\tau e^{i\tau\Delta}K_{0,0,1},\\
    K_{0,0,0}&=\mathcal{F}_{0}^{[NLSE]}(\tau;\dfrac{1}{2};u^n),\\
    K_{0,0,1}&=\mathcal{F}_{0}^{[NLSE]}(\tau;\dfrac{1}{2};u^n+\tau K_{0,0,0}+2\tau K_{0,1,1}-\tau K_{0,0,1}),\\
    K_{0,1,1}&=\mathcal{F}_{0}^{[NLSE]}(\tau;1;u^n+\tau K_{0,0,0}+2\tau K_{0,1,1}-\tau K_{0,0,1}).
\end{align*}
which is equivalent to
\begin{align*}
   u^{n+1}&=e^{i\tau\Delta}u^n+\tau e^{i\tau\Delta}\left(\mathcal{F}_{0}^{[NLSE]}\left(\tau;\dfrac{1}{2};u^n\right)+2\mathcal{F}_{0}^{[NLSE]}\left(\tau;1;e^{-i\tau\Delta}u^{n+1}\right)\right.\\&\quad\quad\quad\quad\quad\quad\quad\quad\quad\quad\quad\quad\quad\quad\quad\quad\quad\quad\quad\quad\quad\quad\quad\quad\left.-\mathcal{F}_{0}^{[NLSE]}\left(\tau;\dfrac{1}{2};e^{-i\tau\Delta}u^{n+1}\right)\right)\\
   &=e^{i\tau\Delta}u^n-i\frac{\tau}{2} e^{i\tau\Delta}\left((u^n)^2\varphi_1(-i\tau\Delta)\overline{u^n}\right)-i\frac{\tau}{2} \left((u^{n+1})^2\varphi_1(i\tau\Delta)\overline{u^{n+1}}\right).
\end{align*}
This is precisely the method described in \cite{AlamaBronsard23,banicamaierhoferschratz22}.
\end{example}
Further examples of RK resonance-based methods are provided in Section~\ref{sec:examples_of_symplectic_integrators}.}
\section{Symplectic Runge--Kutta resonance-based schemes}\label{sec:symplectic_RK_res_based_schemes}
{Following the above general construction of RK resonance-based schemes we now focus on two Hamiltonian dispersive PDEs, the KdV equation and NLSE and use the aformentioned approach to design symplectic low-regularity integrators. Indeed both are separable infinite-dimensional Hamiltonian systems that can be written in the abstract Hamiltonian form}
\begin{align*}
\frac{\dd p_a}{\dd t}=-\frac{\partial}{\partial q_a}\mathcal{H}(\mathbf{p},\mathbf{q}),\quad \frac{\dd q_a}{\dd t}=\frac{\partial}{\partial p_a}\mathcal{H}(\mathbf{p},\mathbf{q}), \quad a\in \mathcal{I},
\end{align*}
where the index set $\mathcal{I}$ is $\mathbb{N}$ (for the KdV case) and $\mathbb{Z}^d$ (for the NLSE case). The coordinates and Hamiltonians as well as their quadratic first integrals are recalled in broad terms below, but for further details the reader is referred to standard literature such as \cite{marsden2002introduction}. {Of course, such systems automatically preserve the energy, meaning in particular that $\mathcal{H}(\mathbf{p}(t),\mathbf{q}(t)),t\geq 0,$ is a constant of motion for both equations.}
In order to understand the convergence properties of sums appearing in these discrete formulations it will be convenient to introduce the following Hilbert spaces:
\begin{definition}
For $s\in\mathbb{R}_{+}$, we define for a sequence $\mathbf{v}=\left(v_m\right)_{m\in\mathcal I}$ the norm
\begin{align*}
	\|\mathbf{v}\|_{l^{2}_s}:=\left(\sum_{m\in\mathcal{I}}\langle m\rangle^{2s}|v_m|^2\right)^2
\end{align*}
and we define the space $l_s^2$ by
\begin{align*}
	l_s^2:=\left\{\mathbf{v}\in l^2(\mathbb{Z}\setminus\{0\})\,\vert\, \|\mathbf{v}\|_{l^2_s}<\infty\right\}.
\end{align*}
\end{definition}
The above spaces are isometric to the classical Sobolev spaces $H^s$, $s>0$.
\subsection{Overview of the geometric structure of the KdV equation}
The KdV equation is completely integrable and has an infinite set of first integrals \cite[Section~3.1]{dunajski2010solitons}. In the present {section} {we will mainly focus on the conservation of} its quadratic first integral, the momentum
\begin{align}\label{eqn:quadratic_first_integral_kdv}
I_0^{[KdV]}[u]=\int_{\mathbb{T}}u^2\dd x.
\end{align}
In order to study structure preservation properties of our methods it will be convenient to look at the following infinite-dimensional Hamiltonian formulation of the KdV equation \cite[{Section~2}]{guan2014kdv}. Let us {define} $$p_a(t):=\hat{u}_a(at),\quad q_a(t):=\hat{u}_{-a}(at) \quad\text{ for }\quad a\geq 1.$$  Then the KdV equation \eqref{eqn:standard_IVP_KdV} is equivalent to the following infinite-dimensional Hamiltonian system
\begin{align*}
\frac{\dd p_a}{\dd t}=-\frac{\partial}{\partial q_a}\mathcal{H}_1(\mathbf{p},\mathbf{q}),\quad \frac{\dd q_a}{\dd t}=\frac{\partial}{\partial p_a}\mathcal{H}_1(\mathbf{p},\mathbf{q}),
\end{align*}
where the \textit{time-independent} Hamiltonian $\mathcal{H}_1$ is given by
\begin{align*}
\mathcal{H}_1(\mathbf{p},\mathbf{q})=-i\sum_{{c\in\mathbb{N}}}c^2p_cq_c-\frac{i}{2}\sum_{\substack{{a,b,c\in\mathbb{N}}\\a+b-c=0}}(p_ap_bq_c+q_aq_bp_c).
\end{align*}
The corresponding symplectic form is given by
\begin{align}\label{eqn:symplectic_form_KdV}
	\omega=\sum_{a=1}^\infty \dd p_a\wedge \dd q_a.
\end{align}

With the aid of the Cauchy-Schwarz inequality one can in particular show that whenever $(\mathbf{p},\mathbf{q})\in l^{2}_s,$ i.e. whenever $u\in H^{s}, s>3/2$, all of the above sums converge absolutely and the Hamiltonian and the symplectic form as given above are well-defined.

For further details on the Hamiltonian structure of the KdV equation and the local well-posedness of the KdV equation in low-regularity regimes we refer the reader to \cite{bourgain1993,gubinelli2012rough,iandoli2022,tao2006nonlinear}.

\subsection{Overview of the geometric structure of the nonlinear Schr\"odinger equation}\label{sec:structure_of_NLSE}
	The polynomial NLSE \eqref{eqn:Cauchy_problem_NLS} on $\mathbb{T}^d$ is a separable infinite-dimensional Hamiltonian system with (complex) Hamiltonian coordinates (cf. \cite[Section~III.1]{faou2012geometric} and \cite[Section~3.2]{marsden2002introduction}) $\xi_a=\hat{u}_a,\eta_a=\overline{\hat{u}_a}$ and (imaginary-valued) Hamiltonian
	\begin{align*}
		\mathcal{H}(\xi,\eta)=i\sum_{a\in\mathbb{Z}^d}|a|^2\xi_a\eta_a+i\frac{\mu}{2(p+1)}\sum_{\substack{a_i,b_i\in\mathbb{Z}^d\\a_1+\dots+a_{p+1}=b_1+\dots+b_{p+1}}}\xi_{a_1}\cdots\xi_{a_{p+1}}\eta_{b_1}\cdots\eta_{b_{p+1}},
	\end{align*}
such that, formally speaking,
\begin{align}\label{eqn:Hamiltonian_structure_NLS}
	\frac{\dd\xi_a}{\dd t}=-\frac{\partial\mathcal{H}}{\partial \eta_a}(\xi,\eta),\quad	\frac{\dd\eta_a}{\dd t}=\frac{\partial\mathcal{H}}{\partial \xi_a}(\xi,\eta),\quad\forall a\in\mathbb{Z}^d.
\end{align}
In addition, it possesses the following quadratic first integral
\begin{align}\label{eqn:quadratic_first_integral_nlse}
I_0^{[NLSE]}[u]=\int_{\mathbb{T}}|u|^2\dd x.
\end{align}
The corresponding symplectic form is given by \begin{align}\label{eqn:symplectic_form_NLSE}
\omega=\sum_{a\in\mathbb{Z}}\dd\xi_a\wedge\dd\eta_a.
\end{align}

As mentioned in Section~\ref{sec:construction_of_rk_res_schemes}, a central step in the construction of our symplectic low-regularity schemes will be to consider the twisted variable $v=\exp(-it{\partial_x^2})u$ which satisfies
\begin{align}\label{eqn:twisted_NLS}
	i\partial_t v(t)=\mu e^{-it{\partial_x^2}}\left[\left|e^{it{\partial_x^2}}v(t)\right|^{2p}e^{it{\partial_x^2} }v(t)\right].
\end{align}
The twisted equation \eqref{eqn:twisted_NLS} is a Hamiltonian system of the structure \eqref{eqn:Hamiltonian_structure_NLS} with Hamiltonian coordinates
$\xi^{(v)}_a=\hat{v}_a,\eta^{(v)}_a=\overline{\hat{v}_a}$ and \textit{time-dependent} Hamiltonian
\begin{align}\nonumber
	{\mathcal{H}}^{(v)}(t;\xi^{(v)},\eta^{(v)})&=\int_{\mathbb{T}^d}\left|e^{it{\partial_x^2}}u\right|^{2(p+1)}\dd x\\\label{eqn:Hamiltonian_twisted}
	&=\sum_{\substack{a_i,b_i\in\mathbb{Z}^d\\a_1+\dots+a_{p+1}=b_1+\dots+b_{p+1}}}c^{(v)}_{a_1,\dots,a_{p+1},b_1,\dots,b_{p+1}}(t)\xi_{a_1}^{(v)}\cdots\xi_{a_{p+1}}^{(v)}\eta_{b_1}^{(v)}\cdots\eta_{b_{p+1}}^{(v)},
\end{align}
where
\begin{align*}
	c^{(v)}_{a_1,\dots,a_{p+1},b_1,\dots,b_{p+1}}(t)=\frac{\mu}{p+1}\exp\left(-it\sum_{l=1}^{p+1}a_l^2+it\sum_{l=1}^{p+1}b_l^2\right).
\end{align*}
Again, if $\bm{\xi},\bm{\eta}\in l^2_s, s>3/2$ all of the above sums converge. For further details on the properties of the NLSE the interested reader is referred to \cite{ablowitz2011nonlinear,tao2006nonlinear}.
\subsection{Symplectic approximations to nonlinear frequency interactions}\label{sec:symplectic_kernel_approximations}
{We note that the approximation of the kernels $\mathcal{K}^{(M,N)}$ in \eqref{eqn:Duhamel_forumla_general_dispersive_Fourier_collected_coeffs} was central to the construction of low-regularity integrators. In order to permit our Runge--Kutta resonance-based methods ultimately to preserve the symplectic structure of the underlying Hamiltonian formulation we have to ensure this approximation originates in another Hamiltonian system with the same symplectic structure. Since for the KdV equation (cf. Section~\ref{sec:RK_res_methods_for_KdV}) no kernel approximation was necessary, this trivially preserves the symplectic structure. However, for the NLSE we have to introduce} a novel low-regularity kernel approximation in Duhamel's formula which respects the symplectic structure of the original equation. In the interest of notational simplicity we focus in the following on the case of the cubic NLSE (i.e. $p=1$ in \eqref{eqn:Cauchy_problem_NLS}) however we note that the construction can be carried out analogously in the general case $p\in\mathbb{N}$ {(see Remark~\ref{rmk:higher_degree_nonlinearities}). While this construction can be performed also in the case $d>1$ the significant low-regularity advantage is only achieved when $d=1$ and thus throughout this section we focus on the case $d=1$.}

{To begin with, we recall that the exact Duhamel kernel in the cubic NLSE is given by \eqref{eqn:cubic_NLSE_kernel}
\begin{align*}
    \mathcal{K}(t;k,k_1,{k}_2,{k}_3)=e^{it(k^2+k_1^2-{k}_2^2-{k}_3^2)}=e^{-2iskk_1+2isk_2k_3},\quad\text{whenever}\ k=k_2+k_3-k_1,
\end{align*}
where for notational convenience we choose to denote the indices $k_1,k_2,k_3$ associated with $\overline{v_{k_1}},\hat{v}_{k_2},\hat{v}_{k_3}$ respectively throughout this section. A simple approximation which preserves the symplectic structure can be obtained from the following} observation:
\begin{align}\begin{split}\label{eqn:second_order_kernel_approx_NLSE}	e^{-2iskk_1+2isk_2k_3}&=e^{-2iskk_1}+e^{2isk_2k_3}-1+\left(e^{-2iskk_1}-1\right)\left(e^{2isk_2k_3}-1\right)\\
 &\approx e^{-2iskk_1}+e^{2isk_2k_3}-1=:\mathcal{K}_{2}(s;k,k_1,k_2,k_3),\end{split}
\end{align}
which incurs a local error of low regularity in the following sense:
{\begin{align}\begin{split}\label{eqn:local_error_second_order_kernel_approx_NLSE}
\left|e^{-2iskk_1+2isk_2k_3}-\mathcal{K}_{2}(s;k,k_1,k_2,k_3)\right|=\left|e^{-2iskk_1}-1\right|\left|e^{2isk_2k_3}-1\right|\leq &s^rC_r|k k_1k_2k_3|^{\frac{r}{2}},\\
&\quad \forall\,0\leq r\leq2, s\in [0,\infty),\end{split}
\end{align}}
for some constants $C_r$ depending only on $r$. We will see how we can construct similar approximations of higher order in {Section}~\ref{sec:symplectic_kernel_approximations_NLSE} below. The approximation \eqref{eqn:second_order_kernel_approx_NLSE} preserves the Hamiltonian formulation in the sense that the expression
{\begin{align*}
\hat{v}_k(t_n+\tau)&=\hat{v}_k(t_n)-i\mu\!\!\!\!\!\!\!\!\sum_{k+k_1=k_2+k_3}\!\!\!\!\!\!\!\!e^{it_n(k^2+k_1^2-k_2^2-k_3^2)}\int_0^{\tau}\mathcal{K}_{2}(s;k,k_1,k_2,k_3)\overline{\hat{v}_{k_1}(t_n+s)} \hat{v}_{k_2}(t_n+s)\hat{v}_{k_3}(t_n+s)\dd s
\end{align*}}
is exactly Duhamel's formula arising from the following piecewise regular infinite-dimensional ODE system which we will see to have similar properties as the NLSE \eqref{eqn:twisted_NLS}
\begin{align}\label{eqn:approximate_ODE_system}
    \frac{\dd}{\dd t} \hat{v}_k(t)=(-i\mu)\sum_{k+k_1=k_2+k_3}e^{it_n(k^2+k_1^2-k_2^2-k_3^2)}\mathcal{K}_{2}(t-t_n;k,k_1,k_2,k_3)\overline{\hat{v}_{k_1}(t)} \hat{v}_{k_2}(t)\hat{v}_{k_3}(t),\ \text{for\ }t\in[t_n,t_{n+1}].
\end{align}
Indeed, from this expression we observe immediately two central properties: Firstly, the normalisation {\eqref{eqn:quadratic_first_integral_nlse} is also} preserved in the approximate system \eqref{eqn:approximate_ODE_system} for any $\hat{v}_k\in l_{s}^2,$ some $s>1/2$, since:
\begin{align*}
    \frac{\dd}{\dd t}\sum_{k\in\mathbb{Z}}\overline{\hat{v}_k}\hat{v}_k&=2\re\left[\sum_{k\in\mathbb{Z}}(-i\mu)\sum_{k+k_1=k_2+k_3}e^{it_n(k^2+k_1^2-k_2^2-k_3^2)}\mathcal{K}_{2}(t-t_n;k,k_1,k_2,k_3)\overline{\hat{v}_k(t)}\overline{\hat{v}_{k_1}(t)} \hat{v}_{k_2}(t)\hat{v}_{k_3}(t)\right]\\
    &=0,
\end{align*}
where the final equality follows immediately from the symmetry
\begin{align}\label{eqn:central_symmetry_kernel_approx_order_2}
    \mathcal{K}_{2}(s;k,k_1,k_2,k_3)=\overline{\mathcal{K}_{2}(s;k_2,k_3,k,k_1)}.
\end{align}
Secondly, as a result of the same symmetry, the ODE system \eqref{eqn:approximate_ODE_system} can also be reformulated as a piecewise infinite-dimensional Hamiltonian system with the same coordinates and symplectic form as introduced in Section~\ref{sec:structure_of_NLSE}:
$\xi^{(v)}_a=\hat{v}_a,\eta^{(v)}_a=\overline{\hat{v}_a}$ and \textit{time-dependent} Hamiltonian
\begin{align}\nonumber
	{\mathcal{H}}^{(v)}(t;\xi^{(v)},\eta^{(v)})&=(-i\mu)\sum_{\substack{k,k_1,k_2,k_3\in\mathbb{Z}^d\\k+k_1=k_2+k_3}}e^{it_n(k^2+k_1^2-k_2^2-k_3^2)}\mathcal{K}_{2}(t-t_n;k,k_1,k_2,k_3)\overline{\hat{v}_k(t)}\xi_{k_2}^{(v)}\xi_{k_3}^{(v)}\eta_{k}^{(v)}\eta_{k_1}^{(v)},
\end{align}
meaning in particular that the symplectic form \eqref{eqn:symplectic_form_NLSE} is preserved over each interval $[t_n,t_n+\tau]$ and thus, by continuity, globally. This justifies the kernel approximation \eqref{eqn:second_order_kernel_approx_NLSE} as a suitable choice and we {can proceed to design RK resonance-based schemes with symplecticity-preserving kernel approximations using $\mathcal{K}_2$ and the construction in \eqref{eqn:RK_res_based_schemes_general}. Note, for a general choice of coefficients $a_{p,q,r}^{\tilde{p},\tilde{q},\tilde{r}},b^{p,q,r}$ these schemes are not necessarily symplectic, but indeed we will provide, in Theorem~\ref{thm:preservation_of_symplectic_form} sufficient conditions on the coefficients for this to be the case.} {We note that the following maps appear in the construction of these RK resonance-based schemes (cf. \eqref{eqn:nonlinear_maps_in_general_RK_schemes}):}
{\begin{align}\nonumber
   v\mapsto &\mathcal{F}^{[NLSE,2]}_p(\tau ;c_q;v)\\\nonumber&\hspace{0.29cm}:={-i\mu}\sum_{k\in\mathbb{Z}}e^{ix k}\sum_{k+k_1=k_2+k_3}\frac{1}{\tau^{p+1}}\int_{0}^{c_q\tau}\left[e^{-2iskk_1}+e^{2isk_2k_3}-1\right] s^{p}\dd s\overline{\hat{v}_{k_1}}\hat{v}_{k_2}\hat{v}_{k_3}\\\label{eqn:central_integrals_in_NLSE_RK_construction_order_2}
   &\quad={-i\mu}\sum_{k\in\mathbb{Z}}e^{ix k}\sum_{k+k_1=k_2+k_3}c_q^{p+1} \left[\varphi_{p+1}(-2ic_q\tau kk_1)+\varphi_{p+1}(2ic_q\tau k_2k_3)-\frac{1}{p+1}\right]\overline{\hat{v}_{k_1}}\hat{v}_{k_2}\hat{v}_{k_3}.
\end{align}}
The final important aspect of the kernel approximation \eqref{eqn:second_order_kernel_approx_NLSE} is that the functionals $\mathcal{F}^{[NLSE,2]}_p(\tau; c_q;v)$ have exact representations in physical space meaning we can use FFT based methods to compute their action on a spectral discretisation, for example:
\begin{align*}
    {(-i\mu)^{-1}}&\mathcal{F}^{[NLSE,2]}_0(\tau ;c_q;v)\\
    &=\frac{i}{2\tau}e^{-ic_q\tau \partial_x^2}\partial_x^{-1}\left(\left[e^{-ic_q\tau \partial_x^2}\overline{\partial_x^{-1}v}\right]\left[e^{ic_q\tau \partial_x^2}\left(v\right)^2\right]\right)-\frac{i}{2\tau}\partial_x^{-1}\left(\overline{\partial_x^{-1}v}\left(v\right)^2\right)\\&\quad+\frac{i}{2\tau}\overline{v}e^{-ic_q\tau \partial_x^2}\left(e^{ic_q\tau \partial_x^2}\partial_x^{-1}v\right)^2-\frac{i}{2\tau}\overline{v}\left(\partial_x^{-1}v\right)^2-c_q |v|^2v\\
    &\quad {+c_q^{2}\left(\int_{\mathbb{T}}|v|^2v\dd x+\overline{\hat{v}_0}v^2-\overline{\hat{v}_0}\int_{\mathbb{T}}\left(v\right)^2\dd x+2\hat{v}_0|v|^2-\left(\hat{v}_0\right)^2\overline{v}\right),}\\
    {(-i\mu)^{-1}}&\mathcal{F}^{[NLSE,2]}_1(\tau ;c_q;v)\\&= -\frac{ic_q}{2\tau}e^{-ic_q\tau \partial_x^2}\partial_x^{-1}\left(\left[e^{-ic_q\tau \partial_x^2}\overline{\partial_x^{-1}v}\right]\left[e^{ic_q\tau \partial_x^2}\left(v\right)^2\right]\right)\\
    &\quad+\frac{1}{4\tau^2}e^{-ic_q\tau \partial_x^2}\partial_x^{-2}\left(\left[e^{-ic_q\tau \partial_x^2}\overline{\partial_x^{-2}v}\right]\left[e^{ic_q\tau \partial_x^2}\left(v\right)^2\right]\right)-\frac{1}{4\tau^2}\partial_x^{-2}\left(\overline{\partial_x^{-2}v}\left(v\right)^2\right)\\
    &\quad+\frac{ic_q}{2\tau}\overline{v}e^{-ic_q\tau \partial_x^2}\left(e^{ic_q\tau \partial_x^2}\partial_x^{-1}v\right)^2+\frac{1}{4\tau^2}\overline{v}e^{-ic_q\tau \partial_x^2}\left(e^{ic_q\tau \partial_x^2}\partial_x^{-2}v\right)^2-\frac{1}{4\tau^2}\overline{v}\left(\partial_x^{-2}v\right)^2\\
    &\quad -\frac{1}{2}c_q^2 |v|^2v{+\frac{c_q^{3}}{2}\left(\int_{\mathbb{T}}|v|^2v\dd x+\overline{\hat{v}_0}v^2-\overline{\hat{v}_0}\int_{\mathbb{T}}\left(v^n\right)^2\dd x+2\hat{v}_0|v|^2-\left(\hat{v}_0\right)^2\overline{v}\right).}\\
\end{align*}
\begin{remark}\label{rmk:need_for_higher_order_kernel_approximations}
    However, we note that by virtue of the estimate \eqref{eqn:local_error_second_order_kernel_approx_NLSE} the approximation \eqref{eqn:second_order_kernel_approx_NLSE} necessarily incurs a local error of order {$\mathcal{O}(\tau^3)$} in the Duhamel formula, meaning any numerical scheme designed based on this formula can have global convergence order no more than {$\mathcal{O}(\tau^2)$}. Thus we need to find a structured way of performing kernel approximations with comparable properties to arrive at {higher order methods}.
\end{remark}

\subsubsection{Systematic {higher order} symplectic kernel approximations for the NLSE}\label{sec:symplectic_kernel_approximations_NLSE}
We note that the above construction of structure respecting kernel approximation is entirely new in the construction of low-regularity integrators and differs significantly from prior work (e.g. in \cite{ostermann2018low, bruned_schratz_2022}). {However, as mentioned in Remark~\ref{rmk:need_for_higher_order_kernel_approximations}, when using \eqref{eqn:central_symmetry_kernel_approx_order_2} in the construction of a Runge--Kutta resonance-based scheme, the resulting method can be at most of second order in $\tau$. Thus let us now describe how we can construct higher order symplecticity-preserving approximations to the kernel function $\exp(-2iskk_1+2isk_2k_3)$ to ultimately attain Runge--Kutta resonance-based schemes of higher order.} Let us fix $\sigma\in\mathbb{N}$ choose a set of distinct interpolation points $0\leq \gamma_1<\dots<\gamma_{\sigma}\leq 1$ (for example Clenshaw--Curtis points \cite{clenshaw1960method,trefethen2019approximation}) and define by $\mathcal{P}_\sigma[f]$ the unique interpolating polynomial of degree $\sigma-1$ which matches the function values of $f$ at the points $\tau \gamma_j, j=1,\dots, \sigma$, i.e.
\begin{align*}
    \mathcal{P}_\sigma[f](\tau \gamma_j)=f(\tau \gamma_j), \ j=1,\dots, \sigma.
\end{align*}
Let us consider the kernel approximation $\exp(-2is kk_1+2isk_2k_3)\approx\mathcal{K}_{2\sigma}(s;k,k_1,k_2,k_3)$ where
\begin{align}\begin{split}\label{eqn:general_stable_kernel_approx}
    \mathcal{K}_{2\sigma}(s;k,k_1,k_2,k_3)&:=e^{-2iskk_1}\mathcal{P}_\sigma[\exp(2i\,\cdot\,k_2k_3)]+e^{2isk_2k_3}\mathcal{P}_\sigma[\exp(-2i\,\cdot\,kk_1)]\\&\quad\quad\quad\quad\quad\quad\quad\quad\quad\quad\quad\quad-\mathcal{P}_\sigma[\exp(2i\,\cdot\,k_2k_3)]\mathcal{P}_\sigma[\exp(-2i\,\cdot\,kk_1)],\end{split}
\end{align}
which immediately results in a stable numerical scheme since all coefficients that appear are of the form $\exp(is \omega), \omega\in\mathbb{R},$ and therefore uniformly bounded. Moreover, the approximations $\mathcal{K}_{2\sigma}$ have the central symmetry
\begin{align}\label{eqn:central_symmetry_kernel_approx}
\mathcal{K}_{2\sigma}(s;k,k_1,k_2,k_3)=\overline{\mathcal{K}_{2\sigma}(s;k_2,k_3,k,k_1)},
\end{align}
which analogously to \eqref{eqn:central_symmetry_kernel_approx_order_2} immediately implies that the approximated Duhamel formula arises under a modified flow which exactly preserves the symplectic form and quadratic first integral (the normalisation). In addition, we can estimate the local error of this approximation {as follows.}
{\begin{proposition}\label{prop:local_error_higher_order_symplectic_kernel_approx} The approximation \eqref{eqn:general_stable_kernel_approx} is such that for any {measurable function $g:[0,\tau]\rightarrow\mathbb{C}$ and} any $d\in\mathbb{N},\beta\in[0,2d]$:
    \begin{align*}
        \left|\int_{0}^{\tau} e^{-2is kk_1+2isk_2k_3} g(s)\dd s-\int_{0}^{\tau} \mathcal{K}_{2\sigma}(s;k,k_1,k_2,k_3)g(s)\dd s\right|\leq \tau^{1+2\beta} |kk_1k_2k_3|^{\beta}\sup_{s\in[0,\tau]}|g(s)|
    \end{align*}
\end{proposition}}
{This is} based on the following standard approximation result:
\begin{lemma}[{Theorem 4.2 in \cite{powell1981approximation}}]
If $f\in\mathcal{C}^{(\sigma)}([0,\tau])$, then for every $s\in[0,\tau]$ there is a $\theta_s\in[0,\tau]$ such that
\begin{align*}
    f(s)-P_\sigma[f](s)=\frac{f^{(\sigma)}(\theta_s)}{\sigma!}\Pi_{j=1}^{\sigma}(s-\gamma_j).
\end{align*}
\end{lemma}
Note this implies the following simple consequence by interpolation:
\begin{cor}\label{cor:refined_interpolation_estimate}
    Suppose $\{f_\alpha(s)\}_{\alpha\in A}$ is a family of $\mathcal{C}^\sigma$ functions which are uniformly bounded in their function values and $\sigma^{th}$ derivatives in the form
    \begin{align*}
    |f^{(j)}_\alpha(s)|\leq C_{j,\alpha},\, j=0,\sigma, 
    \end{align*}
    for some constants $C_{j,\alpha}>0$, then we have the following bound for all $s\in[0,\tau]$ and all $\beta\in[0,\sigma]$:
    \begin{align*}
        |f_\alpha(s)-P_\sigma[f_\alpha](s)|\leq \tau^\beta \left(\frac{C_{\sigma,\alpha}}{\sigma!}\right)^{\beta}C_{0,\alpha}^{1-\beta}.
    \end{align*}
\end{cor}
\begin{proof}[{Proof of Proposition~\ref{prop:local_error_higher_order_symplectic_kernel_approx}}] We observe the simple identity
    \begin{align*}
    e^{-2is kk_1+2isk_2k_3}&=e^{-2iskk_1}\mathcal{P}_\sigma[\exp(2i\,\cdot\,k_2k_3)](s)+e^{2isk_2k_3}\mathcal{P}_\sigma[\exp(-2i\,\cdot\,kk_1)](s)\\
    &\quad\quad\quad -\mathcal{P}_\sigma[\exp(2i\,\cdot\,k_2k_3)](s)\mathcal{P}_\sigma[\exp(-2i\,\cdot\,kk_1)](s)\\
    &\quad\quad\quad+\left(e^{-2iskk_1}-\mathcal{P}_\sigma[\exp(-2i\,\cdot\,kk_1)](s)\right)\left(e^{2isk_2k_3}-\mathcal{P}_\sigma[\exp(2i\,\cdot\,k_2k_3)](s)\right).
\end{align*}
This identity combined with {Corollary~\ref{cor:refined_interpolation_estimate}} immediately yields the desired estimate.
\end{proof}

{As above, in the RK resonance-based schemes we then fix $\sigma\in\mathbb{N}$ and consider the following nonlinear operators:}
\begin{align}\label{eqn:central_integrals_in_NLSE_RK_construction}
	\mathcal{F}^{[NLSE,2\sigma]}_p(\tau;c_q;v)=(-i\mu)\sum_{k\in\mathbb{Z}}e^{ixk}\sum_{k+k_1=k_2+k_3}\frac{1}{\tau^{p+1}}\int_0^{c_q\tau}\mathcal{K}_{2\sigma}(s;k,k_1,k_2,k_3)s^p\dd s  \overline{\hat{v}_{k_1}} \hat{v}_{k_2}\hat{v}_{k_3}.
\end{align}
We note that the integrals appearing in the definition of $\mathcal{F}^{[NLSE,2\sigma]}_p(\tau;c_q;v)$ result in terms of the form
\begin{align}\label{eqn:general_expression_F_NLSE_p_1}
    \sum_{k+k_1=k_2+k_3}\frac{1}{\tau}e^{2i\gamma_j\tau k_2k_3}\varphi_{l}(-2ic_q\tau kk_1) \overline{\hat{v}_{k_1}} \hat{v}_{k_2}\hat{v}_{k_3},\  \text{and}\ 
    \sum_{k+k_1=k_2+k_3}\frac{1}{\tau}e^{-2i\gamma_j\tau kk_1}\varphi_{l}(2ic_q\tau k_2k_3) \overline{\hat{v}_{k_1}} \hat{v}_{k_2}\hat{v}_{k_3},
\end{align}
for some $l\in\mathbb{N},\gamma_j,c_q\in\mathbb{R}$. All of these terms have an expression in physical space, meaning the action of $\mathcal{F}^{[NLSE,2\sigma]}_p(\tau;c_q;v)$ can be computed efficiently using FFT based methods for a spectral spatial discretisation. An example of this is the term $T^{(1)}$, given by 
\begin{align*}
     \hat{T}^{(1)}_k=\sum_{k+k_1=k_2+k_3}\frac{1}{\tau}e^{2i\gamma_j\tau k_2k_3}\varphi_{1}(-2ic_q\tau kk_1) \overline{\hat{v}_{k_1}} \hat{v}_{k_2}\hat{v}_{k_3},
\end{align*}
which can be written in the form
\begin{align*}
     \hat{T}^{(1)}_k&=\sum_{\substack{k+k_1=k_2+k_3\\k,k_1\neq0}}\frac{1}{\tau}e^{i\gamma_j\tau (k_2+k_3)^2}e^{-i\gamma_j\tau k_2^2}e^{-i\gamma_j\tau k_3^2}\frac{(-i)}{2c_q\tau (ik)(ik_1)}\left(e^{-i\tau c_q (k_2+k_3)^2}e^{i\tau c_q k^2}e^{i\tau c_qk_1^2}-1\right)\overline{\hat{v}_{k_1}} \hat{v}_{k_2}\hat{v}_{k_3}\\
     &\quad+ \sum_{\substack{k+k_1=k_2+k_3\\k=0}}\frac{1}{\tau}e^{i\gamma_j\tau (k_2+k_3)^2}e^{-i\gamma_j\tau k_2^2}e^{-i\gamma_j\tau k_3^2}\overline{\hat{v}_{k_1}} \hat{v}_{k_2}\hat{v}_{k_3}\\
      &\quad+ \sum_{\substack{k+k_1=k_2+k_3\\k_1=0}}\frac{1}{\tau}e^{i\gamma_j\tau (k_2+k_3)^2}e^{-i\gamma_j\tau k_2^2}e^{-i\gamma_j\tau k_3^2}\overline{\hat{v}_{k_1}} \hat{v}_{k_2}\hat{v}_{k_3}\\
       &\quad- \sum_{\substack{k+k_1=k_2+k_3\\k=k_1=0}}\frac{1}{\tau}e^{i\gamma_j\tau (k_2+k_3)^2}e^{-i\gamma_j\tau k_2^2}e^{-i\gamma_j\tau k_3^2}\overline{\hat{v}_{k_1}} \hat{v}_{k_2}\hat{v}_{k_3}\\
    &=\sum_{\substack{k+k_1=k_2+k_3\\k,k_1\neq0}}\frac{1}{\tau}e^{i\gamma_j\tau (k_2+k_3)^2}e^{-i\gamma_j\tau k_2^2}e^{-i\gamma_j\tau k_3^2}\frac{(-i)}{2c_q\tau (ik)(ik_1)}\left(e^{-i\tau c_q (k_2+k_3)^2}e^{i\tau c_q k^2}e^{i\tau c_qk_1^2}-1\right)\overline{\hat{v}_{k_1}} \hat{v}_{k_2}\hat{v}_{k_3}\\
     &\quad+ \sum_{\substack{k+k_1=k_2+k_3\\k=0}}\frac{1}{\tau}e^{i\gamma_j\tau (k_2+k_3)^2}e^{-i\gamma_j\tau k_2^2}e^{-i\gamma_j\tau k_3^2}\overline{\hat{v}_{k_1}} \hat{v}_{k_2}\hat{v}_{k_3}\\
      &\quad+ \sum_{\substack{k+k_1=k_2+k_3\\k_1=0}}\frac{1}{\tau}e^{i\gamma_j\tau (k_2+k_3)^2}e^{-i\gamma_j\tau k_2^2}e^{-i\gamma_j\tau k_3^2}\overline{\hat{v}_{0}} \hat{v}_{k_2}\hat{v}_{k_3}\\
       &\quad- \sum_{\substack{k_2+k_3=0\\k=k_1=0}}\frac{1}{\tau}e^{-i\gamma_j\tau k_2^2}e^{-i\gamma_j\tau k_3^2}\overline{\hat{v}_{0}} \hat{v}_{k_2}\hat{v}_{k_3},
\end{align*}
and thus has the following representation in physical space
\begin{align*}
    T^{(1)}&=\frac{(-i)}{2c_q\tau^2}e^{-i\tau c_q\partial_x^2}\partial_x^{-1}\left[\overline{e^{i\tau c_q\partial_x^2}\partial_x^{-1}v}e^{-i(\gamma_j-c_q)\tau \partial_x^2}\left(e^{i\gamma_j\tau \partial_x^2}v\right)^2\right]+\frac{i}{2c_q\tau^2}\overline{v}e^{-i\gamma_j\tau \partial_x^2}\left(e^{i\gamma_j\tau \partial_x^2}v\right)^2\\
    &\quad+\frac{1}{\tau}\int_{\mathbb{T}}\overline{v}e^{-i\gamma_j\tau \partial_x^2}\left(e^{i\gamma_j\tau \partial_x^2}v\right)^2\dd x+\frac{1}{\tau} \overline{\hat{v}_{0}}e^{-i\gamma_j\tau \partial_x^2}\left(e^{i\gamma_j\tau \partial_x^2}v\right)^2+\frac{1}{\tau} \overline{\hat{v}_{0}}\int_{\mathbb{T}}\left(e^{i\gamma_j\tau \partial_x^2}v\right)^2\dd x,
\end{align*}
where we recall the definition of $\partial_x^{-1}$ from \eqref{eqn:definition_dx_inv}. Similar expressions are available for the remaining terms appearing in \eqref{eqn:general_expression_F_NLSE_p_1}.
\begin{remark}\label{rmk:higher_degree_nonlinearities}
For higher order nonlinearities ($p>1$ in \eqref{eqn:Cauchy_problem_NLS}) Duhamel's formula in the twisted variable takes the form
\begin{align*}
    \hat{v}_{a_0}(t_n+\tau)&=\hat{v}_{a_0}(t_n)-i\mu\!\!\!\!\!\!\!\!\!\!\!\!\!\!\sum_{\sum_{j=0}^\mu a_j-\sum_{j=0}^\mu b_j}e^{it_n(\sum_{j=0}^\mu a_j^2-\sum_{j=0}^\mu b_j^2)}\\
    &\quad\quad\quad\quad\quad\int_0^\tau \!\!e^{-is(\sum_{j,k=0, j\neq k}^\mu a_ja_k-\sum_{j,k=0, j\neq k}^\mu b_jb_k)}\prod_{j=1}^{\mu}\overline{\hat{v}_{a_j}(t_n+s)}\prod_{j=0}^{\mu}\hat{v}_{b_j}(t_n+s)\dd s.
\end{align*}
Thus the Duhamel kernel is given by $\exp\left(-is\sum_{j,k=0, j\neq k}^\mu a_ja_k+is\sum_{j,k=0, j\neq k}^\mu b_jb_k\right)$ for which we can use similar kernel approximations to the above, in particular for a second order approximation we can write
\begin{align}\label{eqn:higher_order_nonlinearity_second_order_approx}
e^{-is\sum_{j,k=0, j\neq k}^\mu a_ja_k+is\sum_{j,k=0, j\neq k}^\mu b_jb_k}&\approx\sum_{j,k=0, j\neq k}^\mu e^{-is a_ja_k}+\sum_{j,k=0, j\neq k}^\mu e^{isb_jb_k}-2\mu+1
\end{align}
which, by iterating \eqref{eqn:second_order_kernel_approx_NLSE}, leads to the local error
\begin{align*}
    &\left(e^{-is\sum_{j,k=0, j\neq k}^\mu a_ja_k+is\sum_{j,k=0, j\neq k}^\mu b_jb_k-2is b_0b_1}-1\right)\left(e^{2is b_0b_1}-1\right)\\
    &\quad\quad+\left(e^{-is\sum_{j,k=0, j\neq k}^\mu a_ja_k+is\sum_{j,k=0, j\neq k}^\mu b_jb_k-2is b_0b_1-2is b_0b_2}-1\right)\left(e^{2is b_0b_1}-1\right)\left(e^{2is b_0b_2}-1\right)+\cdots,
\end{align*}
and therefore, in the spirit of estimate \eqref{eqn:local_error_second_order_kernel_approx_NLSE}, the approximation \eqref{eqn:higher_order_nonlinearity_second_order_approx} leads to a loss of $3\beta/2$ derivatives to provide an approximation of order $\mathcal{O}(\tau^{\beta})$ to the kernel and a local error of $\mathcal{O}(\tau^{1+\beta})$ in any resulting scheme so long as the Runge--Kutta interpolation coefficients are chosen appropriately. This is still significantly better than splitting methods and exponential integrators which require a loss of $2\beta$ derivatives to achieve a comparable local error. 
\end{remark}
\subsubsection{Further examples of RK resonance-based schemes for NLSE}
Following the above discussion {of constructing symplecticity-preserving kernel approximations in Section~\ref{sec:symplectic_kernel_approximations_NLSE} (specifically Proposition~\ref{prop:local_error_higher_order_symplectic_kernel_approx})} let us recall the RK resonance-based schemes for the NLSE {as follows (cf. Section~\ref{sec:RK_res_schemes_for_general_eqns}):}
\begin{align}\begin{split}\label{eqn:RK_res_based_schemes_NLSE}
	u^{n+1}&=e^{i\tau{\partial_x^2}}u^{n}+\tau \sum_{p,q,r=0}^{S} b^{p,q,r}e^{i\tau{\partial_x^2} }K_{p,q,r},\\
	K_{p,q,r}&=\mathcal{F}^{[NLSE,2d]}_p(\tau; c_q; u^{n}+\tau \sum_{\tilde{p},\tilde{q},\tilde{r}=0}^{S} a_{p,q,r}^{\tilde{p},\tilde{q},\tilde{r}}K_{\tilde{p},\tilde{q},\tilde{r}}),\end{split}
\end{align}
for some constants $c_q, a_{p,q,r}^{\tilde{p},\tilde{q},\tilde{r}},b^{p,q,r}\in\mathbb{R}$. To illustrate that the formulation \eqref{eqn:RK_res_based_schemes_NLSE} encompasses a wide range of low-regularity integrators we consider the following two examples.

\begin{example}[First order method {comparable to \cite{ostermann2018low}}] 
    Take $d=1$,$S=0, c_0=b^{0,0,0}=1,a_{0,0,0}^{0,0,0}=0$, leading to the scheme
    \begin{align*}
        	u^{n+1}&=e^{i\tau{\partial_x^2} }u^{n}+\tau e^{i\tau{\partial_x^2} }\mathcal{F}^{[NLSE,2]}_0(\tau; 1; u^{n}), 
    \end{align*}
    which in Fourier coordinates corresponds to
    \begin{align*}
        \hat{u}^{n+1}_k=e^{-ik^2\tau}\hat{u}^n_k+(-i\mu)\tau e^{-ik^2\tau}\sum_{k+k_1=k_2+k_3} \left[\varphi_{1}(-2i\tau kk_1)+\varphi_{1}(2i\tau k_2k_3)-1\right]\overline{\hat{u}_{k_1}^n}\hat{u}_{k_2}^n\hat{u}_{k_3}^n.
    \end{align*}
    \end{example}
    Using the estimate \eqref{eqn:local_error_second_order_kernel_approx_NLSE} and a stability analysis comparable to Section~\ref{sec:convergence_analysis_NLSE} one can prove the following convergence result for this method:
    \begin{proposition}
         {Let} $\gamma\in[0,1], T>0$ and $r>1/2${. Suppose, for $u_0\in H^{r+\gamma}$, the NLSE \eqref{eqn:Cauchy_problem_NLS} has a solution $u(\,\cdot\,)$ on $t\in[0,T]$ taking values in $H^{r+\gamma}$. Then} there are constants $C,\tau_0>0${, independent of $u_0,u$,} such that for all $\tau\in(0,\tau_0)$ we have
    \begin{align*}
        \|u(n\tau)-u^n\|_{H^r}\leq C \tau^{\gamma}\sup_{t\in[0,n\tau]}\|u(t)\|_{H^{r+\gamma}},\quad\forall 0\leq n\leq \lfloor T/\tau\rfloor{.}
    \end{align*}
    \end{proposition}
\begin{example}[Optimal (non-symplectic) second-order scheme] We can also consider the case $S=1$ which allows us to consider the following construction:
\begin{align}\begin{split}\label{eqn:NLSE_non-symplectic_second_order_integrator}
    u^{n+1}&=e^{i{\partial_x^2} \tau}\left[u^{n}+\tau K_{0,0,0}+\tau K_{1,0,1}-\tau K_{1,0,0}\right],\\
	K_{0,0,0}&=\mathcal{F}^{[NLSE,2]}_0(\tau;1;u^n),\\
	K_{1,0,0}&=\mathcal{F}^{[NLSE,2]}_1(\tau;1;u^n),\\
	K_{1,0,1}&=\mathcal{F}^{[NLSE,2]}_1(\tau;1;u^{n}+\tau K_{0,0,0}+\tau K_{1,0,1}-\tau K_{1,0,0}),\end{split}
\end{align}
corresponding to the choice of parameters
\begin{align*}
    a^{0,0,0}_{1,0,1}&=1, a^{1,0,0}_{1,0,1}=-1, a^{1,0,1}_{1,0,1}=1,\\
    b^{0,0,0}&=1,b^{1,0,0}=-1,b^{1,0,1}=1, c_0=1,
\end{align*}
and all remaining constants equal to zero. This is equivalent to the following implicit numerical scheme:
\begin{align}\begin{split}\label{eqn:non-symplectic_second_order_NLSE_fourier_coordinates}
\hat{u}^{n+1}_k&=e^{-ik^2\tau}\hat{u}^n_k+(-i\mu)e^{-ik^2\tau}\sum_{k+k_1=k_2+k_3}\int_0^\tau \left[e^{-2iskk_1}+e^{2isk_2k_3}-1\right]\\
    &\quad\quad\quad\quad\quad\quad\quad\quad\quad\quad\quad\quad\quad\left[\overline{\hat{u}_{k_1}^n}\hat{u}_{k_2}^n\hat{u}_{k_3}^n+\frac{s}{\tau}\left(e^{i\tau(k_2^2+k_3^2-k_1^2)}\overline{\hat{u}_{k_1}^{n+1}}\hat{u}_{k_2}^{n+1}\hat{u}_{k_3}^{n+1}-\overline{\hat{u}_{k_1}^n}\hat{u}_{k_2}^n\hat{u}_{k_3}^n\right)\right]\dd s.
    \end{split}
\end{align}
\end{example}
It turns out this scheme has a local error of the following form:
\begin{proposition}\label{prop:local_error_non-symplectic_second_order_scheme} {Fix $R>0,s>1/2,\gamma\in[0,2]$, and let} us denote by {$\tau\mapsto\phi_{\tau}(u(t_n))\in H^{s+2}$} the solution to \eqref{eqn:Cauchy_problem_NLS} with initial condition {$u(t_n)\in H^{s+2}$} and denote by $\Phi_{\tau}$ the time-stepping scheme \eqref{eqn:NLSE_non-symplectic_second_order_integrator}{. Then} there is a $\tau_R>0$ such that for all $\tau\in [0,\tau_R)$ we have
whenever $\sup_{t\in[0,\tau]}\|{\phi_{t}}(u(t_n))\|_{H^{s+2}}<R$ then
\begin{align*}
	\left\|{\phi_\tau}(u(t_n))-\Phi_{\tau}(u(t_n))\right\|_{H^s}\leq c_R\tau^{3} 
	\end{align*}
for some constant $c_R>0$ depending on $R>0,s$.
\end{proposition}
\begin{proof}
This claim is proved in Appendix~\ref{app:proof_local_error_non-symplectic_second_order_scheme}.
\end{proof}
\subsection{Structure preservation properties}\label{sec:structure_preservation}
The ultimate purpose of our construction of RK resonance-based methods \eqref{eqn:RK_res_based_schemes_KdV} \& \eqref{eqn:RK_res_based_schemes_NLSE} is to use the large number of additional degrees of freedom as compared to prior (explicit) low-regularity schemes to facilitate structure preservation properties of the corresponding methods.

\subsubsection{Conservation of quadratic invariants in the direct flow of KdV and NLSE}
We begin our discussion by characterising those RK resonance-based methods which conserve the quadratic invariants \eqref{eqn:quadratic_first_integral_kdv} \& \eqref{eqn:quadratic_first_integral_nlse}.
\begin{theorem}\label{thm:L2_pres_condition_direct_flow} Suppose {that} the real-valued coefficients $b^{p,q,r},a^{\tilde{p},\tilde{q},\tilde{r}}_{p,q,r}$ satisfy
	\begin{align}\label{eqn:condition_for_L2_preservation}
	{b^{\tilde{p},\tilde{q},\tilde{r}}}b^{p,q,r}=b^{p,q,r}{a^{\tilde{p},\tilde{q},\tilde{r}}_{p,q,r}}+{b^{\tilde{p},\tilde{q},\tilde{r}}}a^{p,q,r}_{\tilde{p},\tilde{q},\tilde{r}},
	\end{align}
 for all indices $0\leq p,q,r,\tilde{p},\tilde{q},\tilde{r}\leq S$. Then the RK resonance-based schemes \eqref{eqn:RK_res_based_schemes_KdV} \& \eqref{eqn:RK_res_based_schemes_NLSE} preserve the corresponding quadratic first integrals, \eqref{eqn:quadratic_first_integral_kdv} \& \eqref{eqn:quadratic_first_integral_nlse}, exactly {whenever $u\in H^{r}$ with $r\geq 1$ \& $r>1/2$} respectively.
\end{theorem}
{\begin{remark}
    We note that the condition \eqref{eqn:condition_for_L2_preservation} is indeed very similar to the characterisation of classical Runge--Kutta schemes which preserve quadratic invariants. In particular, it was first shown in \cite{cooper87} that a Runge--Kutta method with coefficient matrix $A_{ij}, 0\leq i,j\leq S$, and weights $b_i, 0 \leq S$, preserves quadratic invariants of the underlying differential equation if 
\begin{align*}
    b_i A_{i j}+b_j A_{j i}=b_i b_j \quad \text { for all } i, j=0, \ldots, S.
\end{align*}
\end{remark}}
In the proof {of Theorem~\ref{thm:L2_pres_condition_direct_flow}} we rely on the following crucial lemma (where we denote by $C^\infty$ infinitely differentiable functions on the domain $\mathbb{T}$):
\begin{lemma}\label{lem:orthogonality_condition_L2}
	For any $u{\in C^\infty}$ we have
	\begin{align}\label{eqn:orthogonality_condition_KdV}
	\int_{\mathbb{T}} u\mathcal{F}^{[KdV]}_p(\tau; c_q;u)\dd x&=0,\\\label{eqn:orthogonality_condition_NLSE}
 \re\left[\int_{\mathbb{T}} \overline{u}\mathcal{F}^{[NLSE,2d]}_p(\tau; c_q;u)\dd x\right]&=0,
	\end{align}
 for all $p,d\in\mathbb{N},\tau\in\mathbb{R}_{+},c_q\in[0,1]$, where the nonlinear operators in the identities are defined as in \eqref{eqn:central_integrals_in_KdV_RK_construction} \& \eqref{eqn:central_integrals_in_NLSE_RK_construction}.
\end{lemma}
\begin{proof}
Since \eqref{eqn:quadratic_first_integral_kdv} is preserved under the exact flow of the KdV equation we have, for all initial data $u^0$
\begin{align}\label{eqn:direct_orthogonality_KdV}
	\sum_{m\in\mathbb{Z}} \hat{u}_m(t) \sum_{-m=a+b} (-i)m\e^{it3mab} \hat{u}_a(t)\hat{u}_b(t)=2\int_{\mathbb{T}} u(t,x)u'(t,x)\dd x=\frac{\dd}{\dd t}\int_{\mathbb{T}} u(t,x)^2\dd x=0{.}
 \end{align}
 Thus in particular the identity \eqref{eqn:direct_orthogonality_KdV} holds at $t=0$ and therefore for any (time-independent) $u\in H^{r}$ we have
\begin{align*}
 \sum_{m\in\mathbb{Z}} \hat{u}_m \sum_{-m=a+b} (-i)m\e^{it3mab} \hat{u}_a\hat{u}_b=0.
\end{align*}
Note due to the regularity assumptions on $u$ and local well-posedness of KdV all of the above sums converge absolutely. Multiplying \eqref{eqn:direct_orthogonality_KdV} by $s^p\tau^{-(p+1)}$ and integrating over $[0,c_q\tau]$ with respect to $s$ immediately implies \eqref{eqn:orthogonality_condition_KdV}.

The orthogonality property \eqref{eqn:orthogonality_condition_NLSE} follows in similar vein from \eqref{eqn:central_symmetry_kernel_approx} which implies that for any fixed (time-independent) $u\in H^{r}$
\begin{align*}
    2\re\left[\sum_{k\in\mathbb{Z}}(-i\mu)\sum_{k+k_1=k_2+k_3}e^{it_n(k^2+k_1^2-k_2^2-k_3^2)}\mathcal{K}_{d}(t-t_n;k,k_1,k_2,k_3)\overline{\hat{u}_k}\overline{\hat{u}_{k_1}} \hat{v}_{k_2}\hat{v}_{k_3}\right]&=0
\end{align*}
which we can again multiply by $s^p\tau^{-(p+1)}$ and integrate over $[0,c_q\tau]$ to arrive at \eqref{eqn:orthogonality_condition_NLSE}.
\end{proof}
\begin{proof}[Proof of Theorem~\ref{thm:L2_pres_condition_direct_flow}] Using \eqref{eqn:orthogonality_condition_NLSE} we will prove the result for the NLSE and we note that the corresponding characterisation for the KdV equation follows analogously from using \eqref{eqn:orthogonality_condition_KdV}. {Suppose to begin with that $u\in C^{\infty}$.} Stability estimates similar to those presented in {Section}~\ref{sec:convergence_analysis_NLSE} show that under the assumptions of the theorem $K_{p,q,r}\in H^{r}$ for each $0\leq p,q,r\leq S$. We have
	\begin{align*}
	\int_{\mathbb{T}} |u^{n+1}|^2-|u^{n}|^2\dd x&=\int_{\mathbb{T}}|u^{n+1}-u^n|^2\dd x+2\re\int_{\mathbb{T}}\overline{u^{n}}(u^{n+1}-u^n)\dd x\\
	&=\int_{\mathbb{T}}\left|\tau e^{i\tau{\partial_x^2}} \sum_{p,q,r=0}^{S} b^{p,q,r}K_{p,q,r}\right|^2\dd x+2\re\int_{\mathbb{T}}\overline{u^{n}}\left(\tau  e^{i\tau{\partial_x^2}}\sum_{p,q,r=0}^{S} b^{p,q,r}K_{p,q,r}\right)\dd x\\
	&=\tau^2\sum_{p,q,r,\tilde{p},\tilde{q},\tilde{r}=0}^S{b^{\tilde{p},\tilde{q},\tilde{r}}}b^{p,q,r}\int_{\mathbb{T}}\overline{ e^{i\tau{\partial_x^2}}K_{\tilde{p},\tilde{q},\tilde{r}}} e^{i\tau{\partial_x^2}}K_{p,q,r}\dd x\\
	&\quad+2\tau \sum_{p,q,r=0}^S\re\left[b^{p,q,r}\int_{\mathbb{T}}\overline{u^n} e^{i\tau{\partial_x^2}}\mathcal{F}^{[NLSE,d]}_p\left(c_q;\tau;u^{n}+\tau \sum_{\tilde{p},\tilde{q},\tilde{r}=0}^Sa_{p,q,r}^{\tilde{p},\tilde{q},\tilde{r}}K_{\tilde{p},\tilde{q},\tilde{r}}\right)\dd x\right].
	\end{align*}
Thus, using \eqref{eqn:orthogonality_condition_NLSE}, we find
\begin{align*}
&=\tau^2\sum_{p,q,r,\tilde{p},\tilde{q},\tilde{r}=0}^S{b^{\tilde{p},\tilde{q},\tilde{r}}}b^{p,q,r}\int_{\mathbb{T}}\overline{ e^{i\tau{\partial_x^2}}K_{\tilde{p},\tilde{q},\tilde{r}}} e^{i\tau{\partial_x^2}}K_{p,q,r}\dd x\\
	&\quad-2\tau^2 \sum_{p,q,r=0}^S\re\left[b^{p,q,r}\int_{\mathbb{T}}\overline{\sum_{\tilde{p},\tilde{q},\tilde{r}=0}^Sa_{p,q,r}^{\tilde{p},\tilde{q},\tilde{r}} e^{i\tau{\partial_x^2}}K_{\tilde{p},\tilde{q},\tilde{r}}} e^{i\tau{\partial_x^2}}K_{p,q,r}\dd x\right].
	\end{align*}
Therefore we have (bringing all terms under the same summation)
\begin{align*}
\int_{\mathbb{T}} |u^{n+1}|^2-|u^{n}|^2\dd x&=\sum_{p,q,r,\tilde{p},\tilde{q},\tilde{r}=0}^S\left({b^{\tilde{p},\tilde{q},\tilde{r}}}b^{p,q,r}-b^{p,q,r}{a^{\tilde{p},\tilde{q},\tilde{r}}_{p,q,r}}-{b^{\tilde{p},\tilde{q},\tilde{r}}}a^{p,q,r}_{\tilde{p},\tilde{q},\tilde{r}}\right)\int_{\mathbb{T}}\overline{ e^{i\tau{\partial_x^2}}K_{\tilde{p},\tilde{q},\tilde{r}}} e^{i\tau{\partial_x^2}}K_{p,q,r}\dd x,
\end{align*}
and the result follows {for $u\in C^\infty$}. {From Section~\ref{sec:convergence_analysis_NLSE} we then know that $u^{n}\mapsto u^{n+1}$ is a continuous map on $H^{r}, r>1/2$ and that the map $u\mapsto \int_{\mathbb{T}}|u|^2\dd x$ is continuous as a map from $L^2$ to $\mathbb{R}$. Thus the result follows by density.}
\end{proof}

\subsubsection{Symplectic Runge--Kutta resonance-based methods}
It turns out, analogously to classical Runge--Kutta methods, that the conditions \eqref{eqn:condition_for_L2_preservation} are also sufficient for the method to preserve the symplectic form. For classical Runge--Kutta methods this follows directly from the fact that they are closed under differentiation (cf. \cite{bochev1994quadratic}). Seeing as we are in the PDE case we will instead follow a more direct approach first given by \cite{sanz1988runge}.
\begin{theorem}\label{thm:preservation_of_symplectic_form}
    Suppose that the real-valued coefficients $b^{p,q,r},a^{\tilde{p},\tilde{q},\tilde{r}}_{p,q,r}$ satisfy
	\begin{align}\label{eqn:condition_for_L2_preservation_vsymplecticity}
	{b^{\tilde{p},\tilde{q},\tilde{r}}}b^{p,q,r}=b^{p,q,r}{a^{\tilde{p},\tilde{q},\tilde{r}}_{p,q,r}}+{b^{\tilde{p},\tilde{q},\tilde{r}}}a^{p,q,r}_{\tilde{p},\tilde{q},\tilde{r}},
	\end{align}
 for all indices $0\leq p,q,r,\tilde{p},\tilde{q},\tilde{r}\leq S$. Then the RK resonance-based schemes \eqref{eqn:RK_res_based_schemes_KdV} \& \eqref{eqn:RK_res_based_schemes_NLSE} preserve the corresponding symplectic form (\eqref{eqn:symplectic_form_KdV} \& \eqref{eqn:symplectic_form_NLSE}, respectively) exactly {under the same regularity assumptions as in Theorem~\ref{thm:L2_pres_condition_direct_flow}}.
\end{theorem}
In the interest of brevity we describe the proof of this statement only for the (slightly more challenging) NLSE case, the simpler KdV case follows analogously, adapting relevant notation to the corresponding symplectic formulation of the KdV equation. First of all we note that using the symmetry condition \eqref{eqn:central_symmetry_kernel_approx} we can rewrite the RK resonance-based method \eqref{eqn:RK_res_based_schemes_NLSE} simultaneously both in $u$ and $\overline{u}$ to understand its action on the Hamiltonian coordinates $\bm{\xi}=(\hat{u}_k)_{k\in\mathbb{Z}},\bm{\eta}=(\overline{\hat{u}_k})_{k\in\mathbb{Z}}$:
\begin{align*}
\xi_k^{n+1}&=e^{-ik^2 \tau}\xi_k^{n}+\tau \sum_{p,q,r=0}^{S} b^{p,q,r}e^{-ik^2 \tau}\widehat{\left(K_{p,q,r}\right)}_k,\\
\eta_k^{n+1}&=e^{ik^2 \tau}\eta_k^{n}+\tau \sum_{p,q,r=0}^{S} b^{p,q,r}e^{ik^2 \tau}\widehat{\left(L_{p,q,r}\right)}_k,\\
\bm{K}_{p,q,r}&=\bm{f}^{[NLSE,2d]}_p\left(\tau;c_q;\bm{\xi}+\tau \sum_{\tilde{p},\tilde{q},\tilde{r}=0}^{S} a_{p,q,r}^{\tilde{p},\tilde{q},\tilde{r}}\bm{K}_{\tilde{p},\tilde{q},\tilde{r}};\bm{\eta}+\tau \sum_{\tilde{p},\tilde{q},\tilde{r}=0}^{S} a_{p,q,r}^{\tilde{p},\tilde{q},\tilde{r}}\bm{L}_{\tilde{p},\tilde{q},\tilde{r}}\right),\\
\bm{L}_{p,q,r}&=\bm{g}^{[NLSE,2d]}_p\left(\tau;c_q;\bm{\xi}+\tau \sum_{\tilde{p},\tilde{q},\tilde{r}=0}^{S} a_{p,q,r}^{\tilde{p},\tilde{q},\tilde{r}}\bm{K}_{\tilde{p},\tilde{q},\tilde{r}};\bm{\eta}+\tau \sum_{\tilde{p},\tilde{q},\tilde{r}=0}^{S} a_{p,q,r}^{\tilde{p},\tilde{q},\tilde{r}}\bm{L}_{\tilde{p},\tilde{q},\tilde{r}}\right),
\end{align*}
where we defined the functions $\bm{f}^{[NLSE,2d]}_p,\bm{g}^{[NLSE,2d]}_p$ as follows
\begin{align*}
    \left(f^{[NLSE,2d]}_p(\tau;c_q;\bm{v};\bm{w})\right)_k&=(-i\mu)\sum_{k+k_1=k_2+k_3}\frac{1}{\tau^{p+1}}\int_0^{c_q\tau}\mathcal{K}_{2d}(s;k,k_1,k_2,k_3)s^p\dd s  w_{k_1} v_{k_2}v_{k_3},\\
    \left(g^{[NLSE,2d]}_p(\tau;c_q;\bm{v};\bm{w})\right)_k&=i\mu\sum_{k+k_1=k_2+k_3}\frac{1}{\tau^{p+1}}\int_0^{c_q\tau}\mathcal{K}_{2d}(s;k_2,k_3,k,k_1)s^p\dd s  v_{k_1} w_{k_2}w_{k_3}.
\end{align*}
In the proof of Theorem~\ref{thm:preservation_of_symplectic_form} we rely on the following crucial lemma, which is an analogue of Lemma~\ref{lem:orthogonality_condition_L2} and a direct consequence of the symplectic form being a quadratic first integral of the tangent flow to the NLSE.
\begin{lemma}\label{lem:orthogonality_condition_symplectic_form}
For any $\bm{v},\bm{w}\in l^2_s\left(\mathbb{Z}\right), s>1/2$,
    \begin{align*}
    \dd\bm{v}\wedge\dd\left(\bm{g}^{[NLSE,2d]}_p(\tau;c_q;\bm{v};\bm{w})\right)+\dd\left(\bm{f}^{[NLSE,2d]}_p(\tau;c_q;\bm{v};\bm{w})\right)\wedge\dd\bm{w}=0
    \end{align*}
\end{lemma}
\begin{proof} By the polynomial nature of the nonlinearity, and noting that $$\mathcal{K}_{2d}(s;k_2,k_3,k,k_1)=\mathcal{K}_{2d}(s;k_2,k_3,k_1,k)=\mathcal{K}_{2d}(s;k_3,k_2,k,k_1),$$ we have
\begin{align*}
 \dd\bm{v}\wedge\dd\left(\bm{g}^{[NLSE,2d]}_p(\tau;c_q;\bm{v};\bm{w})\right)&=i\mu \frac{1}{\tau^{p+1}}\int_{0}^{c_q\tau}\sum_{k+k_1=k_2+k_3}\mathcal{K}_{2d}(s;k_2,k_3,k,k_1)\left(\dd v_k\wedge \dd v_{k_1}\right) w_{k_2}w_{k_3}\ s^{p}\dd s\\
 &+2i\mu \frac{1}{\tau^{p+1}}\int_{0}^{c_q\tau}\sum_{k+k_1=k_2+k_3}\mathcal{K}_{2d}(s;k_2,k_3,k,k_1)\left(\dd v_k\wedge \dd w_{k_2}\right) v_{k_1}w_{k_3}\ s^{p}\dd s.
\end{align*}
Since $\wedge$ is antisymmetric and $\mathcal{K}_{2d}(s;k_2,k_3,k,k_1)$ is symmetric in $k,k_1$ the first term cancels and we are left with
\begin{align}\label{eqn:auxilliary_equation_symplectic_identity1}
     \dd\bm{v}\wedge\dd\left(\bm{g}^{[NLSE,2d]}_p(\tau;c_q;\bm{v};\bm{w})\right)&= 2i\mu \frac{1}{\tau^{p+1}}\int_{0}^{c_q\tau}\!\!\!\!\!\!\sum_{k+k_1=k_2+k_3}\!\!\!\!\!\!\mathcal{K}_{2d}(s;k_2,k_3,k,k_1)\left(\dd v_k\wedge \dd w_{k_2}\right) v_{k_1}w_{k_3}\ s^{p}\dd s.
\end{align}
Similarly we have
\begin{align}\nonumber
    \dd\left(\bm{f}^{[NLSE,2d]}_p(\tau;c_q;\bm{v};\bm{w})\right)\wedge\dd\bm{w}&=-2i\mu \frac{1}{\tau^{p+1}}\int_{0}^{c_q\tau}\!\!\!\sum_{k+k_1=k_2+k_3}\!\!\!\!\!\mathcal{K}_{2d}(s;k,k_1,k_2,k_3)\left(\dd v_{k_2}\wedge \dd w_{k}\right) w_{k_1}v_{k_3} s^{p}\dd s\\\label{eqn:auxilliary_equation_symplectic_identity2}
    &=-2i\mu \frac{1}{\tau^{p+1}}\int_{0}^{c_q\tau}\!\!\!\sum_{k+k_1=k_2+k_3}\!\!\!\!\!\mathcal{K}_{2d}(s;k_2,k_3,k,k_1)\left(\dd v_{k}\wedge \dd w_{k_2}\right) w_{k_3}v_{k_1} s^{p}\dd s,
\end{align}
where in the final line we simply relabelled the dummy indices in the summation. Adding \eqref{eqn:auxilliary_equation_symplectic_identity1} \& \eqref{eqn:auxilliary_equation_symplectic_identity2} gives the desired result.
\end{proof}
\begin{proof}[Proof of Theorem~\ref{thm:preservation_of_symplectic_form}]
    We follow the steps taken in \cite{sanz1988runge}. In the interest of brevity we prove the statement for the NLSE and note that it follows analogously for the KdV equation. We want to understand the evolution of $\omega=\sum_{a\in\mathbb{Z}}\dd\xi_a\wedge\dd\eta_a$, in particular we would like to show that
    \begin{align*}
        \dd\bm{\xi}^{n+1}\wedge\dd\bm{\eta}^{n+1}=\dd\bm{\xi}^{n}\wedge\dd\bm{\eta}^{n},
    \end{align*}
    where $\bm{\xi}^{n}=\left(u_k^{n}\right)_{k\in\mathbb{Z}}$ and $\bm{\eta}^{n}=\left(\overline{u_k^{n}}\right)_{k\in\mathbb{Z}}$. 
Differentiating and taking external products we find
\begin{align*}
   \dd\bm{\xi}^{n+1}\wedge\dd\bm{\eta}^{n+1}&=\dd\left(e^{i\tau{\partial_x^2}}\bm{\xi}^{n}\right)\wedge\dd\left(e^{-i\tau{\partial_x^2}}\bm{\eta}^{n}\right)+\sum_{p,q,r=0}^{S} b^{p,q,r}\dd\left(e^{i\tau{\partial_x^2}}\bm{K}_{p,q,r}^{n}\right)\wedge\dd\left(e^{-i\tau{\partial_x^2}}\bm{\eta}^{n}\right)\\
&\quad+\sum_{p,q,r=0}^{S} b^{p,q,r}\dd\left(e^{i\tau{\partial_x^2}}\bm{\xi}^{n}\right)\wedge\dd\left(e^{-i\tau{\partial_x^2}}\bm{L}_{p,q,r}^{n}\right)\\
&\quad+\sum_{p,q,r=0}^{S}\sum_{\tilde{p},\tilde{q},\tilde{r}=0}^{S} b^{p,q,r}b^{\tilde{p},\tilde{q},\tilde{r}}\dd\left(e^{i\tau{\partial_x^2}}\bm{K}_{p,q,r}^{n}\right)\wedge\dd\left(e^{-i\tau{\partial_x^2}}\bm{L}_{p,q,r}^{n}\right).
\end{align*}
By linearity we note that for any vectors $\bm{x},\bm{y}\in l^2_s({\mathbb{Z})},s>1/2,$ we have
\begin{align*}
\dd\left(e^{i\tau{\partial_x^2}}\bm{x}\right)\wedge\dd\left(e^{-i\tau{\partial_x^2}}\bm{y}\right)=\sum_{k\in\mathbb{Z}}\dd \left(e^{-i\tau k^2}x_k\right)\wedge\left(e^{i\tau k^2}y_k\right)=\sum_{k\in\mathbb{Z}}e^{-i\tau k^2}e^{i\tau k^2}\dd x_k\wedge\dd y_k=\dd \bm{x}\wedge\dd\bm{y}.
\end{align*}
Thus the above immediately simplifies to
\begin{align*}
    \dd\bm{\xi}^{n+1}\wedge\dd\bm{\eta}^{n+1}-\dd\bm{\xi}^{n}\wedge\dd\bm{\eta}^{n}&=\sum_{p,q,r=0}^{S} b^{p,q,r}\dd\bm{\xi}^{n}\wedge\dd\bm{L}_{p,q,r}^{n}+\sum_{p,q,r=0}^{S} b^{p,q,r}\dd\bm{K}_{p,q,r}^{n}\wedge\dd\bm{\eta}^{n}\\
    &\quad+\sum_{p,q,r=0}^{S}\sum_{\tilde{p},\tilde{q},\tilde{r}=0}^{S} b^{p,q,r}b^{\tilde{p},\tilde{q},\tilde{r}}\dd\bm{K}_{p,q,r}^{n}\wedge\dd\bm{L}_{p,q,r}^{n}.
\end{align*}
By adding and subtracting the same terms we arrive at
\begin{align*}
    \dd\bm{\xi}^{n+1}\wedge\dd\bm{\eta}^{n+1}-\dd\bm{\xi}^{n}\wedge\dd\bm{\eta}^{n}&=\sum_{p,q,r=0}^{S} b^{p,q,r}\dd\left(\bm{\xi}^{n}+\tau \sum_{\tilde{p},\tilde{q},\tilde{r}=0}^{S} a_{p,q,r}^{\tilde{p},\tilde{q},\tilde{r}}\bm{K}_{\tilde{p},\tilde{q},\tilde{r}}\right)\wedge\dd\bm{L}_{p,q,r}^{n}\\
    &\quad+\sum_{p,q,r=0}^{S} b^{p,q,r}\dd\bm{K}_{p,q,r}^{n}\wedge\dd\left(\bm{\eta}^{n}+\tau \sum_{\tilde{p},\tilde{q},\tilde{r}=0}^{S} a_{p,q,r}^{\tilde{p},\tilde{q},\tilde{r}}\bm{L}_{\tilde{p},\tilde{q},\tilde{r}}\right)\\
    &\quad+\sum_{p,q,r=0}^{S}\sum_{\tilde{p},\tilde{q},\tilde{r}=0}^{S} \left(b^{p,q,r}b^{\tilde{p},\tilde{q},\tilde{r}}-b^{p,q,r}{a^{\tilde{p},\tilde{q},\tilde{r}}_{p,q,r}}-{b^{\tilde{p},\tilde{q},\tilde{r}}}a^{p,q,r}_{\tilde{p},\tilde{q},\tilde{r}}\right)\dd\bm{K}_{p,q,r}^{n}\wedge\dd\bm{L}_{p,q,r}^{n}.
\end{align*}
Thus, if \eqref{eqn:condition_for_L2_preservation_vsymplecticity} holds, then
\begin{align*}
    \dd\bm{\xi}^{n+1}\wedge\dd\bm{\eta}^{n+1}-\dd\bm{\xi}^{n}\wedge\dd\bm{\eta}^{n}&=\sum_{p,q,r=0}^{S} b^{p,q,r}\dd\left(\bm{\xi}^{n}+\tau \sum_{\tilde{p},\tilde{q},\tilde{r}=0}^{S} a_{p,q,r}^{\tilde{p},\tilde{q},\tilde{r}}\bm{K}_{\tilde{p},\tilde{q},\tilde{r}}\right)\wedge\dd\bm{L}_{p,q,r}^{n}\\
    &\quad+\sum_{p,q,r=0}^{S} b^{p,q,r}\dd\bm{K}_{p,q,r}^{n}\wedge\dd\left(\bm{\eta}^{n}+\tau \sum_{\tilde{p},\tilde{q},\tilde{r}=0}^{S} a_{p,q,r}^{\tilde{p},\tilde{q},\tilde{r}}\bm{L}_{\tilde{p},\tilde{q},\tilde{r}}\right),
\end{align*}
and the result follows from Lemma~\ref{lem:orthogonality_condition_symplectic_form}.
\end{proof}

\subsection{Examples}\label{sec:examples_of_symplectic_integrators}

An interesting consequence of the above conditions for symplecticity is that even in the RK resonance-based setting they can only be satisfied by implicit methods, which justifies the paradigm shift from explicit low-regularity integrators to implicit ones taken in the present work.
\begin{cor}\label{cor:symplectic_RK_res_are_implicit}
Any consistent method satisfying \eqref{eqn:condition_for_L2_preservation} is necessarily implicit.
\end{cor}
\begin{proof}
In order for the method \eqref{eqn:RK_res_based_schemes_KdV} (\eqref{eqn:RK_res_based_schemes_NLSE} respectively) to be consistent we have to have $b^{p,q,r}\neq 0$ for some $0\leq p,q,r\leq S$. Consider the identity \eqref{eqn:condition_for_L2_preservation} for $p=\tilde{p},q=\tilde{q},r=\tilde{r}$ which gives
    \begin{align*}
    \left(b^{p,q,r}\right)^2=2b^{p,q,r}a^{p,q,r}_{p,q,r},
    \end{align*}
    i.e. $a^{p,q,r}_{p,q,r}=b^{p,q,r}/2\neq0$ which of course means that the equation defining $K_{p,q,r}$ is implicit.
\end{proof}
In other words, we can at best hope for diagonally implicit symplectic low-regularity integrators in the classes \eqref{eqn:RK_res_based_schemes_KdV} \& \eqref{eqn:RK_res_based_schemes_NLSE}. In the following we present two examples of such symplectic low-regularity schemes. The first example has an analogue in the classical midpoint rule.

\begin{example}[Resonance-based midpoint rule]\label{ex:resonance-based_midpoint_rule}
If we take $S=0$, and the choice $c_0=1, a_{0,0,0}^{0,0,0}=1/2, b_{0,0,0}=1$ then the method \eqref{eqn:RK_res_based_schemes_KdV} simplifies to
\begin{align*}
    u^{n+1}&=e^{-\tau\partial_x^3}u^{n}+\tau e^{-\tau\partial_x^3}K_{0,0,0},\\
	K_{0,0,0}&=\mathcal{F}^{[KdV]}_0(\tau; 1; u^{n}+\frac{\tau}{2}K_{0,0,0}).
\end{align*}
This can be further simplified into the following form
\begin{align*}
    u^{n+1}&=e^{-\tau\partial_x^3}u^{n}+\tau e^{-\tau\partial_x^3}\mathcal{F}^{[KdV]}_0\left(\tau; 1; \frac{u^n+e^{\tau\partial_x^3}u^{n+1}}{2}\right),
\end{align*}
or, alternatively, in physical coordinates
\begin{align}\label{eqn:resonance_based_midpoint_rule_u_KdV}
u^{n+1}=\e^{-\tau\partial_x^3}u^{n}+\frac{1}{24}\left(\partial_x^{-1}u^{n+1}+\e^{-\tau\partial_x^3}\partial_x^{-1}u^n\right)^2-\frac{1}{24}\e^{-\tau\partial_x^3}\left(\e^{\tau \partial_x^3}\partial_x^{-1}u^{n+1}+\partial_x^{-1}u^n\right)^2.
\end{align}
The method resembles the classical midpoint rule, but is able to capture nonlinear frequency interactions in the KdV flow more carefully, thus leading to improved convergence in low-regularity regimes. An analogous method can of course be constructed for the NLSE by taking the same coefficients in \eqref{eqn:RK_res_based_schemes_NLSE}, which leads to the scheme
\begin{align*}
	u^{n+1}&=e^{i\tau {\partial_x^2}}u^{n}+\tau e^{i\tau {\partial_x^2}}\mathcal{F}^{[NLSE,2]}_0\left(\tau; 1;\frac{u^{n}+e^{-i\tau{\partial_x^2}} u^{n+1}}{2}\right),
\end{align*}
taking the following form in physical coordinates:
\begin{align}\begin{split}\label{eqn:resonance_based_midpoint_rule_u_NLSE}
			u^{n+1}&=e^{i\partial_x^2 \tau}u^{n}-i\mu\left[\frac{i}{2}\partial_x^{-1}\left(\left[e^{-i\tau \partial_x^2}\overline{\partial_x^{-1}u^{n+\frac{1}{2}}}\right]\left[e^{i\tau \partial_x^2}\left(u^{n+\frac{1}{2}}\right)^2\right]\right)-\frac{i}{2}e^{i\partial_x^2 \tau}\partial_x^{-1}\left(\overline{\partial_x^{-1}u^{n+\frac{1}{2}}}\left(u^{n+\frac{1}{2}}\right)^2\right)\right]\\
   &\quad-i\mu e^{i\partial_x^2 \tau}\left[\frac{i}{2}\overline{u^{n+\frac{1}{2}}}e^{-i\tau \partial_x^2}\left(e^{i\tau \partial_x^2}\partial_x^{-1}u^{n+\frac{1}{2}}\right)^2-\frac{i}{2}\overline{u^{n+\frac{1}{2}}}\left(\partial_x^{-1}u^{n+\frac{1}{2}}\right)^2-\tau|u^{n+\frac{1}{2}}|^2u^{n+\frac{1}{2}}\right]\\
   &\quad-i\mu\tau\left[\int_{\mathbb{T}}|u^{n+\frac{1}{2}}|^2u^{n+\frac{1}{2}}\dd x+\overline{\hat{u}^{n+\frac{1}{2}}_0}e^{i\partial_x^2 \tau}(u^{n+\frac{1}{2}})^2-\overline{\hat{u}^{n+\frac{1}{2}}_0}\int_{\mathbb{T}}\left(u^{n+\frac{1}{2}}\right)^2\dd x+2\hat{u}_0^{n+\frac{1}{2}}e^{i\partial_x^2 \tau}\left(|u^{n+\frac{1}{2}}|^2\right)\right.\\
   &\quad\quad\quad\quad\quad\quad\quad\left.-\left(\hat{u}^{n+\frac{1}{2}}_0\right)^2e^{i\partial_x^2 \tau}\overline{u^{n+\frac{1}{2}}}\right],
		\end{split}
\end{align}
where $u^{n+\frac{1}{2}}:=(u^n+\exp(-i\tau \partial_x^2)u^{n+1})/{2}$. We will study the convergence properties of this resonance-based midpoint rule in further detail in Sections~\ref{sec:convergence_analysis_KdV} \& \ref{sec:convergence_analysis_NLSE}.
\end{example}
In the next section we will see in particular that the resonance-based midpoint rule has the same low-regularity requirements for convergence as the second order scheme presented in \cite{bruned_schratz_2022}, meaning it converges in $H^{r},r>1/2,$ at order $\mathcal{O}(\tau^{\gamma})$ if the exact solution is at least in $H^{r+\gamma}$ for the NLSE and $H^{r+2\gamma}$ for the KdV equation, for $\gamma\in [0,2]$. 

The symplectic midpoint rule is, of course, just one example of a large number of symplectic low-regularity integrators in this class. By virtue of the construction, for methods of order greater than two, the regularity requirements will be slightly higher for RK resonance-based schemes than for direct explicit resonance-based constructions based on Duhamel iterates \cite{bruned_schratz_2022}. However, even for higher order methods the fact that the first iteration of Duhamel's formula is approximated using highly oscillatory quadrature techniques as shown in Section~\ref{sec:construction_of_rk_res_schemes} means the regularity requirements will still be lower than classical integrators including splitting methods and exponential integrators, even in the higher order setting. A large subclass of symplectic RK resonance-based schemes is given by diagonally implicit schemes:
\begin{example}[Diagonally implicit scheme with $S>0$] Motivated by diagonally implicit symplectic Runge--Kutta methods (cf. \cite[Section 3]{meng1992symplectic}), we can construct further symplectic resonance-based schemes, for example with the choice:
\begin{align*}
    S&=1,c_1=1, b^{0,1,0}=1,b^{0,1,1}=0,b^{1,1,0}=-1,b^{1,1,1}=1\\
    a^{0,1,0}_{0,1,0}&=\frac{1}{2},a^{0,1,0}_{0,1,1}=1,a^{0,1,0}_{1,1,0}=1,a^{0,1,0}_{1,1,1}=1,a^{1,1,0}_{1,1,0}=-\frac{1}{2},a^{1,1,0}_{1,1,1}=-1,a^{1,1,1}_{1,1,1}=\frac{1}{2}.
    \end{align*}
The method then takes the form
\begin{align*}
    u^{n+1}&=e^{i\tau{\partial_x^2}}u^n+\tau e^{i\tau{\partial_x^2}}K_{0,1,0}-\tau e^{i\tau{\partial_x^2}} K_{1,1,0}+\tau e^{i\tau{\partial_x^2}}K_{1,1,1}\\
    K_{0,1,0}&=\mathcal{F}_{0}^{[NLSE,2]}\left(\tau;1;u^n+\frac{1}{2}\tau K_{0,1,0}\right)\\
    K_{1,1,0}&=\mathcal{F}_{1}^{[NLSE,2]}\left(\tau;1;u^n+\tau K_{0,1,0}-\frac{1}{2}\tau K_{1,1,0}\right)\\
    K_{1,1,1}&=\mathcal{F}_{1}^{[NLSE,2]}\left(\tau;1;u^n+\tau K_{0,1,0}-\tau K_{1,1,0}+\frac{1}{2}\tau K_{1,1,1}\right).
\end{align*}
which leads to a second order scheme for the NLSE with similar convergence properties to the resonance-based midpoint rule (as studied in Theorem~\ref{thm:global_error_midpoint_NLSE}). The same choice of coefficients leads, of course, to a second order low-regularity symplectic scheme for the KdV equation as well.
\end{example}
\section{Convergence analysis}\label{sec:convergence_analysis}
{While there are some very recent approaches to structured error analysis for explicit low-regularity integrators \cite{bruned_schratz_2022}, the implicit case is generally much more challenging due to additional stability considerations in the solution of the implicit equations. In this section, we outline a general recipe which consists of three steps: (i) solubility of the implicit equations; (ii) stability analysis; and (iii) local error estimates, that then combine to a global convergence result. We exhibit this recipe in detail on two specific cases - the resonance-based midpoint rule from Example~\ref{ex:resonance-based_midpoint_rule} for both the KdV equation and the NLSE, before outlining how these ideas can be generalised to higher order RK resonance-based methods in Section~\ref{sec:general_idea_convergence_analysis}.}
\subsection{Convergence analysis of resonance-based midpoint rule in the KdV setting}\label{sec:convergence_analysis_KdV}
In this section we focus on the {error analysis of the resonance-based midpoint rule for the} KdV equation but a similar convergence analysis applies to the NLSE case as well and the main steps and differences in that analysis are given in Section~\ref{sec:convergence_analysis_NLSE}. It will be helpful for the analysis to consider the formulation in the twisted variable, whereby \eqref{eqn:resonance_based_midpoint_rule_u_KdV} becomes
\begin{align}\label{eqn:resonance_based_midpoint_rule_v_KdV}
v^{n+1}=v^{n}+\frac{1}{24}\e^{(t_n+\tau)\partial_x^3}\left(\e^{-(t_n+\tau)\partial_x^3}\partial_x^{-1}\left(v^{n+1}+v^n\right)\right)^2-\frac{1}{24}\e^{t_n\partial_x^3}\left(\e^{-t_n\partial_x^3}\partial_x^{-1}\left(v^{n+1}+v^n\right)\right)^2.
\end{align}

The main results in this section are then the following:
\begin{theorem}\label{thm:global_error_in_H1_first_order}Let us denote by $v(t)$ the exact solution to \eqref{eqn:interation_picture_kdv}, let $v^n, n\geq 0,$ be the iterates in the numerical method \eqref{eqn:resonance_based_midpoint_rule_v_KdV}, and let $t_n=n\tau$. {Given $T>0,R>0$} there is a $\tau_R>0$ such that for all $\tau\in [0,\tau_R)$, and as long as $\sup_{t\in [0,T]}\|v(t)\|_{H^3}<R/2$, we have
	\begin{align*}
		\|v(t_n)-v^{n}\|_{H^1}\leq \tau {c_{R,T}}, \ \forall 0\leq n\leq \left\lfloor \frac{T}{\tau}\right\rfloor,
	\end{align*}
	for some constant $c_{R,T}>0$ depending on {$R,T$, but which may be chosen independently of $\tau$.}
\end{theorem}
\begin{theorem}\label{thm:global_error_in_H1_second_order}
Let $v(t), v^n, n\geq0,$ and $t_n=n\tau$ be as in Theorem~\ref{thm:global_error_in_H1_first_order}. Then we have that given $T>0,R>0$ there is a $\tau_R>0$ such that for all $\tau\in [0,\tau_R)$, and as long as $\sup_{t\in [0,T]}\|v(t)\|_{H^5}<R/2$, we have
\begin{align*}
	\|v(t_n)-v^{n}\|_{H^1}\leq \tau^2{c_{R,T}}, \ \forall 0\leq n\leq \left\lfloor \frac{T}{\tau}\right\rfloor,
\end{align*}
for some constant ${c_{R,T}}>0$ depending on {$R,T$, but which may be chosen independently of $\tau$.}
\end{theorem}
Of course, the isomorphism properties of the twisting map $u(t,x)\mapsto v(t,x)=\exp(t\partial_x^3)u(t,x)$ imply that the analogous results hold true also for the variable $u$ and the method \eqref{eqn:resonance_based_midpoint_rule_u_KdV}:

\begin{cor}\label{cor:convergence_result_res_based_midpoint_KdV_u}
Let $u(t)$ be the exact solution to \eqref{eqn:standard_IVP_KdV} and be $u^n, n\geq 0,$ the iterates in the numerical method \eqref{eqn:resonance_based_midpoint_rule_u_KdV}, and $t_n=n\tau$. Given $T>0,R>0$ there is a $\tau_R>0$ such that for all $\tau\in [0,\tau_R)$:
\begin{itemize}
	\item If $\sup_{t\in [0,T]}\|u(t)\|_{H^3}<R/2$, we have
	\begin{align*}
	\|u(t_n)-u^{n}\|_{H^1}\leq \tau {c_{R,T}}, \ \forall 0\leq n\leq \left\lfloor \frac{T}{\tau}\right\rfloor,
	\end{align*}
	\item if $\sup_{t\in [0,T]}\|u(t)\|_{H^5}<R/2$, we have
\begin{align*}
\|u(t_n)-u^{n}\|_{H^1}\leq \tau^2 {C_{R,T}}, \ \forall 0\leq n\leq \left\lfloor \frac{T}{\tau}\right\rfloor,
\end{align*}
\end{itemize} 
for some constants ${C_{R,T}},{c_{R,T}}$ depending on {$R,T$, but which may be chosen independently of $\tau$.}
\end{cor}
We note in particular that the regularity requirement in the second order estimate in Corollary~\ref{cor:convergence_result_res_based_midpoint_KdV_u} presents a non-trivial improvement over what might classically be expected - indeed had we not incorporated the resonance-structure in our design of the method we should see that second order convergence in $H^1$ can only be achieved for solutions in $H^6$ (as is the case for example with Strang splitting, cf. \cite{holden2011operator,holden2013operator}). 

Parts of our error analysis will follow the ideas in \cite{hofmanova2017exponential}. {However, we need to account for the implicit nature of our method and make use of novel estimates for certain integrals arising from Duhamel's formula to understand the second order convergence properties of this scheme. This particular aspect of the error analysis requires a novel approach, a `mild form' of the classical Gaussian quadrature analysis of the midpoint rule.}

We will present the proof of Theorems~\ref{thm:global_error_in_H1_first_order} \& \ref{thm:global_error_in_H1_second_order} in Section~\ref{sec:error_analysis_in_H1}, as a result of the lemmas introduced and proved in the following sections.

\begin{remark}
	We note that under Assumption~\ref{assumption:zero_mass}, for any $s\in\mathbb{N}$, the norm $\|\,\cdot\,\|_{H^s}$ on the quotient space $H^s/\langle 1\rangle$ is equivalent to $\|\partial_x^s\,\cdot\,\|_{H^0}$ and to the definition in terms of Fourier modes
	\begin{align*}
	\|g\|_{H^s}:=\left(\sum_{m\in\mathbb{Z}\setminus\{0\}}|m|^{2s}|\hat{g}_m|^2\right)^{\frac{1}{2}}.
	\end{align*}
	Therefore we will use these three notations interchangeably throughout the present section.
\end{remark}
\subsubsection{Remarks on the implicit nature of the symplectic resonance-based scheme}\label{sec:properties_of_implicit_equations}
In contrast to classical resonance-based methods \cite{Bronsardbruned2022,bruned_schratz_2022,ostermann2018low,roussetschratz2021,rousset2021general}, which are all explicit, all symplectic schemes presented in the present work are implicit (cf. Corollary~\ref{cor:symplectic_RK_res_are_implicit}).

The implicit nature of the method brings about novel challenges, such as the solution of a nonlinear equation at every time step and the stability analysis of the method. Through rigorous and careful analysis we are able to prove that fixed-point iterations yield a satisfactory means for solving the nonlinear system and derive stability and convergence results of the implicit method \eqref{eqn:resonance_based_midpoint_rule_u_KdV}. A particular strength of this approach is that no CFL condition needs to be imposed on the time-step and spatial discretisation. This results in a method that is truly able to resolve low-regularity solutions, unlike Runge--Kutta methods and {even exponential integrators, the latter of which typically rely on the weaker CFL condition $\Delta t\lesssim \Delta x$ due the Burgers type nonlinearity in the KdV equation. {In particular, to the best of our knowledge \eqref{eqn:resonance_based_midpoint_rule_u_KdV} is, in fact,} the first structure-preserving integrator for the KdV equation which does not require a CFL condition (cf. \cite{ascher2005symplectic,celledoni2008symmetric}).

In this section we show how one may efficiently solve the implicit equation in our scheme \eqref{eqn:resonance_based_midpoint_rule_u_KdV}  with fixed-point iterations at every time step. In practical implementation it is found (cf. Section~\ref{sec:numerical_experiments}) that even for moderate timesteps only a small number of fixed-point iterates is required for convergence. For the analysis let us define the following map:
\begin{align*}
\mathcal{S}_1(\tilde{v}):=v^{n}+\frac{1}{24}\e^{(t_n+\tau)\partial_x^3}\left(\e^{-(t_n+\tau)\partial_x^3}\partial_x^{-1}\left(\tilde{v}+v^n\right)\right)^2-\frac{1}{24}\e^{t_n\partial_x^3}\left(\e^{-t_n\partial_x^3}\partial_x^{-1}\left(\tilde{v}+v^n\right)\right)^2.
\end{align*}
The main result concerning the solution of \eqref{eqn:resonance_based_midpoint_rule_v_KdV} is then the following.
\begin{theorem}\label{thm:fixed_point_iterations_v}
	Fix $l\in \{1,2,3\}$ and $R>0$. Then there is a $\tau_R>0$ such that for all $\tau\in [0,\tau_R)$ and any $v^n\in B_R(H^l):=\{\tilde{v}\in H^l\,\vert\, \|\tilde{v}\|_{H^l}<R\}$ we have $v^{n+1}$ the exact solution of \eqref{eqn:resonance_based_midpoint_rule_v_KdV} is given by the following limit in $H^l$:
	\begin{align}\label{eqn:limit_expression_v^n+1}
	v^{n+1}=\lim_{j\rightarrow\infty}\mathcal{S}_1^{(j)}(v^n),\quad\text{where\ \ }\mathcal{S}_1^{(j)}(v^n)=\underbrace{\mathcal{S}_1\circ\cdots\circ\mathcal{S}_1}_{j-\text{times}}(v^n).
	\end{align}
	{Moreover, we have the estimate
	\begin{align}\label{eqn:distance_next_time_step_from_previous_Hl}
	\left\|v^{n+1}-v^n\right\|_{H^l}\leq \tau^{\frac{1}{2}} \tilde{C}_R,
	\end{align}
	and, if additionally $v^n\in B_{R}(H^{l+1})$,
	\begin{align}\label{eqn:distance_next_time_step_from_previous_Hl+1}
	\left\|v^{n+1}-v^n\right\|_{H^l}\leq \tau \tilde{C}_R,
	\end{align}
	for some $\tilde{C}_R>0$ which depends only on $R$ and $l$.}
\end{theorem}

Note an analogous result holds true for the numerical method in $u$, \eqref{eqn:resonance_based_midpoint_rule_u_KdV}, and is proven in Corollary~\ref{cor:fixed_point_iterations_u}. In order to prove Theorem~\ref{thm:fixed_point_iterations_v} we will rely on the following {two lemmas.}

\begin{lemma}\label{lem:stability_lemma_1} Let us introduce the notation
	\begin{align}\label{eqn:def_of_mathcalF}
	\mathcal{G}(t_n,\tau,\tilde{v}):=\frac{1}{6}\e^{(t_n+\tau)\partial_x^3}\left(\e^{-(t_n+\tau)\partial_x^3}\partial_x^{-1}\tilde{v}\right)^2-\frac{1}{6}\e^{t_n\partial_x^3}\left(\e^{-t_n\partial_x^3}\partial_x^{-1}\tilde{v}\right)^2.
	\end{align}
	Then, for $l=1,2,3$, there is a continuous function $M_l:\mathbb{R}_{\geq0}\times\mathbb{R}_{\geq0}\rightarrow \mathbb{R}_{\geq0}$ such that
	\begin{align*}
	\| \mathcal{G}(t_n,\tau,f)-\mathcal{G}(t_n,\tau,g)\|_{H^l}\leq {\tau^{\frac{1}{2}}} M_l\left(\|f\|_{H^l},\|g\|_{H^l}\right)\|f-g\|_{H^l}.
	\end{align*}
\end{lemma}
\begin{proof}[Proof of Lemma~\ref{lem:stability_lemma_1}]
	For the case $l=1$ see the proof of \cite[Eq.~(38)]{hofmanova2017exponential} and for the case $l=2$ see the proof of \cite[Lemma~2.4]{hofmanova2017exponential}. The case $l=3$ follows in a similar way, and for completeness we have included the proof in Appendix~\ref{app:proof_of_stability_lemma_1}.
\end{proof}
{Finally, we can establish the central contraction mapping property which facilitates the proof of theorem~\ref{thm:fixed_point_iterations_v}.
\begin{lemma}\label{lem:contraction_mapping_property}
    For any $R>0$ there is a constant $\tau_R>0$ such that for all $\tau\in [0,\tau_R)$ the following is true. If $v^n,f,g\in H^l$ are such that $\|f\|_{H^l},\|g\|_{H^l}<2R, \|v^n\|_{H^l}<R$ then 
    \begin{align*}
    \left\|\mathcal{S}_1(f)-\mathcal{S}_1(g)\right\|_{H^l}\leq \frac{1}{2}\|f-g\|_{H^l}.
    \end{align*}
\end{lemma}
\begin{proof}
    By Lemma~\ref{lem:stability_lemma_1} we have
    \begin{align*}
    \left\|\mathcal{S}_1(f)-\mathcal{S}_1(g)\right\|_{H^l}&=\left\|\mathcal{G}\left(\tau,\frac{f+v^n}{2}\right)-\mathcal{G}\left(\tau,\frac{g+v^n}{2}\right)\right\|_{H^l}\\
    &\leq \frac{{\tau^{\frac{1}{2}}}}{2} M_{l}\left(\left\|\frac{v^n+f}{2}\right\|_{H^l},\left\|\frac{v^n+g}{2}\right\|_{H^l}\right)\|f-g\|_{H^l}.
    \end{align*}
    Since $M_l$ is continuous, we have $\tilde{M}_l:=\sup_{|a|,|b|<2R}M_l(a,b)<\infty$ and the claim follows by taking {$\tau_R=\tilde{M}_l^{-2}$}.
\end{proof}
}
\begin{proof}[Proof of Theorem~\ref{thm:fixed_point_iterations_v}]
	Our goal is to apply a contraction mapping argument for $\tau\in [0,\tau_R),$ with $\tau_R>0$ sufficiently small. {Letting again $\tilde{M}_l:=\sup_{|a|,|b|<2R}M_l(a,b)<\infty$ we have, by Lemma~\ref{lem:stability_lemma_1},
\begin{align}\label{eqn:stability_estimate_tau_half}
\|\mathcal{S}_1(v^n)-v^n\|_{H^l}\leq \tau^{\frac{1}{2}} \tilde{M}_l\left\|v^n\right\|_{H^l}.
\end{align}
Thus $\|\mathcal{S}_1(v^n)\|_{H^r}\leq (1+\tau^{1/2} \tilde{M}_l)\left\|v^n\right\|_{H^l}$, and so if we let $\tau_R=\tilde{M}_l^{-2}/4$ we find by induction on $J${, and using Lemma~\ref{lem:contraction_mapping_property},}
\begin{align}\label{eqn:contraction_sum}
\|\mathcal{S}_1^{(J)}(v^n)-v^n\|_{H^l}&\leq \sum_{j=0}^{J-1}\|\mathcal{S}^{(j+1)}(v^n)-\mathcal{S}^{(j)}(v^n)\|_{H^l}\leq\|\mathcal{S}_1(v^n)-v^n\|_{H^l}\sum_{j=0}^{J-1}2^{-j}\leq2\|\mathcal{S}_1(v^n)-v^n\|_{H^l},\\\nonumber
 \text{and}\quad \|\mathcal{S}_1^{(J)}(v^n)\|&<R+2\|\mathcal{S}_1(v^n)-v^n\|_{H^l}<2R,
\end{align}
for all $J\in\mathbb{N}$.} Thus $\left\langle \mathcal{S}_1^{(j)}(v^n)\right\rangle_{j\in\mathbb{N}}$ is a Cauchy sequence and its limit in $H^l$ is a fixed point of $\mathcal{S}_1$, hence \eqref{eqn:limit_expression_v^n+1} follows.

{Taking $J\rightarrow \infty$ in \eqref{eqn:contraction_sum} we find
\begin{align}\label{eqn:generic_estimate_distance_previous_to_current}
    \|v^{n+1}-v^n\|_{H^l}\leq 2\|\mathcal{S}_1(v^n)-v^n\|_{H^l}.
\end{align}
The estimate \eqref{eqn:distance_next_time_step_from_previous_Hl} follows then by combining \eqref{eqn:stability_estimate_tau_half} and \eqref{eqn:generic_estimate_distance_previous_to_current}.} {We now have by the construction of the resonance-based method \eqref{eqn:resonance_based_midpoint_rule_v_KdV}:
\begin{align}\nonumber
    \left\|S_1(v^n)-v^n\right\|_{H^l}&=\left\|\frac{1}{2}\int_{0}^{\tau}\e^{(t_k+s)\partial_x^3}\partial_x\left(\e^{-(t_k+s)\partial_x^3}v^n\right)^2\dd s\right\|_{H^l}\\\label{eqn:stability_estimate_tau_one}
    &\leq \frac{1}{2}\int_{0}^{\tau}\left\|\left(\e^{-(t_k+s)\partial_x^3}v^n\right)^2\right\|_{H^{l+1}}\dd s\leq C_l \tau \|v^n\|_{H^{l+1}}^2,
\end{align}
for some constant $C_l>0$ independent of $v$, where in the final line we made use of the bilinear estimates Lemma~\ref{lem:bilinear_estimates}. Combining \eqref{eqn:stability_estimate_tau_one} and \eqref{eqn:generic_estimate_distance_previous_to_current} implies \eqref{eqn:distance_next_time_step_from_previous_Hl+1}.}
\end{proof}
The results of Theorem~\ref{thm:fixed_point_iterations_v} extend directly to the solution $u^{n+1}$ of \eqref{eqn:resonance_based_midpoint_rule_u_KdV}. For this let us introduce the map
\begin{align*}
\mathcal{S}_2(\tilde{u}):=\e^{-\tau\partial_x^3}u^{n}+\frac{1}{24}\left(\partial_x^{-1}\tilde{u}+\e^{-\tau\partial_x^3}\partial_x^{-1}u^{n}\right)^2-\frac{1}{24}\e^{-\tau\partial_x^3}\left(\e^{\tau \partial_x^3}\partial_x^{-1}\tilde{u}+\partial_x^{-1}u^{n}\right)^2.
\end{align*}
{\begin{cor}\label{cor:fixed_point_iterations_u}
		Fix $l\in \{1,2,3\}$ and $R>0$. Then there is a $\tilde{\tau}_R>0$ such that for all $\tau\in [0,\tilde{\tau}_R)$ and any $u^n\in B_R(H^l):=\{\tilde{u}\in H^l\,\vert\, \|\tilde{u}\|_{H^l}<R\}$ we have $u^{n+1}$ the exact solution of \eqref{eqn:resonance_based_midpoint_rule_u_KdV} is given by the following limit in $H^l$:
		\begin{align}\label{eqn:limit_expression_u^n+1}
		u^{n+1}=\lim_{j\rightarrow\infty}\mathcal{S}_2^{(j)}(u^n),\quad\text{where\ \ }\mathcal{S}_2^{(j)}(u^n)=\underbrace{\mathcal{S}_2\circ\cdots\circ\mathcal{S}_2}_{j-\text{times}}(u^n).
		\end{align}
		Moreover we have the estimate
		\begin{align*}
		\left\|u^{n+1}-u^n\right\|_{H^l}\leq \tau^{\frac{1}{2}} \tilde{C}_R,
		\end{align*}
		and, if additionally $u^n\in B_R(H^{l+1})$,
		\begin{align*}
		\left\|u^{n+1}-u^n\right\|_{H^l}\leq \tau \tilde{C}_R,
		\end{align*}
		for some $\tilde{C}_R>0$ which depends on $R$ and $l$.
\end{cor}}
\begin{proof}
	Let us define
\begin{align*}
\tilde{\mathcal{G}}(t_n,\tau,\tilde{u}):=\e^{-(t_n+\tau)\partial_x^3}\mathcal{G}(t_n,\tau,\e^{t_n\partial_x^3}\tilde{u}).
\end{align*}
Then, by recalling $u^n=\exp(-t_n\partial_x^3)v^n$, we can express $\mathcal{S}_2$ as
\begin{align*}
\mathcal{S}_2(\tilde{u})=\e^{-\tau\partial_x^3}u^n+\tilde{\mathcal{G}}\left(\tau,\frac{\e^{\tau\partial_x^3}\tilde{u}+u^n}{2} \right).
\end{align*}
Now, since $v\mapsto \exp(\pm t\partial_x^3)v$ is an isometry on $H^l$ we have from Lemma~\ref{lem:stability_lemma_1} that
\begin{align}\label{eqn:stability_lemma_for_u_implicit}
\| \tilde{\mathcal{G}}(t_n,\tau,f)-\tilde{\mathcal{G}}(t_n,\tau,g)\|_{H^l}\leq  {\tau^{\frac{1}{2}}} M_l\left(\|f\|_{H^l},\|g\|_{H^l}\right)\|f-g\|_{H^l}.
\end{align}
Hence the result follows by taking exactly the same steps as in the proof of Theorem~\ref{thm:fixed_point_iterations_v} but by replacing the use of Lemma~\ref{lem:stability_lemma_1} with the estimate \eqref{eqn:stability_lemma_for_u_implicit}.
\end{proof}

\subsubsection{Error analysis in $H^2$}\label{sec:error_analysis_in_H2}
As a first step in our proof we need to establish the boundedness of our numerical solution in $H^2$. In this section we will prove convergence and hence boundedness of the numerical solution in $H^2$ for initial data that lies in $H^3$. The stability of our numerical scheme is proved in Lemma~\ref{lem:stability_estimate_implicit_H2}, the local error bound is given in Lemma~\ref{lem:local_error_estimate_H2} and the global error bound is given in Theorem~\ref{thm:global_error_in_H2}.

\paragraph{Stability {in $H^2$}.}
Let us denote by $\Phi_{t_n,\tau}:H^l\rightarrow H^l$ the nonlinear solution map of \eqref{eqn:resonance_based_midpoint_rule_v_KdV}, i.e. let $\Phi_{t_n,\tau}$ be such that 
\begin{align*}
v^{n+1}=\Phi_{t_n,\tau}(v^n).
\end{align*}
We can then show the following stability estimate:
\begin{lemma}\label{lem:stability_estimate_implicit_H2}
	Fix $R>0$. Then there is a $\tau_R>0$ such that for all $\tau\in [0,\tau_R)$ and any $f\in B_R(H^2), g\in B_R(H^3)$ we have
	\begin{align*}
	\|\Phi_{t_n,\tau}(f)-\Phi_{t_n,\tau}(g)\|_{H^2}\leq \exp(\tau \tilde{C}_R)\|f-g\|_{H^2},
	\end{align*}
	where $\tilde{C}_R>0$ depends only on $R$.
\end{lemma}
In order to prove this result we rely on the following bound from \cite{hofmanova2017exponential}:
{\begin{lemma}\label{lem:auxillary_stability_estimate_hofmanova}
Let $\mathcal{F}$ be defined as in \eqref{eqn:def_of_mathcalF}. Then, there is a continuous function $L:\mathbb{R}_{\geq 0}\times \mathbb{R}_{\geq 0}\rightarrow \mathbb{R}_{\geq 0}$ such that for any $f\in H^{2},g\in H^3$ and any $t_n\geq 0$ we have
\begin{align*}
    \left|\left\langle\partial_x^2\left(\mathcal{G}(t_n,\tau,f)-\mathcal{G}(t_n,\tau,g)\right),\partial_x^2(f-g)\right\rangle\right|\leq \tau L(\|f\|_{H^2},\|g\|_{H^3})\|f-g\|_{H^2}^2,
\end{align*}
where by $\langle\,\cdot\,,\,\cdot\,\rangle$ we denoted the usual $L^2$-inner product.
\end{lemma}
\begin{proof}
	See Lemma~2.3 from \cite{hofmanova2017exponential}.
\end{proof}}
We can now proceed to prove the stability estimate Lemma~\ref{lem:stability_estimate_implicit_H2}, {the proof of this stability result is comparable to the stability analysis of the trapezoidal rule, where in our case the boundedness of the nonlinear operators is provided by the estimate in Lemma~\ref{lem:auxillary_stability_estimate_hofmanova}.}
\begin{proof}[Proof of Lemma~\ref{lem:stability_estimate_implicit_H2}] By the definition of the numerical method \eqref{eqn:resonance_based_midpoint_rule_v_KdV} we have for any function $f$:
	\begin{align}\label{eqn:recall_definition_of_numerical_method_in_terms_of_mathcalF}
\Phi_{t_n,\tau}(f)=f+\mathcal{G}\left(t_n,\tau, \frac{f+\Phi_{t_n,\tau}(f)}{2}\right).		
	\end{align}
{Thus, we have
\begin{align*}
\|&\Phi_{t_n,\tau}(f)-\Phi_{t_n,\tau}(g)\|_{H^2}^2\\&= \|f-g\|_{H^2}^2+\underbrace{2\left\langle\partial_x^2\left(\mathcal{G}\left(t_n,\tau, \frac{f+\Phi_{t_n,\tau}(f)}{2}\right)-\mathcal{G}\left(t_n,\tau, \frac{g+\Phi_{t_n,\tau}(g)}{2}\right)\right),\partial_x^2(f-g)\right\rangle}_{=:A} \\
&\quad+\underbrace{\left\|\mathcal{G}\left(t_n,\tau, \frac{f+\Phi_{t_n,\tau}(f)}{2}\right)-\mathcal{G}\left(t_n,\tau, \frac{g+\Phi_{t_n,\tau}(g)}{2}\right)\right\|_{H^2}^2}_{=:B}.\\
\end{align*}
We can then estimate the term $A$ as follows:
\begin{align*}
    |A|&\leq2\left|\left\langle\partial_x^2\left(\mathcal{G}\left(t_n,\tau, \frac{f+\Phi_{t_n,\tau}(f)}{2}\right)-\mathcal{G}\left(t_n,\tau, \frac{g+\Phi_{t_n,\tau}(g)}{2}\right)\right),\partial_x^2\left(\frac{f+\Phi_{t_n,\tau}(f)}{2}-\frac{g+\Phi_{t_n,\tau}(g)}{2}\right)\right\rangle\right|\\
    &\quad +\left\|\partial_x^2\left(\mathcal{G}\left(t_n,\tau, \frac{f+\Phi_{t_n,\tau}(f)}{2}\right)-\mathcal{G}\left(t_n,\tau, \frac{g+\Phi_{t_n,\tau}(g)}{2}\right)\right)\right\|_{L^2}^2\\
    &\leq \tau L\left(\frac{1}{2}\left\|f+\Phi_{t_n,\tau}(f)\right\|_{H^2},\frac{1}{2}\left\|g+\Phi_{t_n,\tau}(g)\right\|_{H^3}\right) \frac{1}{2}\|f-g+\Phi_{t_n,\tau}(f)-\Phi_{t_n,\tau}(g)\|_{H^2}^2+B,
\end{align*}
where we used \eqref{eqn:recall_definition_of_numerical_method_in_terms_of_mathcalF} and Lemma~\ref{lem:auxillary_stability_estimate_hofmanova}. For the term $B$ we have by Lemma~\ref{lem:stability_lemma_1}
\begin{align*}
B\leq \tau M_2\left(\frac{1}{2}\left\|f+\Phi_{t_n,\tau}(f)\right\|_{H^2},\frac{1}{2}\left\|g+\Phi_{t_n,\tau}(g)\right\|_{H^3}\right)^2\frac{1}{2}\|f-g+\Phi_{t_n,\tau}(f)-\Phi_{t_n,\tau}(g)\|_{H^2}^2.
\end{align*}}
Now we have from {\eqref{eqn:distance_next_time_step_from_previous_Hl}} that there is a $\tilde{\tau}_R>0$ such that, for all $\tau\in[0,\tilde{\tau}_R)$,
\begin{align*}
\|\Phi_{t_n,\tau}(f)\|_{H^2}&\leq \|f\|_{H^2}+\|f-\Phi_{t_n,\tau}(f)\|_{H^2}\leq R+R=2R,\\
\|\Phi_{t_n,\tau}(g)\|_{H^3}&\leq \|g\|_{H^3}+\|g-\Phi_{t_n,\tau}(g)\|_{H^3}\leq R+R=2R.
\end{align*}
Therefore, by the continuity of the functions $L,M_2$, there is a constant $\tilde{C}_R>0$ such that for all $\tau\in[0,\tilde{\tau}_R)$
\begin{align*}
\|\Phi_{t_n,\tau}(f)-\Phi_{t_n,\tau}(g)\|_{H^2}^2\leq \|f-g\|_{H^2}^2+\tau \frac{\tilde{C}_R}{2}\left(\|f-g\|_{H^2}^2+\|\Phi_{t_n,\tau}(f)-\Phi_{t_n,\tau}(g)\|_{H^2}^2\right).
\end{align*}
Equivalently for all $\tau\in[0,\tilde{\tau}_R)$
\begin{align*}
\|\Phi_{t_n,\tau}(f)-\Phi_{t_n,\tau}(g)\|_{H^2}^2\leq \frac{1+\frac{\tau}{2}\tilde{C}_R}{1-\frac{\tau}{2}\tilde{C}_R}\|f-g\|_{H^2}^2.
\end{align*}
We now recall that
\begin{align*}
	\frac{1+x/2}{1-x/2}\leq \exp{\left(\frac{3x}{2}\right)},\quad \forall x<1
\end{align*}
Thus the result follows immediately by taking $\tau_R=\min\{\tilde{\tau}_R,1/\tilde{C}_R\}$.
\end{proof}
\paragraph{Local error in $H^2$.}
We can now proceed to estimate the local error of a single time step. We note at this point that a local error of $\mathcal{O}(\tau^{3/2})$ is sufficient to guarantee convergence and hence boundedness of the numerical method in $H^2$. We will describe estimates that provide faster convergence rates in $H^1$ in Section~\ref{sec:error_analysis_in_H1}.
\begin{lemma}\label{lem:local_error_estimate_H2}
	Let us denote by {$t\mapsto\phi_{t_n,t}(z)$ the solution to }
 {\begin{align}\label{eqn:interation_picture_kdv_tn}
     \begin{cases}
		\partial_t\left(\phi_{t_n,t}(z)\right)=\frac{1}{2}\e^{(t_n+t)\partial_x^3}\partial_x\left(\e^{-(t_n+t)\partial_x^3}\phi_{t_n,t}(z)\right)^2, t\in [0,\tau]\\
		\phi_{t_n,0}(z)=z.
	\end{cases}
 \end{align}}
 Fix $R>0$, then there is a $\tau_R>0$ such that for all $\tau\in [0,\tau_R)$ and any {$z\in B_R(H^3)$} such that $\sup_{t\in[0,\tau]}\|\phi_{t_n,t}({z})\|_{H^3}<R$ we have
	\begin{align*}
	\|\phi_{t_n,\tau}({z})-\Phi_{t_n,\tau}({z})\|_{H^2}\leq c_R\tau^{\frac{3}{2}} 
	\end{align*}
	for some constant $c_R>0$ depending only on $R>0$.
\end{lemma}
\begin{proof} The proof of this statement is closely inspired by the proof of Lemma 2.5 in \cite{hofmanova2017exponential}. {However, we need to account for the implicit nature of our scheme by using Theorem~\ref{thm:fixed_point_iterations_v}.} According to Duhamel's formula \eqref{eqn:duhamel_in_fourier_for_v} and the construction of our resonance-based scheme \eqref{eqn:resonance_based_midpoint_rule_v_KdV} we have
	\begin{align}\begin{split}\label{eqn:first_estimate_local_error_H2}
		&\|\phi_{t_n,\tau}({z})-\Phi_{t_n,\tau}({z})\|_{H^2}\\
		&\quad\leq \underbrace{\left\|\frac{1}{2}\int_{0}^\tau \e^{(t_n+s)\partial_x^3}\partial_x\left(\e^{-(t_n+s)\partial_x^3}{\phi_{t_n,s}(z)}\right)^2-\partial_x\left(\e^{-(t_n+s)\partial_x^3}{z}\right)^2\dd s\right\|_{H^2}}_{=:A_1}\\
		&\quad \quad+\underbrace{\left\|\frac{1}{2}\int_{0}^\tau \e^{(t_n+s)\partial_x^3}\partial_x\left(\e^{-(t_n+s)\partial_x^3}{z}\right)^2-\partial_x\left(\e^{-(t_n+s)\partial_x^3}\frac{{z}+\Phi_{t_n,\tau}({z})}{2}\right)^2\dd s\right\|_{H^2}}_{=:A_2}.
		\end{split}
	\end{align}
We will now estimate each term $A_1,A_2$ individually. For the first term, $A_1$ the following bound was shown in \cite[Lemma~2.5]{hofmanova2017exponential} (under the assumption that $\sup_{t\in[0,\tau]}\|\phi_{t_n,\tau}({z})\|_{H^3}<R$):
\begin{align}\label{eqn:second_estimate_local_error_H2}
A_1\leq c_{1,R}\tau^{3/2},
\end{align}
where $c_{1,R}>0$ is a constant depending on $R$. For the second term $A_2$ we note
\begin{align*}
A_2&=\left\|\frac{1}{2}\int_{0}^\tau \e^{(t_n+s)\partial_x^3}\partial_x\left(\e^{-(t_n+s)\partial_x^3}{z}\right)^2-\partial_x\left(\e^{-(t_n+s)\partial_x^3}\frac{{z}+\Phi_{t_n,\tau}({z})}{2}\right)^2\dd s\right\|_{H^2}\\
&\leq \underbrace{\left\|\frac{1}{2}\int_{0}^\tau \e^{(t_n+s)\partial_x^3}\partial_x\left[\left(\e^{-(t_n+s)\partial_x^3}{z}\right)\left(\e^{-(t_n+s)\partial_x^3}({z}-\Phi_{t_n,\tau}({z}))\right)\right]\dd s\right\|_{H^2}}_{A_{2,1}}\\
&\quad+\underbrace{\left\|\frac{1}{2}\int_{0}^\tau \e^{(t_n+s)\partial_x^3}\partial_x\left(\e^{-(t_n+s)\partial_x^3}\frac{{z}-\Phi_{t_n,\tau}({z})}{2}\right)^2\dd s\right\|_{H^2}}_{A_{2,2}}.
\end{align*}
We can again estimate those contributions individually. {Firstly, we have by Theorem~\ref{thm:fixed_point_iterations_v} (specifically \eqref{eqn:distance_next_time_step_from_previous_Hl}) and by the usual bilinear estimate Lemma~\ref{lem:bilinear_estimates} that under the assumptions on $\tau,{z}$,
\begin{align}\label{eqn:third_estimate_local_error_H2}
	A_{2,1}\leq \tilde{c}_{2}\int_{0}^\tau \|{z}\|_{H^3}\|{z}-\Phi_{t_n,\tau}({z})\|_{H^3}\dd s\leq c_{2,R}\tau^{\frac{3}{2}},
\end{align}
where $c_{2,R}>0$ depends on $R$.} Similarly, we have
\begin{align}\label{eqn:fourth_estimate_local_error_H2}
	A_{2,2}\leq \tilde{c}_{3}\int_{0}^{\tau}\|{z}-\Phi_{t_n,\tau}({z})\|_{H^3}^2\dd s\leq {\tau^2} c_{3,R},
\end{align}
where $c_{3,R}>0$ depends on $R$. Combining \eqref{eqn:first_estimate_local_error_H2}-\eqref{eqn:fourth_estimate_local_error_H2} yields the desired estimate.
\end{proof}
\paragraph{Global error in $H^2$.}
We can now combine the estimates from Lemma~\ref{lem:stability_estimate_implicit_H2} and Lemma~\ref{lem:local_error_estimate_H2} to prove the following global error estimate.
\begin{theorem}\label{thm:global_error_in_H2} Let us again denote by $v(t)$ the exact solution to \eqref{eqn:interation_picture_kdv} and be $v^n, n\geq 0,$ the iterates in the numerical method \eqref{eqn:resonance_based_midpoint_rule_v_KdV}, and $t_n=n\tau$. Given $T>0,R>0$ there is a $\tau_R>0$ such that for all $\tau\in [0,\tau_R)$ and as long as $\sup_{t\in [0,T]}\|v(t)\|_{H^3}<R/2$ we have
	\begin{align*}
	\|v(t_n)-v^{n}\|_{H^2}\leq \tau^{\frac{1}{2}}{c_{R,T}}, \ \forall 0\leq n\leq \left\lfloor \frac{T}{\tau}\right\rfloor,
	\end{align*}
	for some constant {$c_{R,T}>0$ depending on $R,T$, but which may be chosen independently of $\tau$.}
\end{theorem}
\begin{proof}
	By the triangle inequality we have
	\begin{align*}
	\|v(t_n)-v^{n}\|_{H^2}\leq \left\|\Phi_{t_{n-1},\tau}(v(t_{n-1}))-\Phi_{t_{n-1},\tau}(v^{n-1})\right\|_{H^2}+\left\|\phi_{t_{n-1},\tau}(v(t_{n-1})-\Phi_{t_{n-1},\tau}(v(t_{n-1}))\right\|_{H^2}.
	\end{align*}
	Iterating the estimate we have, so long as $v^{k}\in B_R(H^2)$ for $0\leq k\leq n$ and $\tau\in [0,\tilde{\tau}_R)$ where $\tilde{\tau}_R$ is as given in Lemmas \ref{lem:stability_estimate_implicit_H2} \& \ref{lem:local_error_estimate_H2}, that
	\begin{align*}
		\|v(t_n)-v^{n}\|_{H^2}&\leq \e^{\tau C_R}\|v(t_{n-1})-v^{n-1}\|_{H^2}+c_R\tau^{\frac{3}{2}}\\
		&\leq \e^{2\tau C_R}\|v(t_{n-2})-v^{n-2}\|_{H^2}+\e^{\tau C_R}c_R\tau^{\frac{3}{2}}+c_R\tau^{\frac{3}{2}}\\
		&\leq c_R\tau^{\frac{3}{2}}\sum_{k=0}^{n-1}\e^{k\tau C_R}\leq c_R\tau^{\frac{1}{2}} t_n \e^{t_n C_R}.
	\end{align*}
	Thus in particular if we choose $\tau<\tau_R$ where $\tau_R=\min\{\tilde{\tau}_R,R^2\exp(-2TC_R)/(4T^2c_R^2)\}$, we ensure that $v^{n}\in B_R(H^2)$ and the result follows by induction.
\end{proof}

\subsection{Error analysis in $H^1$}\label{sec:error_analysis_in_H1}
Having proved the boundedness of our numerical approximation in $H^2$ (see Theorem~\ref{thm:global_error_in_H2}) we can proceed to study its convergence properties in $H^1$.
\paragraph{Stability.}
As in Section~\ref{sec:error_analysis_in_H2} we begin by proving the crucial stability estimate, based on the following estimate from \cite{hofmanova2017exponential}:
{\begin{lemma}\label{lem:auxillary_stability_estimate_H1_hofmanova}
   	Let $\mathcal{F}$ be defined as in \eqref{eqn:def_of_mathcalF}. Then, there is a continuous function $L:\mathbb{R}_{\geq 0}\times \mathbb{R}_{\geq 0}\rightarrow \mathbb{R}_{\geq 0}$ such that for any $f\in H^{2},g\in H^2$ and any $t_n\geq 0$ we have
	\begin{align*}
	\left|\left\langle\partial_x\left(\mathcal{G}(t_n,\tau,f)-\mathcal{G}(t_n,\tau,g)\right),\partial_x(f-g)\right\rangle\right|\leq \tau L(\|f\|_{H^2},\|g\|_{H^2})\|f-g\|_{H^1}.
	\end{align*} 
\end{lemma}}
\begin{proof}
	See Eq.~(36) from \cite{hofmanova2017exponential}.
\end{proof}
The stability estimate for $H^1$ can now be deduced analogously to the proof of Lemma~\ref{lem:stability_estimate_implicit_H2}:
\begin{lemma}\label{lem:stability_in_H1}
	Fix $R>0$. Then there is a $\tau_R>0$ such that for all $\tau\in [0,\tau_R)$ and any $f, g\in B_R(H^2)$ we have
	\begin{align*}
	\|\Phi_{t_n,\tau}(f)-\Phi_{t_n,\tau}(g)\|_{H^1}\leq \exp(\tau \tilde{C}_R)\|f-g\|_{H^1},
	\end{align*}
	where $\tilde{C}_R>0$ depends on $R$.
\end{lemma}
\begin{proof}
	This statement can be proved analogously to Lemma~\ref{lem:stability_estimate_implicit_H2} at each point replacing the use of Lemma~\ref{lem:auxillary_stability_estimate_hofmanova} by the $H^1$-estimate Lemma~\ref{lem:auxillary_stability_estimate_H1_hofmanova}. In the interest of brevity the arguments are not repeated here.	
\end{proof}
\paragraph{Local error in $H^1$.} We may now proceed to prove the crucial local error estimates on the numerical scheme in $H^1$. In the following we will show two central results: in Lemma~\ref{lem:local_error_estimate_H1_first_order} we show that the method incurs a local error of size $\mathcal{O}(\tau^2)$ provided the solution $v(t)$ remains uniformly bounded in $H^3$ over the time-interval of interest $[t_n,t_n+\tau]$. {However, like the classical midpoint rule, our present method actually exhibits a perhaps surprising improved convergence property: in Lemma~\ref{lem:local_error_estimate_H1_second_order} we demonstrate the method incurs a local error of size $\mathcal{O}(\tau^3)$ provided the solution remains uniformly bounded in $H^5$ over the corresponding time-interval.}
\begin{lemma}\label{lem:local_error_estimate_H1_first_order} As {in Lemma~\ref{lem:local_error_estimate_H2},} let us denote by $\phi_{t_n,\tau}({z})$ {the solution to \eqref{eqn:interation_picture_kdv_tn}.} Fix $R>0$, then there is a $\tau_R>0$ such that for all $\tau\in [0,\tau_R)$ and any ${z}\in B_R(H^3)$ such that $\sup_{t\in[0,\tau]}\|\phi_{t_n,t}({z})\|_{H^3}<R$ we have
	\begin{align*}
		\|\phi_{t_n,\tau}({z})-\Phi_{t_n,\tau}({z})\|_{H^1}\leq c_R\tau^{2} 
	\end{align*}
	for some constant $c_R>0$ depending on $R>0$.
\end{lemma}

\begin{proof}By construction of our numerical scheme \eqref{eqn:resonance_based_midpoint_rule_v_KdV} we have
	\begin{align*}
		\phi^\tau({z})-&\Phi_{t_n}^\tau({z})\\&=\frac{1}{2}\int_{0}^{\tau}\e^{(t_n+s)\partial_x^3}\partial_x\left[\left(\e^{-(t_n+s)\partial_x^3}{\phi_{t_n,s}(z)}\right)^2-\left(\e^{-(t_n+s)\partial_x^3}\frac{1}{2}\left({z}+\Phi_{t_n}^\tau({z})\right)\right)^2\right]\dd s.
	\end{align*}
Thus, using the usual bilinear estimate Lemma~\ref{lem:bilinear_estimates}, we find for some constant $c>0$
\begin{align}\nonumber
	&\left\|\phi^\tau({z})-\Phi_{t_n,\tau}({z})\right\|_{H^1}\\\nonumber
	&
	\quad\leq c\frac{1}{2}\int_0^\tau \left\|{\phi_{t_n,s}(z)}-\frac{1}{2}\left({z}+\Phi_{t_n,\tau}({z})\right)\right\|_{H^2}\left(\|{\phi_{t_n,s}(z)}\|_{H^2}+\frac{1}{2}\left\|{z}+\Phi_{t_n,\tau}({z})\right\|_{H^2}\right)\dd s\\\begin{split}\label{eqn:initial_estimate_H1_local_error_first_order}
	&\quad\leq c\frac{\tau}{2}\sup_{s\in[0,\tau]} \left\|{\phi_{t_n,s}(z)}-\frac{1}{2}\left({z}+\Phi_{t_n,\tau}({z})\right)\right\|_{H^2}\\
	&\hspace{6.5cm}\left(\sup_{s\in[0,\tau]}\|{\phi_{t_n,s}(z)}\|_{H^2}+\frac{1}{2}\left\|{z}+\Phi_{t_n,\tau}({z})\right\|_{H^2}\right)\
 \end{split}
\end{align}
Note that by Theorem~\ref{thm:fixed_point_iterations_v} {(specifically \eqref{eqn:distance_next_time_step_from_previous_Hl+1})} we have for some $\tilde{\tau}_R>0$ and all $\tau\in [0,\tilde{\tau}_R)$
\begin{align}
\frac{1}{2}\left\|{z}+\Phi_{t_n,\tau}({z})\right\|_{H^2}\leq R+\tau\tilde{C}_R\leq \tilde{\tilde{C}}_R,
\end{align}
for some $\tilde{C}_R,\tilde{\tilde{C}}_R>0$ which depend on $R$. It remains to bound the term 
\begin{align*}
	\left\|{\phi_{t_n,s}(z)}-\frac{1}{2}\left({z}+\Phi_{t_n,\tau}({z})\right)\right\|_{H^2}.
\end{align*}
This can be done as follows:
\begin{align}\nonumber
	\left\|{\phi_{t_n,s}(z)}-\frac{1}{2}\left({z}+\Phi_{t_n,\tau}({z})\right)\right\|_{H^2}&\leq 
	\left\|{\phi_{t_n,s}(z)}-{\phi_{t_n,\tau}(z)}\right\|_{H^2}\\\nonumber
	&\quad+
	\left\|{\phi_{t_n,\tau}(z)}-\Phi_{t_n,\tau}({z})\right\|_{H^2}\\&\quad\quad+
	\frac{1}{2}\left\|{z}-\Phi_{t_n,\tau}({z})\right\|_{H^2}.
\end{align}
We then have by Lemma~\ref{lem:local_error_estimate_H2} for some $\tilde{\tilde{\tau}}_R>0$ and all $\tau\in [0,\tilde{\tilde{\tau}}_R)$
\begin{align}
\left\|{\phi_{t_n,\tau}(z)}-\Phi_{t_n,\tau}({z})\right\|_{H^2}\leq c_R\tau^{\frac{3}{2}},
\end{align}
and, by Theorem~\ref{thm:fixed_point_iterations_v},
\begin{align}
\frac{1}{2}\left\|{z}-\Phi_{t_n,\tau}({z})\right\|_{H^2}\leq \tau \tilde{C}_R.
\end{align}
Finally, we can estimate using Duhamel's formula \eqref{eqn:Duhamel_for_v} and the bilinear estimate Lemma~\ref{lem:bilinear_estimates}:
\begin{align}\nonumber
\left\|{\phi_{t_n,s}(z)}-{\phi_{t_n,\tau}(z)}\right\|_{H^2}&=\left\|\frac{1}{2}\int_{s}^\tau \e^{(t_n+\tilde{s})\partial_x^3}\partial_x\left(\e^{-(t_n+\tilde{s})\partial_x^3}{\phi_{t_n,\tilde{s}}(z)}\right)^2\dd\tilde{s}\right\|_{H^2}\\\label{eqn:final_estimate_H1_local_error_first_order}
&\leq c \frac{1}{2}\int_{s}^\tau \left\|{\phi_{t_n,\tilde{s}}(z)}\right\|_{H^3}^2 \dd\tilde{s}\leq \tau \tilde{c}_R
\end{align}
for some constant $\tilde{c}_R$. Thus if we choose $\tau_R=\min\{\tilde{\tau}_R,\tilde{\tilde{\tau}}_R,R/\tilde{C}_R\}$ and combine \eqref{eqn:initial_estimate_H1_local_error_first_order}-\eqref{eqn:final_estimate_H1_local_error_first_order}, the result follows.
\end{proof}

Although Lemma~\ref{lem:local_error_estimate_H1_first_order} is already sufficient to guarantee convergence of our method in $H^1$ we can show that faster rates of $H^1$-convergence can be obtained if we allow for slightly more regular initial data. {This particular approach to the local error analysis, which is based on iterating Duhamel's formula around the midpoint value $v(t_n+\tau/2)$ instead of the usual left endpoint $v(t_n)$, is a completely new and different from ideas that were previously used for the convergence analysis of low-regularity integrators.}

\begin{lemma}\label{lem:local_error_estimate_H1_second_order} Let {$\phi_{t_n,\tau}(z)$} be as above. Fix $R>0$, then there is a $\tau_R>0$ such that for all $\tau\in [0,\tau_R)$ and any ${z}\in B_R(H^5)$ such that $\sup_{t\in[0,\tau]}\|\phi_{t_n,\tau}({z})\|_{H^5}<R$ we have
\begin{align*}
	\|\phi_{t_n,\tau}({z})-\Phi_{t_n,\tau}({z})\|_{H^1}\leq c_R\tau^{3} 
\end{align*}
for some constant $c_R>0$ depending on $R>0$.
\end{lemma}
In order to prove this statement we have to rely on the following crucial estimate:
\begin{lemma}\label{lem:auxilliary_convolution_estimate}
	For any $j,l\in\mathbb{N}$ such that $j+l\geq 1$ there is a constant $c>0$ such that for all $f,g\in H^{j+l}$ and any $F\in L^2$ whose Fourier coefficients satisfy
\begin{align*}
\hat{F}_m\leq \sum_{m=a+b}|m|^l |\hat{f}_a||\hat{g}_b|, \quad \forall m\in\mathbb{Z},
\end{align*}
we have
\begin{align*}
\|F\|_{H^j}\leq c \|f\|_{H^{j+l}}\|g\|_{H^{j+l}}.
\end{align*}
\end{lemma}
\begin{proof}
The proof is given in Appendix~\ref{app:proof_of_auxilliary_convolution_estimate}.
\end{proof}
\begin{proof}[Proof of Lemma~\ref{lem:local_error_estimate_H1_second_order}] Our starting point is again Duhamel's formula \eqref{eqn:Duhamel_for_v} and the definition of our numerical scheme \eqref{eqn:resonance_based_midpoint_rule_v_KdV} which yields:
	\begin{align}\begin{split}\label{eqn:initial_estimate_local_error_H1_second_order}
\phi_{t_n,\tau}&({z})-\Phi_{t_n,\tau}({z})\\
&=\frac{1}{2}\int_{0}^{\tau}\e^{(t_n+s)\partial_x^3}\partial_x\left[\left(\e^{-(t_n+s)\partial_x^3}{\phi_{t_n,s}(z)}\right)^2-\left(\e^{-(t_n+s)\partial_x^3}\frac{1}{2}\left({z}+\Phi_{t_n,\tau}({z})\right)\right)^2\right]\dd s.
\end{split}
	\end{align}
In order to understand higher order convergence properties of the numerical method we have to iterate \eqref{eqn:Duhamel_for_v} to obtain the following expression:
	\begin{align*}
	&\hspace{-1.5cm}\left(\e^{-(t_n+s)\partial_x^3}{\phi_{t_n,s}(z)}\right)^2\\
 &=\left(\e^{-(t_n+s)\partial_x^3}{\phi_{t_n,\tau/2}(z)}+\frac{1}{2}\e^{-(t_n+s)\partial_x^3}\int_{\tau/2}^{s}\e^{(t_n+\tilde{s})\partial_x^3}\partial_x\left(\e^{-(t_n+\tilde{s})\partial_x^3}{\phi_{t_n,\tilde{s}}(z)}\right)^2\dd \tilde{s}\right)^2\\
	&=\left(\e^{-(t_n+s)\partial_x^3}{\phi_{t_n,\tau/2}(z)}\right)^2\\
	&\quad+\underbrace{\left(\e^{-(t_n+s)\partial_x^3}{\phi_{t_n,\tau/2}(z)}\right)\e^{-(t_n+s)\partial_x^3}\int_{\tau/2}^{s}\e^{(t_n+\tilde{s})\partial_x^3}\partial_x\left(\e^{-(t_n+\tilde{s})\partial_x^3}{\phi_{t_n,\tilde{s}}(z)}\right)^2\dd \tilde{s}}_{=:A_1}\\
	&\quad\quad+\underbrace{\frac{1}{4}\left(\e^{-(t_n+s)\partial_x^3}\int_{\tau/2}^{s}\e^{(t_n+\tilde{s})\partial_x^3}\partial_x\left(\e^{-(t_n+\tilde{s})\partial_x^3}{\phi_{t_n,\tilde{s}}(z)}\right)^2\dd \tilde{s}\right)^2}_{=:A_2}.
	\end{align*}
We also have
\begin{align*}
\frac{1}{2}\left({z}+\Phi_{t_n,\tau}({z})\right)-\phi_{t_n,\tau/2}({z})
&=\frac{1}{4}\int_{0}^{\tau}\e^{(t_n+s)\partial_x^3}\partial_x\left[\left(\e^{-(t_n+s)\partial_x^3}\frac{1}{2}\left({z}+\Phi_{t_n,\tau}({z})\right)\right)^2\right]\dd s\\
&\quad\quad -\frac{1}{2}\int_{0}^{\tau/2}\e^{(t_n+s)\partial_x^3}\partial_x\left[\left(\e^{-(t_n+s)\partial_x^3}{\phi_{t_n,s}(z)}\right)^2\right]\dd s\\
&=:D_1
\end{align*}
Thus
\begin{align}\nonumber
&\left(\e^{-(t_n+s)\partial_x^3}\frac{1}{2}\left({z}+\Phi_{t_n,\tau}({z})\right)\right)^2\\\nonumber
&\quad\quad=\left(\e^{-(t_n+s)\partial_x^3}{\phi_{t_n,\tau/2}(z)}+\e^{-(t_n+s)\partial_x^3}\left(\frac{1}{2}\left({z}+\Phi_{t_n,\tau}({z})\right)-{\phi_{t_n,\tau/2}(z)}\right)\right)^2\\\label{eqn:definition_of_B_1}
&\quad\quad=\left(\e^{-(t_n+s)\partial_x^3}{\phi_{t_n,\tau/2}(z)}\right)^2+\underbrace{\left(\e^{-(t_n+s)\partial_x^3}{\phi_{t_n,\tau/2}(z)}\right)\left(\e^{-(t_n+s)\partial_x^3}D_1\right)}_{=:B_1}+\left(\e^{-(t_n+s)\partial_x^3}D_1\right)^2{.}
\end{align}
Thus we have from \eqref{eqn:initial_estimate_local_error_H1_second_order}
\begin{align}\begin{split}\label{eqn:local_error_H1_second_order_combined_estimate}
\left\|\phi_{t_n,\tau}({z})-\Phi_{t_n,\tau}({z})\right\|_{H^1}&\leq c\left\|\int_{0}^{\tau}\e^{(t_n+s)\partial_x^3}\partial_xA_1\dd s\right\|_{H^1}\\
&\quad+ c\int_{0}^{\tau}\left\|A_2\right\|_{H^2}+\left\|B_1\right\|_{H^2}+\left\|\left(\e^{-(t_n+s)\partial_x^3}D_1\right)^2\right\|_{H^2}\dd s.\end{split}
\end{align}
Let us begin by estimating the contributions from $A_2$ and from $\left(\e^{-(t_n+s)\partial_x^3}D_1\right)^2$. For $A_2$ we have
\begin{align}\nonumber
\int_0^\tau \left\|A_2\right\|_{H^2}\dd s&\leq c\int_0^\tau \left\|\int_{\tau/2}^{s}\e^{(t_n+\tilde{s})\partial_x^3}\partial_x\left(\e^{-(t_n+\tilde{s})\partial_x^3}{\phi_{t_n,\tilde{s}}(z)}\right)^2\dd \tilde{s}\right\|_{H^2}^2\dd s \\\label{eqn:estimate_A_2_for_local_error_H1_second_order}
&\leq \tilde{c} \int_0^\tau \left(\int_{\tau/2}^{s}\dd \tilde{s}\right)^2\dd s \sup_{s\in[0,\tau]}\|{\phi_{t_n,s}(z)}\|_{H^3}^4\leq \tilde{c}\tau^3\sup_{s\in[0,\tau]}\|{\phi_{t_n,s}(z)}\|_{H^3}^4{,}
\end{align}
for some constants $c,\tilde{c}>0$ independent of $v$. Similarly we find (for potentially different values of $c,\tilde{c}>0$) using Theorem~\ref{thm:fixed_point_iterations_v} under the assumption that $\tau\in[0,\tau_R)$ as given in the statement of the theorem,
\begin{align*}
\int_0^\tau \left\|\left(\e^{-(t_n+s)\partial_x^3}D_1\right)^2\right\|_{H^2}\dd s&\leq c\int_0^\tau \left\|\int_{0}^{\tau}\e^{(t_n+s)\partial_x^3}\partial_x\left(\e^{-(t_n+s)\partial_x^3}\frac{1}{2}\left({z}+\Phi_{t_n,\tau}({z})\right)\right)^2\dd s\right\|_{H^2}^2\dd s\\
&\quad+c\int_0^\tau\left\|\int_{0}^{\tau/2}\e^{(t_n+s)\partial_x^3}\partial_x\left(\e^{-(t_n+s)\partial_x^3}{\phi_{t_n,s}(z)}\right)^2\dd s\right\|_{H^2}^2\dd s\\
&\leq \tilde{c}\tau^3\left(\sup_{s\in[0,\tau]}\left\|{\phi_{t_n,s}(z)}\right\|_{H^3}^4+\left\|\Phi_{t_n,\tau}({z})-{z}\right\|_{H^3}^4\right)\\
&\leq C_R\tau^3 (1+\tau),
\end{align*}
where $C_R>0$ depends on $R$. Since $\tau<\tau_R$ we thus have the estimate
\begin{align}\label{eqn:estimate_D_1_square_for_local_error_H1_second_order}
\int_0^\tau \left\|\left(\e^{-(t_n+s)\partial_x^3}D_1\right)^2\right\|_{H^2}\dd s&\leq\tau^3 C_R.
\end{align}
Now we aim to estimate the contribution from $B_1$:
\begin{align*}
\int_0^\tau \left\|B_1\right\|_{H^2}\dd s.
\end{align*}
To achieve a suitable estimate let us express $D_1$ in the following way
\begin{align}\begin{split}\label{eqn:fine_estimate_D_1}
D_1&=\underbrace{\frac{1}{4}\int_{0}^{\tau}\e^{(t_n+s)\partial_x^3}\partial_x\left[\left(\e^{-(t_n+s)\partial_x^3}{z}\right)^2\right]\dd s -\frac{1}{2}\int_{0}^{\tau/2}\e^{(t_n+s)\partial_x^3}\partial_x\left[\left(\e^{-(t_n+s)\partial_x^3}{z}\right)^2\right]\dd s}_{=:D_{1,1}}\\
&\quad +\underbrace{\frac{1}{4}\int_{0}^{\tau}\e^{(t_n+s)\partial_x^3}\partial_x\left[\left(\e^{-(t_n+s)\partial_x^3}\frac{1}{2}\left({z}+\Phi_{t_n,\tau}({z})\right)\right)^2-\left(\e^{-(t_n+s)\partial_x^3}{z}\right)^2\right]\dd s}_{=:D_{1,2}} \\
&\quad\quad-\underbrace{\frac{1}{2}\int_{0}^{\tau/2}\e^{(t_n+s)\partial_x^3}\partial_x\left[\left(\e^{-(t_n+s)\partial_x^3}{z}\right)^2-\left(\e^{-(t_n+s)\partial_x^3}{\phi_{t_n,s}(z)}\right)^2\right]\dd s}_{=:D_{1,3}}.
\end{split}
\end{align}
We have furthermore the following expression in terms of Fourier coefficients:
\begin{align*}
D_{1,1}&=\frac{1}{4}\sum_{m\in\mathbb{Z}}\e^{imx}\sum_{m=a+b}\e^{-it_n(m^3-a^3-b^3)}im \hat{v}_a\hat{v}_b\left(\int_{\tau/2}^\tau \e^{-is(m^3-a^3-b^3)}\dd s-\int_{0}^{\tau/2} \e^{-is(m^3-a^3-b^3)}\dd s\right)\\
&=\frac{1}{4}\sum_{m\in\mathbb{Z}}\e^{imx}\sum_{m=a+b}\e^{-it_n(m^3-a^3-b^3)}im \hat{v}_a\hat{v}_b\frac{(\e^{-i\frac{\tau mab}{2}}-1)^2}{-imab}.
\end{align*}
Noting that $|\exp(ix)-1|/|x|\leq 1$ for $x\in \mathbb{R}$ we therefore have that
\begin{align*}
\left|\widehat{\partial_x^2 D_{1,1}}_m\right|\leq \tau^2\sum_{m=a+b}|m|^4|a||b|\hat{v}_a||\hat{v}_b|.
\end{align*}
Thus, by Lemma~\ref{lem:auxilliary_convolution_estimate}, it immediately follows that 
\begin{align*}
\|D_{1,1}\|_{H^2}\leq c\tau^2 \left\|v\right\|_{H^5}^2,
\end{align*}
for some constant $c>0$, independent of $R$. We also have, using \eqref{eqn:Duhamel_for_v},
\begin{align*}
\|D_{1,3}\|_{H^2}&\leq \tilde{c}\int_0^{\tau/2}\left\|{z}+{\phi_{t_n,s}(z)}\right\|_{H^3}\left\|{z}-{\phi_{t_n,s}(z)}\right\|_{H^3}\dd s\leq \tau^2 C_R,
\end{align*}
where $C_R>0$ depends on $R$, and a similar estimate can be derived analogously for $D_{1,2}$ for any $\tau\in[0,\tau_R)$ as defined in the assumptions of Theorem~\ref{thm:fixed_point_iterations_v}:
\begin{align*}
\|D_{1,2}\|_{H^2}&\leq \tau^2 C_R.
\end{align*}
From \eqref{eqn:fine_estimate_D_1}\&\eqref{eqn:definition_of_B_1} we thus have for any $\tau\in[0,\tau_R)$ and some $\tilde{C}_R>0$ depending only on $R$
\begin{align}\label{eqn:estimate_of_B_1_local_error_H1_second_order}
\int_0^\tau \left\|B_1\right\|_{H^2}\dd s&\leq \tau^3 \tilde{C}_R.
\end{align}
It remains to estimate the contribution from $A_1$. For this it is helpful to write {$\phi'_{t_n,t}(z):=\partial_t \phi_{t_n,t}(z)$ and to observe that}
\begin{align*}
A_1&=\left(\e^{-(t_n+s)\partial_x^3}{\phi_{t_n,\tau/2}(z)}\right)\e^{-(t_n+s)\partial_x^3}\int_{\tau/2}^{s}{\phi'_{t_n,\tilde{s}}(z)}\dd \tilde{s}\\
&=\underbrace{\left(\e^{-(t_n+s)\partial_x^3}{\phi_{t_n,\tau/2}(z)}\right)\e^{-(t_n+s)\partial_x^3}(s-\tau/2){\phi'_{t_n,\tau/2}(z)}}_{=:A_{1,1}}\\
&\quad+\underbrace{\left(\e^{-(t_n+s)\partial_x^3}{\phi_{t_n,\tau/2}(z)}\right)\e^{-(t_n+s)\partial_x^3}\int_{\tau/2}^{s}\left({\phi'_{t_n,\tilde{s}}(z)}-{\phi'_{t_n,\tau/2}(z)}\right)\dd \tilde{s}}_{=:A_{1,2}}.
\end{align*}
Now we observe that
\begin{align}\label{eqn:estimate_A_1_local_error_H1_second_order}
\left\|\int_{0}^{\tau}\e^{(t_n+s)\partial_x^3}\partial_xA_1\dd s\right\|_{H^1}\leq \left\|\int_{0}^{\tau}\e^{(t_n+s)\partial_x^3}\partial_xA_{1,1}\dd s\right\|_{H^1}+\int_{0}^{\tau}\|A_{1,2}\|_{H^2}\dd s.
\end{align}
And we have by Lemma~\ref{lem:bilinear_estimates}
\begin{align*}
\int_0^\tau \|A_{1,2}\|_{H^2}\dd s\leq c \tau ^2 \sup_{s\in[0,\tau]}\|{\phi_{t_n,s}(z)}\|_{H^2}\sup_{s\in[0,\tau]}\|{\phi'_{t_n,s}(z)}-{\phi'_{t_n,\tau/2}(z)}\|_{H^2}.
\end{align*}
Now by \eqref{eqn:interation_picture_kdv} we have in terms of Fourier coefficients{, writing as a shorthand notation $\hat{v}_a(t_n+s),\hat{v}_{a}'(t_n+s)$ for the $a^{th}$ Fourier coefficient of $\phi_{t_n,s}(z),\phi'_{t_n,s}(z)$ respectively,}
\begin{align*}
{\phi'_{t_n,s}(z)}-{\phi'_{t_n,\tau/2}(z)}&=\int_{\tau/2}^{s}{\partial_{t}^2\phi_{t_n,t}(z)\vert_{t=\tilde{s}}}\dd\tilde{s}\\
&=\sum_{m\in\mathbb{Z}}\e^{imx}\sum_{m=a+b}\int_{\tau/2}^{s}\frac{1}{2}\e^{-i(t_n+\tilde{s})3abm}(-3iabm)im\hat{v}_a(t_n+\tilde{s})\hat{v}_b(t_n+\tilde{s})\dd\tilde{s}\\
&\quad+\sum_{m\in\mathbb{Z}}\e^{imx}\sum_{m=a+b}\int_{\tau/2}^{s}\e^{-i(t_n+\tilde{s})3abm}im\widehat{v'}_a(t_n+\tilde{s})\hat{v}_b(t_n+\tilde{s})\dd\tilde{s}\\
&=\frac{1}{2}\int_{\tau/2}^s\e^{(t_n+\tilde{s})\partial_x^3}\partial_x^2\left(\e^{-(t_n+\tilde{s})\partial_x^3}\partial_x{\phi_{t_n,\tilde{s}}(z)}\right)^2\dd \tilde{s}\\
&\quad+\int_{\tau/2}^s \e^{(t_n+\tilde{s})\partial_x^3}\partial_x\left[\left(\e^{-(t_n+\tilde{s})\partial_x^3}{\phi_{t_n,\tilde{s}}(z)}\right)\left(\e^{-(t_n+\tilde{s})\partial_x^3}{\phi'_{t_n,\tilde{s}}(z)}\right)\right]\dd \tilde{s}\\
&=\frac{1}{2}\int_{\tau/2}^s\e^{(t_n+\tilde{s})\partial_x^3}\partial_x^2\left(\e^{-(t_n+\tilde{s})\partial_x^3}\partial_x{\phi_{t_n,\tilde{s}}(z)}\right)^2\dd \tilde{s}\\
&\quad+\frac{1}{2}\int_{\tau/2}^s\e^{(t_n+\tilde{s})\partial_x^3}\partial_x\left[\left(\e^{-(t_n+\tilde{s})\partial_x^3}{\phi_{t_n,\tilde{s}}(z)}\right)\partial_x\left(\e^{-(t_n+\tilde{s})\partial_x^3}{\phi_{t_n,\tilde{s}}(z)}\right)^2\right]\dd \tilde{s}.
\end{align*}
Thus, similarly to the derivation of \eqref{eqn:estimate_A_2_for_local_error_H1_second_order}, we have the estimate
\begin{align}\label{eqn:estimate_A_12}
\int_{0}^{\tau}\|A_{1,2}\|_{H^2}\dd s\leq c\tau^3\left( \sup_{s\in[0,\tau]}\|{\phi_{t_n,s}(z)}\|_{H^5}^3+\sup_{s\in[0,\tau]}\|{\phi_{t_n,s}(z)}\|_{H^5}^2\right).
\end{align}
To estimate $A_{1,1}$ let us look at the Fourier coefficients of $A:=\int_{0}^{\tau}\e^{(t_n+s)\partial_x^3}\partial_xA_{1,1}\dd s$. We have
\begin{align*}
|\hat{A}_m|&=\left|\sum_{m=a+b} im \int_0^\tau \e^{-i(m^3-a^3-b^3)(t_n+s)}(s-\tau/2)\dd s \hat{v}_a(t_n+\tau/2)\hat{v}_b'(t_n+\tau/2)\right|\\
&=\left|\sum_{m=a+b} im \int_0^\tau \e^{-i(m^3-a^3-b^3)(t_n+s)}(s-\tau/2)\dd s \hat{v}_a(t_n+\tau/2)\hat{v}_b'(t_n+\tau/2)\right|\\
&\leq \sum_{m=a+b}\left|\frac{3mab\tau \cos\left(\frac{3mab\tau}{2}\right)-2\sin\left(\frac{3mab\tau}{2}\right)}{(3mab)^2}\right||m|\left|\hat{v}_a(t_n+\tau/2)\hat{v}_b'(t_n+\tau/2)\right|.
\end{align*}
Using Lemma~\ref{lem:auxilliary_convolution_estimate} and observing that $\left|\frac{x\cos\left(x/2\right)-2\sin(x/2)}{x^2}\right|\leq \frac{|x|}{12}$ for all $x\in\mathbb{R}$, we have
\begin{align}\nonumber
\left\|\int_0^\tau \e^{(t_n+s)\partial_x^3}\partial_x A_{1,1}\dd s\right\|_{H^1}=\|A\|_{H^1}&\leq c\tau^3 \left\|{\phi_{t_n,\tau/2}(z)}\right\|_{H^4}\left\|{\phi'_{t_n,\tau/2}(z)}\right\|_{H^4}\\\label{eqn:estimate_A_11}
&\leq \tilde{c}\tau^3\left\|{\phi_{t_n,\tau/2}(z)}\right\|_{H^4}\left\|{\phi_{t_n,\tau/2}(z)}\right\|_{H^5}^2
\end{align}
for some constants $c,\tilde{c}>0$.

We can now easily conclude the result by combining the estimates \eqref{eqn:local_error_H1_second_order_combined_estimate}, \eqref{eqn:estimate_A_2_for_local_error_H1_second_order}, \eqref{eqn:estimate_D_1_square_for_local_error_H1_second_order}, \eqref{eqn:estimate_of_B_1_local_error_H1_second_order}, \eqref{eqn:estimate_A_1_local_error_H1_second_order}, \eqref{eqn:estimate_A_12} \& \eqref{eqn:estimate_A_11}.
\end{proof}

\begin{remark}
	A slightly more detailed (and tedious) analysis of the error in terms of Fourier coefficients would allow us to show the following slightly more resolved result which implies Lemmas~\ref{lem:local_error_estimate_H1_first_order}\&\ref{lem:local_error_estimate_H1_second_order}: Fix $R>0$, then for any $\gamma\in[1,2]$ there is a $\tau_{R,\gamma},c_{R,
	\gamma}>0$ such that for all $\tau\in [0,\tau_{R,\gamma})$ and any ${z}\in B_R(H^{1+2\gamma})$ such that $\sup_{t\in[0,\tau]}\|\phi_{t_n,\tau}({z})\|_{H^{1+2\gamma}}<R$ we have
	\begin{align*}
	\|\phi_{t_n,\tau}({z})-\Phi_{t_n,\tau}({z})\|_{H^1}\leq c_{R,
	\gamma}\tau^{1+{\gamma}}.
	\end{align*}
\end{remark}

\paragraph{Global error in $H^1$.}
We can now prove Theorems~\ref{thm:global_error_in_H1_first_order}\&\ref{thm:global_error_in_H1_second_order} in similar vein to our proof of Theorem~\ref{thm:global_error_in_H2}.
\begin{proof}[Proof of Theorem~\ref{thm:global_error_in_H1_first_order}]
		By the triangle inequality we have
	\begin{align}\label{eqn:global_error_proof_general_estimate}
	\|v(t_n)-v^{n}\|_{H^1}\leq \left\|\Phi_{t_{n-1},\tau}(v(t_{n-1}))-\Phi_{t_{n-1},\tau}(v^{n-1})\right\|_{H^1}+\left\|\phi_{t_{n-1},\tau}(v(t_{n-1}){)}-\Phi_{t_{n-1},\tau}(v(t_{n-1}))\right\|_{H^1}.
	\end{align}
	We note that by Theorem~\ref{thm:global_error_in_H2} we may choose $\tau_R$ such that, whenever $\tau\in[0,\tau_R)$ we have $v^{k}\in B_R(H^2)$ for all $0\leq k\leq n$. Thus we may iterate above estimate \eqref{eqn:global_error_proof_general_estimate} and find (so long as $\tau_R$ is smaller than the constants defined in Lemmas~\ref{lem:stability_in_H1} \& \ref{lem:local_error_estimate_H1_first_order} and $\tau\in [0,\tau_R)$) that, using Lemmas \ref{lem:stability_in_H1} \& \ref{lem:local_error_estimate_H1_first_order},
\begin{align*}
\|v(t_n)-v^{n}\|_{H^1}&\leq \e^{\tau C_R}\|v(t_{n-1})-v^{n-1}\|_{H^1}+c_R\tau^{2}\\
&\leq \e^{2\tau C_R}\|v(t_{n-2})-v^{n-2}\|_{H^1}+\e^{\tau C_R}c_R\tau^{2}+c_R\tau^{2}\\
&\leq c_R\tau^{2}\sum_{k=0}^{n-1}\e^{k\tau C_R}\leq c_R\tau t_n \e^{t_n C_R}{,}
\end{align*}
which completes the proof.
\end{proof}
\begin{proof}[Proof of Theorem~\ref{thm:global_error_in_H1_second_order}]
	This result follows analogously by replacing Lemma~\ref{lem:local_error_estimate_H1_first_order} with Lemma~\ref{lem:local_error_estimate_H1_second_order} in the above proof.
\end{proof}
\subsection{Convergence analysis of the resonance-based midpoint rule for the NLSE}\label{sec:convergence_analysis_NLSE}
As mentioned in Example~\ref{ex:resonance-based_midpoint_rule} the resonance-based midpoint rule can also be constructed for the NLSE (and in particular is given in \eqref{eqn:resonance_based_midpoint_rule_u_NLSE}), {however we highlight that it differs significantly from the analysis presented in \cite{ostermann2018low} because we introduced the novel kernel approximation $\mathcal{K}_2(s;k,k_1,k_2,k_3)$ in the construction of our scheme \eqref{eqn:resonance_based_midpoint_rule_u_NLSE}.} The convergence analysis of this method can be performed similarly to the KdV case described in Section~\ref{sec:convergence_analysis_KdV}. For completeness we outline the main steps, in order also to highlight the regularity assumptions required for the NLSE case. The main result in this section is the convergence estimate:
\begin{theorem}\label{thm:global_error_midpoint_NLSE}
Let $u(t)$ be the exact solution to \eqref{eqn:Cauchy_problem_NLS} and be $u^n, n\geq 0,$ the iterates in the numerical method \eqref{eqn:resonance_based_midpoint_rule_u_NLSE}, and $t_n=n\tau$. For any $r>1/2, R>0 ,\gamma\in[0,2]$, there is a $\tau_R>0$ such that for all $\tau\in [0,\tau_R)$: If $\sup_{t\in [0,T]}\|u(t)\|_{H^{r+\gamma}}<R/2$, we have
	\begin{align*}
	\|u(t_n)-u^{n}\|_{H^r}\leq \tau^{\gamma} {C_{R,T,\gamma,r}},\quad \forall 0\leq n\leq\left\lfloor\frac{T}{\tau}\right\rfloor,
	\end{align*}
for some constant ${C_{R,T,\gamma,r}}>0$ depending on {$R,T,\gamma,r$, but which may be chosen independently of $\tau$.}
\end{theorem}

\subsubsection{Solution of implicit equations}
Our first step is again the solution of the implicit equations using fixed point iterates. For this we introduce the following auxilliary function in the NLSE case:
\begin{align*}
	\mathcal{S}(u):=e^{i\tau {\partial_x^2}}u^{n}+\tau e^{i\tau {\partial_x^2}}\mathcal{F}^{[NLSE,2]}_0\left(\tau; 1;\frac{u^{n}+e^{-i\tau{\partial_x^2}} u}{2}\right)
\end{align*}

Then we can prove the following crucial result:
\begin{theorem}\label{thm:sln_of_implicit_eqns_NLSE}
	Let $R>0$ and $s>1/2$. Then there is a constant $\tau_R$ such that for all $\tau\in(0,\tau_R)$ and any $u^n\in B_R(H^{s})=\{u\in H^s\,\vert\, \|u\|_{H^s}<R\}$ we have that $u^{n+1}$, the exact solution of \eqref{eqn:resonance_based_midpoint_rule_u_NLSE}, is given by the following limit in $H^s$:
	\begin{align}\label{eqn:limit_expression_for_u^{k+1}}
		u^{n+1}=\lim_{j\rightarrow \infty}\mathcal{S}^{(j)}(e^{i\tau{\partial_x^2}}u^n),\quad \text{where\ }\mathcal{S}^{(j)}(u)=\underbrace{\mathcal{S}\circ\cdots\circ\mathcal{S}}_{j\text{-times}}(u).
	\end{align}
Moreover, we have the estimate
\begin{align}\label{eqn:estimate_implicit_equation_forward_distance}
	\left\|u^{n+1}-e^{i\tau{\partial_x^2}}u^{n}\right\|_{H^s}\leq \tau \tilde{C}_R,
\end{align}
for some $\tilde{C}_R$ which depends only on $R$ (and $s$).
\end{theorem}
For the proof of this statement we exploit the following stability estimate.
\begin{lemma}\label{lem:crucial_stability_estimate}
	Fix $R>0, s>1/2$, then there is a $\tau_R>0$ such that for every $\tau\in(0,\tau_R)$ and any $u^{n},w,v\in B_R(H^s)$ we have 
	\begin{align*}
		\|\mathcal{S}(w)-\mathcal{S}(v)\|_{H^s}\leq \tau C_{R}\left\|w-v\right\|_{H^s}
	\end{align*}
\end{lemma}
\begin{proof}
Since $\exp(it{\partial_x^2})$ is an isometry on $H^{s}$ we have
    \begin{align*}
		&\|\mathcal{S}(w)-\mathcal{S}(v)\|_{H^s}^2=\\
		&\quad\tau^2|\mu|^2 \sum_{k\in\mathbb{Z}}\langle k\rangle^{2s}\left|\sum_{k+k_1=k_2+k_3}\left[\varphi_{1}(-2i\tau kk_1)+\varphi_{1}(2i\tau k_2k_3)-1\right]\left(\overline{(\hat{u}_{k_1}^{n}+\hat{w}_{k_1})}(\hat{u}_{k_2}^{n}+\hat{w}_{k_2})(\hat{u}_{k_3}^{n}+\hat{w}_{k_3})\right.\right.\\
 &\quad\quad\quad\quad\quad\quad\quad\quad\quad\quad\quad\quad\quad\quad\quad\quad\quad\quad\quad\quad\quad\quad\quad\quad\quad\quad\quad\quad\left.\left. -\overline{(\hat{u}_{k_1}^{n}+\hat{v}_{k_1})}(\hat{u}_{k_2}^{n}+\hat{v}_{k_2})(\hat{u}_{k_3}^{n}+\hat{v}_{k_3})\right)\right|^2{.}
	\end{align*}
Now we have the simple expression: $abc-def=(a-d)bc+d(b-e)c+de(c-f)$, which implies after a few applications of the triangle inequality and noting that the coefficients $\varphi_{1}(-2i\tau kk_1)+\varphi_{1}(2i\tau k_2k_3)-1$ are uniformly bounded in $k,k_j, j=1,2,3$, that
\begin{align*}
\|\mathcal{S}(w)-\mathcal{S}(v)\|_{H^s}^2\leq C \tau^2 \left\|(w-v)\right\|_{H^s}^2\left\|u^{n}+w\right\|_{H^s}\left\|u^{n}+v\right\|_{H^s}
\end{align*}
and the result follows.
\end{proof}
\begin{proof}[Proof of Theorem~\ref{thm:sln_of_implicit_eqns_NLSE}]
	If we take $\tau_R=1/(2C_R)$ with $C_R$ as in Lemma~\ref{lem:crucial_stability_estimate} we find that $\|\mathcal{S}(w)-\mathcal{S}(v)\|_{H^s}\leq \frac{1}{2}\|w-v\|_{H^s}$ and thus \eqref{eqn:limit_expression_for_u^{k+1}} follows by the Banach fixed point theorem. Let us now write
	\begin{align*}
	\left\|\mathcal{S}^{(J)}\left(e^{i\tau{\partial_x^2}}u^{k}\right)-e^{i\tau{\partial_x^2}}u^k\right\|_{H^s}&\leq \sum_{j=0}^{J-1}\left\|\mathcal{S}^{(j+1)}\left(e^{i\tau{\partial_x^2}}u^{k}\right)-\mathcal{S}^{(j)}\left(e^{i\tau{\partial_x^2}}u^{k}\right)\right\|_{H^s}\\
	&\leq \sum_{j=0}^{J-1}2^{-j}\left\|\mathcal{S}\left(e^{i\tau{\partial_x^2} }u^{k}\right)-e^{i\tau{\partial_x^2}}u^{k}\right\|_{H^{s}}\\
	&\leq 2\left\|\mathcal{S}\left(e^{i\tau{\partial_x^2}}u^k\right)-e^{i\tau{\partial_x^2}}u^{k}\right\|_{H^s}\leq 2 C_R\tau,
	\end{align*}
where the final estimate follows similarly to the proof of Lemma~\ref{lem:crucial_stability_estimate}. The result then follows analogously to the proof of Theorem~\ref{thm:fixed_point_iterations_v}.
\end{proof}

\subsubsection{Stability}
Having understood the solution of the implicit equation we can turn our attention to the stability analysis of the numerical scheme. Let us denote by $\Phi_{\tau}(u)$ the solution of the implicit time stepping scheme \eqref{eqn:resonance_based_midpoint_rule_u_NLSE}.
\begin{proposition}\label{prop:stability_estimate_NLSE}
	Fix $R>0$ and $s>1/2$. Then there is a $\tau_{R}>0$ and $C_R>0$ such that for all $\tau\in(0,\tau_R)$ and any $w,v\in B_{R}(H^s)$ we have
	\begin{align*}
		\left\|\Phi_{\tau}(v)-\Phi_{\tau}(w)\right\|_{H^s}\leq \exp\left(\tau C_R\right)\|v-w\|_{H^s}
	\end{align*}
where $C_R>0$ depends only on $R$ (and $s$).
\end{proposition}
\begin{proof}
We have
\begin{align*}
	&\left\|\Phi_{\tau}(v)-\Phi_{\tau}(w)\right\|_{H^s}\leq  \left\|e^{i\tau{\partial_x^2}}(v-w)\right\|_{H^s}\\
&\quad+\tau\left\|e^{i\tau{\partial_x^2}}\mathcal{F}^{[NLSE,2]}_0\left(\tau; 1;\frac{v+e^{-i\tau{\partial_x^2}} \Phi_{\tau}(v)}{2}\right)-e^{i\tau{\partial_x^2}}\mathcal{F}^{[NLSE,2]}_0\left(\tau; 1;\frac{w+e^{-i\tau{\partial_x^2}} \Phi_{\tau}(w)}{2}\right)\right\|_{H^s}
\end{align*}
Using Lemma~\ref{lem:crucial_stability_estimate} we find that 
\begin{align*}
	\left\|\Phi_{\tau}(v)-\Phi_{\tau}(w)\right\|_{H^s}\leq \left\|v-w\right\|_{H^s}+\tau \frac{C_R}{2}\left\|(v-w)+(\Phi_{\tau}(v)-\Phi_{\tau}(w))\right\|_{H^s}
\end{align*}
which implies using the triangle inequality again:
\begin{align*}
	\left\|\Phi_{\tau}(v)-\Phi_{\tau}(w)\right\|_{H^s}\leq\frac{1+\tau C_R/2}{1-\tau C_R/2}\|v-w\|_{H^s}\leq e^{\frac{3\tau C_R}{2}}\|v-w\|_{H^s},
\end{align*}
so long as we take $\tau_R<C_R^{-1}$.
\end{proof}

\subsubsection{Local error}
The next step is to estimate the local error of the approximation. This depends on the specifics of the low-regularity error \eqref{eqn:local_error_second_order_kernel_approx_NLSE} in our symplectic kernel approximations $\mathcal{K}_{2}(s;k,k_1,k_2,k_3)$.

\begin{lemma}\label{lem:local_error_estimates_NLSE} Let us denote by {$\tau\mapsto\phi_{\tau}(u(t_n))$} the solution to \eqref{eqn:Cauchy_problem_NLS} with initial condition $\phi_{0}(u(t_n))=u(t_n)$. Fix $R>0,s>1/2,\gamma\in[0,2]$, then there is a $\tau_R>0$ such that for all $\tau\in [0,\tau_R)$ we have
whenever $\sup_{t\in[0,\tau]}\|{\phi_{t}}(u(t_n))\|_{H^{s+\gamma}}<R$ then
\begin{align*}
		\|\phi_{\tau}(u(t_n))-\Phi_{\tau}(u(t_n))\|_{H^s}\leq c_R\tau^{1+\gamma} 
	\end{align*}
for some constant $c_R>0$ depending on $R>0,s,\gamma$.
\end{lemma}
\begin{proof}
We have
\begin{align*}
    \Phi_{\tau}(u(t_n))-&\Bigg(e^{i\tau\partial_x^2}u(t_n)\\
    &\quad-\left.i\mu e^{i\tau\partial_x^2}\int_{0}^{\tau}e^{-is{\partial_x^2}}\left|e^{is{\partial_x^2}}\frac{u(t_n)+e^{-i\tau{\partial_x^2}} \Phi_\tau(u(t_n))}{2}\right|^{2}e^{is{\partial_x^2}}\frac{u(t_n)+e^{-i\tau{\partial_x^2}} \Phi_\tau(u(t_n))}{2}\dd s\right)\\
    &\quad\quad\quad=-i\mu\sum_{k\in\mathbb{Z}}e^{ix k}\sum_{k+k_1=k_2+k_3}\left[\int_{0}^{\tau}e^{-2iskk_1+2isk_2k_3}-\mathcal{K}_2(s;k,k_1,k_2,k_3)\dd s\right]\overline{\hat{w}_{k_1}^n}\hat{w}_{k_2}^n\hat{w}_{k_3}^n,
\end{align*}
where for notational simplicity we wrote $w=\frac{u(t_n)+e^{-i\tau{\partial_x^2}} \Phi_\tau(u(t_n))}{2}$. By estimate \eqref{eqn:local_error_second_order_kernel_approx_NLSE} it therefore immediately follows that for any $s>1/2,\gamma\in[0,2]$
\begin{align*}
    &\Bigg\|\Phi_{\tau}(u(t_n))-\Bigg(e^{i\tau\partial_x^2}u(t_n)\\
    &\quad\quad\quad\quad\quad\quad\quad\left.\left.-i\mu e^{i\tau\partial_x^2}\int_{0}^{\tau}e^{-is{\partial_x^2}}\left|e^{is{\partial_x^2}}\frac{u(t_n)+e^{-i\tau{\partial_x^2}} \Phi_\tau(u(t_n))}{2}\right|^{2}e^{is{\partial_x^2}}\frac{u(t_n)+e^{-i\tau{\partial_x^2}} \Phi_\tau(u(t_n))}{2}\dd s\right)\right\|_{H^{s}}\\
    &\quad\quad\quad\quad\quad\quad\quad\quad\quad\quad\quad\quad\quad\quad\quad\quad\quad\quad\quad\quad\quad\quad\quad\quad\leq C\tau^{1+\gamma}\left\|\frac{u(t_n)+e^{-i\tau{\partial_x^2}} \Phi_\tau(u(t_n))}{2}\right\|_{H^{s+\gamma}}\\
    &\quad\quad\quad\quad\quad\quad\quad\quad\quad\quad\quad\quad\quad\quad\quad\quad\quad\quad\quad\quad\quad\quad\quad\quad\leq \frac{C}\tau^{1+\gamma}\left(\left\|u(t_n)\right\|_{H^{s+\gamma}}+\left\|\Phi_\tau(u(t_n))\right\|_{H^{s+\gamma}}\right)\\
\end{align*}
for a constant $C>0$ independent of $\tau$ and where in the final line we used Theorem~\ref{thm:sln_of_implicit_eqns_NLSE}. Thus it remains to consider the quantity
\begin{align*}
       &{\phi_{\tau}}(u(t_n))\\&-\left(e^{i\tau\partial_x^2}u(t_n)-i\mu e^{i\tau\partial_x^2}\int_{0}^{\tau}e^{-is{\partial_x^2}}\left|e^{is{\partial_x^2}}\frac{u(t_n)+e^{-i\tau{\partial_x^2}} \Phi_\tau(u(t_n))}{2}\right|^{2}e^{is{\partial_x^2}}\frac{u(t_n)+e^{-i\tau{\partial_x^2}} \Phi_\tau(u(t_n))}{2}\dd s\right)
  \end{align*}
  which can be estimated {in $H^s$} by taking similar steps to the proof of the local error estimates for the KdV equation (since now the integral kernel corresponds to the exact flow) in Lemmas~\ref{lem:local_error_estimate_H1_first_order}\&\ref{lem:local_error_estimate_H1_second_order}.
\end{proof}

\subsubsection{Global error estimate}
We can now prove Theorem~\ref{thm:global_error_midpoint_NLSE} analogously to the global error estimates in Section~\ref{sec:convergence_analysis_KdV} by combining Proposition~\ref{prop:stability_estimate_NLSE} and Lemma~\ref{lem:local_error_estimates_NLSE}.
\subsection{{Comments on the general case}}\label{sec:general_idea_convergence_analysis}
{Following the steps introduced at the beginning of Section~\ref{sec:convergence_analysis} we note that the above ideas generalise to the analysis of more general RK resonance-based schemes in the following way. Suppose we are given a RK resonance-based scheme of the form \eqref{eqn:RK_res_based_schemes_general}, then the first step in the above analysis is to establish the well-posedness of the implicit system
\begin{align}\label{eqn:Kequations_general_RK_scheme}
	K_{p,q,r}&=\mathcal{F}^{[\mathcal{L},\rho]}_p(\tau; c_q; u^{n}+\tau \sum_{\tilde{p},\tilde{q},\tilde{r}=0}^{S} a_{p,q,r}^{\tilde{p},\tilde{q},\tilde{r}}K_{\tilde{p},\tilde{q},\tilde{r}}).
\end{align}
Of course, if this system is explicit we can move to step (ii) directly, but if not, in the first instance we aim to show that estimates analogous to Theorems~\ref{thm:fixed_point_iterations_v} \& \ref{thm:sln_of_implicit_eqns_NLSE} holds. This means we need to show that, for $\tau$ sufficiently small, the implicit equations \eqref{eqn:Kequations_general_RK_scheme} can be solved using fixed point iterations and that estimates of the form
\begin{align}\begin{split}\label{eqn:proximity_time_step_general}
    \left\|K_{p,q,r}-\mathcal{F}^{[\mathcal{L},\rho]}_p(\tau; c_q; u^{n})\right\|_{H^s}\leq \tau C_1,\quad \left\|u^{n+1}-e^{i\tau\mathcal{L}(\nabla)}u^n\right\|_{H^s}\leq \tau C_1,
    \end{split}
\end{align}
hold for some $s\geq0$, all $\tau>0$ sufficiently small and for some $C_1>0$ which may depend on $\|u^n\|_{H^s}$. (ii) Once this estimate is established we denote the RK resonance-based scheme by $\Phi_{\tau}^{[\mathcal{L},\rho]}$ and use \eqref{eqn:proximity_time_step_general} to obtain a stability estimate, similar to Lemma~\ref{lem:stability_in_H1} and Proposition~\ref{prop:stability_estimate_NLSE}, which takes the form
\begin{align*}
    \left\|\Phi_{\tau}^{[\mathcal{L},\rho]}(u)-\Phi_{\tau}^{[\mathcal{L},\rho]}(w)\right\|_{H^s}\leq e^{\tau C_2}\left\|u-w\right\|_{H^s},
\end{align*}
where $C_2>0$ may depend on the size of $u,w$ in an appropriate norm. This estimate exploits the properties of the maps $\mathcal{F}^{[\mathcal{L},\rho]}_p$ (introduced in \eqref{eqn:nonlinear_maps_in_general_RK_schemes}) which, for the KdV equation and the kernel approximations shown in this manuscript for the NLSE case, are all unconditionally stable in appropriate Sobolev norms. (iii) Finally, we study the local error of the scheme establishing bounds similar to Lemmas~\ref{lem:local_error_estimate_H1_first_order} \& \ref{lem:local_error_estimates_NLSE}. Generally speaking this step is easier to perform on the twisted variable noting that the change of variable $v(t)=\exp(-it\mathcal{L}(\nabla))u(t)$ is an isometry on $H^s$. The idea for this step is to find and compare two expansions: Firstly, an expansion of the exact flow of the twisted variable $\phi_{t_n,\tau}$ which is obtained by suitable iteration of Duhamel's formula \eqref{eqn:Duhamel_forumla_general_dispersive} and, secondly, an expansion of the numerical solution by suitable repeated substitution of the implicit RK stage values into the expression of the solution and the use of \eqref{eqn:proximity_time_step_general} to replace the final implicit term. For example a first iteration of this form in \eqref{eqn:RK_res_based_schemes_general} may take the form
\begin{align*}
	u^{n+1}&=e^{i\tau \mathcal{L}(\nabla)}u^{n}+\tau\sum_{p,q,r=0}^{S} b^{p,q,r}e^{i\tau\mathcal{L}(\nabla)}\mathcal{F}^{[\mathcal{L},\rho]}_p(\tau; c_q; u^{n})+\mathcal{O}(\tau^2),
\end{align*}
where the inherent constant in $\mathcal{O}(\tau^2)$ can be bounded in terms of an appropriate Sobolev norm of $u^{n}$. These two series are then compared term by term, like in the analysis of classical B-series methods, to establish the order of the method using the estimate on the kernel approximation of the form \eqref{eqn:kernel_approximation_general_dispersive_eqn}. The important point to keep in mind in this dispersive setting however is to estimate the difference in these series inside the integrals appearing in Duhamel's formula and \eqref{eqn:nonlinear_maps_in_general_RK_schemes}, similarly to the analysis performed in the proof of Lemmas~\ref{lem:local_error_estimate_H1_first_order} \& \ref{lem:local_error_estimates_NLSE} to ensure that the local error estimates are valid for solutions of low regularity. These series expansions can in principle be done by hand but it is subject of ongoing work to use ideas introduced by \cite{luan2013exponential} for the derivation of stiff order conditions in exponential integrators (using so-called exponential B-series) to perform the local error analysis of RK resonance-based schemes in a more structured manner.}
\section{Numerical experiments}
\label{sec:numerical_experiments}
In this section we test {the new} symplectic resonance-based schemes \eqref{eqn:resonance_based_midpoint_rule_u_KdV} \& \eqref{eqn:resonance_based_midpoint_rule_u_NLSE} numerically, and   exhibit their favourable properties in practice. {In particular, we will study the practical behaviour of our methods in terms of
\begin{itemize}
    \item convergence in low-regularity regimes;
    \item preservation of the quadratic first integrals \eqref{eqn:quadratic_first_integral_kdv} \& \eqref{eqn:quadratic_first_integral_nlse}, which we expect to be preserved close to machine accuracy in our new schemes;
    \item approximate preservation of the energy.
\end{itemize}
The final property requires more subtle study beyond the scope of this work in that, even though in the ODE case symplectic integrators are known to preserve the energy approximately over long times \cite{hairer2013geometric}, the current understanding of the PDE case is much more limited (important contributions were made for example in \cite{faou2004energy,faou2012geometric,faougaucklerlubich2014,gauckler2010splitting}) and in particular there is no general analogue of the ODE result which rigorously guarantees good long-time behaviour of symplectic integrators for PDEs. Nevertheless, we show in some initial numerical experiments that for a large subset of non-resonant frequencies the energy appears to be preserved approximately over long times using the resonance-based midpoint rule. We note that this preservation can be observed to fail at resonant time steps similar to results for splitting methods \cite{faou2012geometric} and we therefore do not make any definite claims about the energy preservation properties of our new methods.}

In all of the experiments our spatial discretisation is a spectral method with $M$ Fourier modes and our initial conditions are of two types: Firstly, smooth initial conditions $u_0\in C^\infty(\mathbb{T})$ of the form
\begin{align}\label{eqn:def_smooth_initial_data_numerics}
    u_0(x)=\frac{\cos(x)}{2+\sin (x)},
\end{align}
with an appropriate normalisation $u_0\mapsto u_0/\|u_0\|_{L^{2}}$ and, in the KdV case, centering $u_0\mapsto u_0-\int_{\mathbb{T}}u_0\dd x$ to satisfy assumption~\ref{assumption:zero_mass}. Secondly, rough initial conditions $u_0\in H^\theta,\theta>1/2$ of the form
\begin{align}\label{eqn:law_for_random_initial_conditions}
\hat{u}_{(M),m}^0=\langle m\rangle^{-\vartheta}U_{M,m},
\end{align}
where the rescaled Fourier coefficients of our initial condition, $\mathbf{U}_M$, are chosen as a single sample {(using} \verb|rand(M,1)+i*rand(M,1)| {in Matlab)} of a uniform random distribution
\begin{align*}
\mathbf{U}_M\sim \mathcal{U}([-1,1]^M+i[-1,1]^M),
\end{align*}
with appropriate normalisation and centering as in the smooth case. {The code associated with this work is available on the GitHub repository} \verb|GLIMPSE| {\cite{GLIMPSE23}}.

\subsection{NLSE}

We begin our discussion with the cubic NLSE, $p=1$ in \eqref{eqn:Cauchy_problem_NLS}, because for this equation very competitive reference methods exist in the literature. In particular, we will evaluate the performance of the resonance-based midpoint rule as an example of a second order symplectic low-regularity integrator in the class \eqref{eqn:RK_res_based_schemes_NLSE}, against the following methods:
\begin{itemize}
    \item the second order low-regularity integrator introduced by \cite[Section~5.1.2]{bruned_schratz_2022}, denoted by `Bruned \& Schratz';
    \item the classical Strang splitting, which is a symplectic scheme for the NLSE (cf. \cite{mclachlan2002splitting}), denoted by `Strang';
    \item the second order $L^2$-norm preserving Lawson method introduced in \cite[Example~3.2]{celledoni2008symmetric}, denoted by `Lawson'.
\end{itemize}

In all of the numerical experiments reference solutions were computed with `Bruned \& Schratz 2022' with a reference time-step of $\tau_{ref}=10^{-6}$ and {a spatial discretisation with $M=2^{14}$ Fourier modes.}

\subsubsection{Convergence rates}

We begin by considering the convergence properties of our new scheme in comparison to previous work. For this we choose initial data of three different levels of regularity $u_0\in H^{2},u_0\in H^{3}$ with $\vartheta=2,3$ in \eqref{eqn:law_for_random_initial_conditions} and $u_0\in C^\infty$ with \eqref{eqn:def_smooth_initial_data_numerics} and measure the $H^1$ error at time $T=1$ for a range of time steps $\tau$. In all of the following numerical experiments we took $M=1024$ Fourier modes and the results are shown in Figures~\ref{fig:convergence_properties_NLSE}\&\ref{fig:time_step_vs_error_Cinf_data_NLSE}. We clearly see that the predicted convergence rates of order one and two are achieved at those levels of regularity as per Theorem~\ref{thm:global_error_midpoint_NLSE}. Even though the error constant of `Bruned and Schratz 2022' is slightly smaller than our new resonance-based midpoint rule, the latter converges at the same rates and is in particular able to clearly outperform the Strang splitting and the $L^2$-norm preserving Lawson method `Lawson' for low-regularity solutions.

\begin{figure}[h!]
\centering
	\begin{subfigure}{0.495\textwidth}
		\centering
		\includegraphics[width=0.97\textwidth]{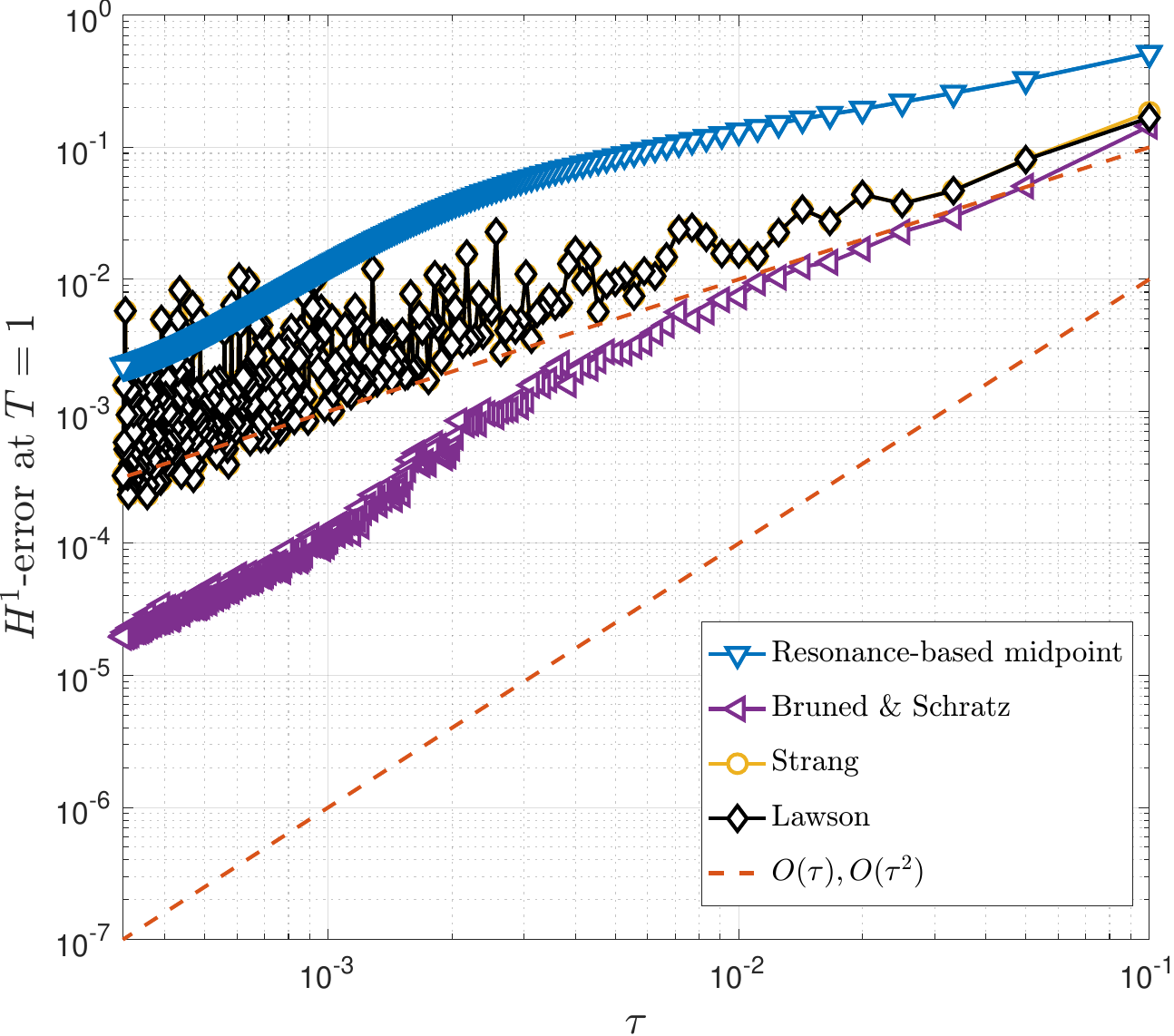}
  \caption{Initial data $u_0\in H^2$, as per \eqref{eqn:law_for_random_initial_conditions} with $\vartheta=2$.}
		\label{fig:time_step_vs_error_H2_data_NLSE}
	\end{subfigure}
	\begin{subfigure}{0.495\textwidth}
		\centering
		\includegraphics[width=0.97\textwidth]{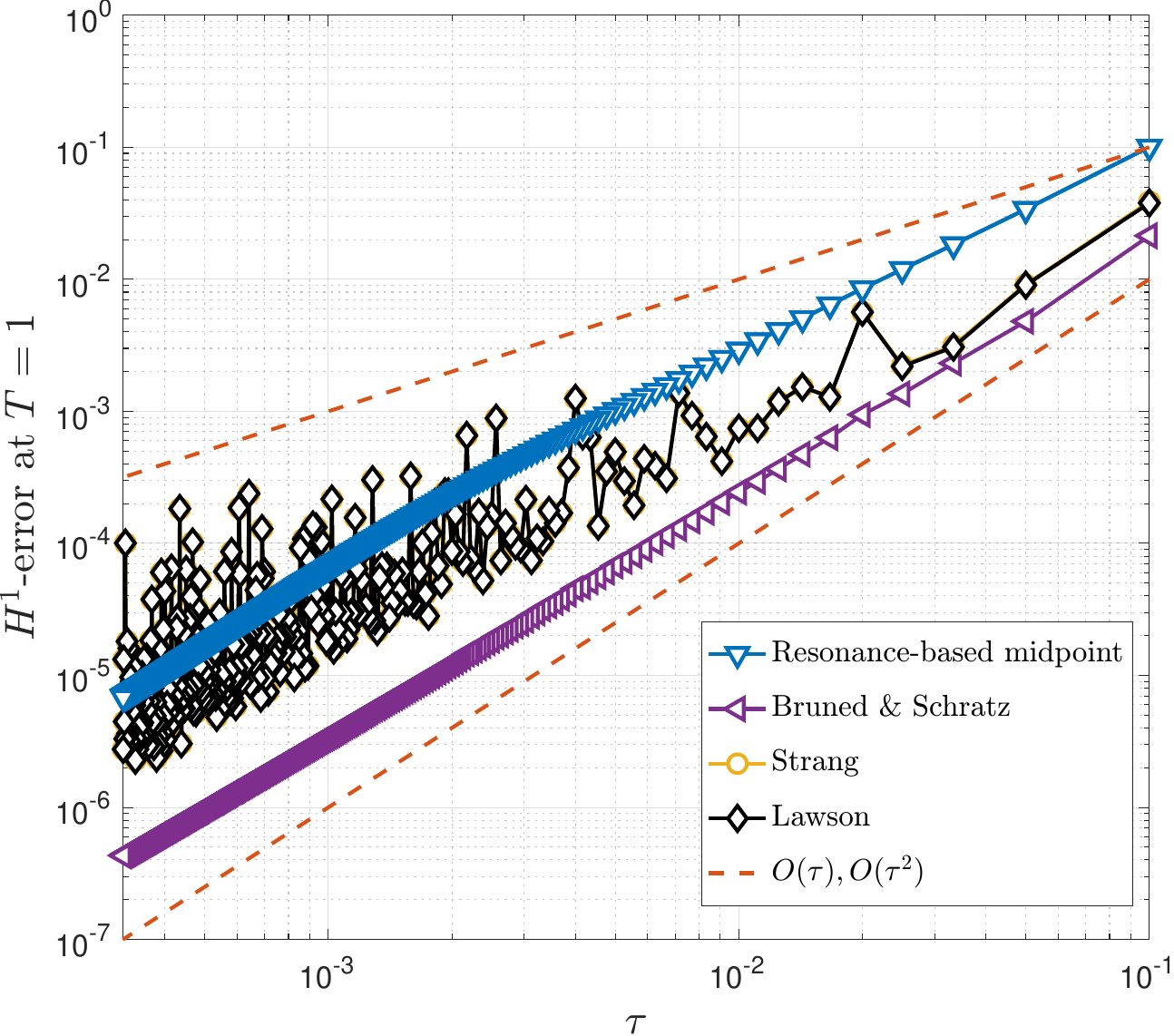}
  \caption{Initial data $u_0\in H^3$, as per \eqref{eqn:law_for_random_initial_conditions} with $\vartheta=3$.}
	\label{fig:time_step_vs_error_H3_data_NLSE}
	\end{subfigure}
	\caption{Order plot measured in $H^1$.}
	\label{fig:convergence_properties_NLSE}
\end{figure}
\begin{figure}[h!]
		\centering
		\includegraphics[width=0.5\textwidth]{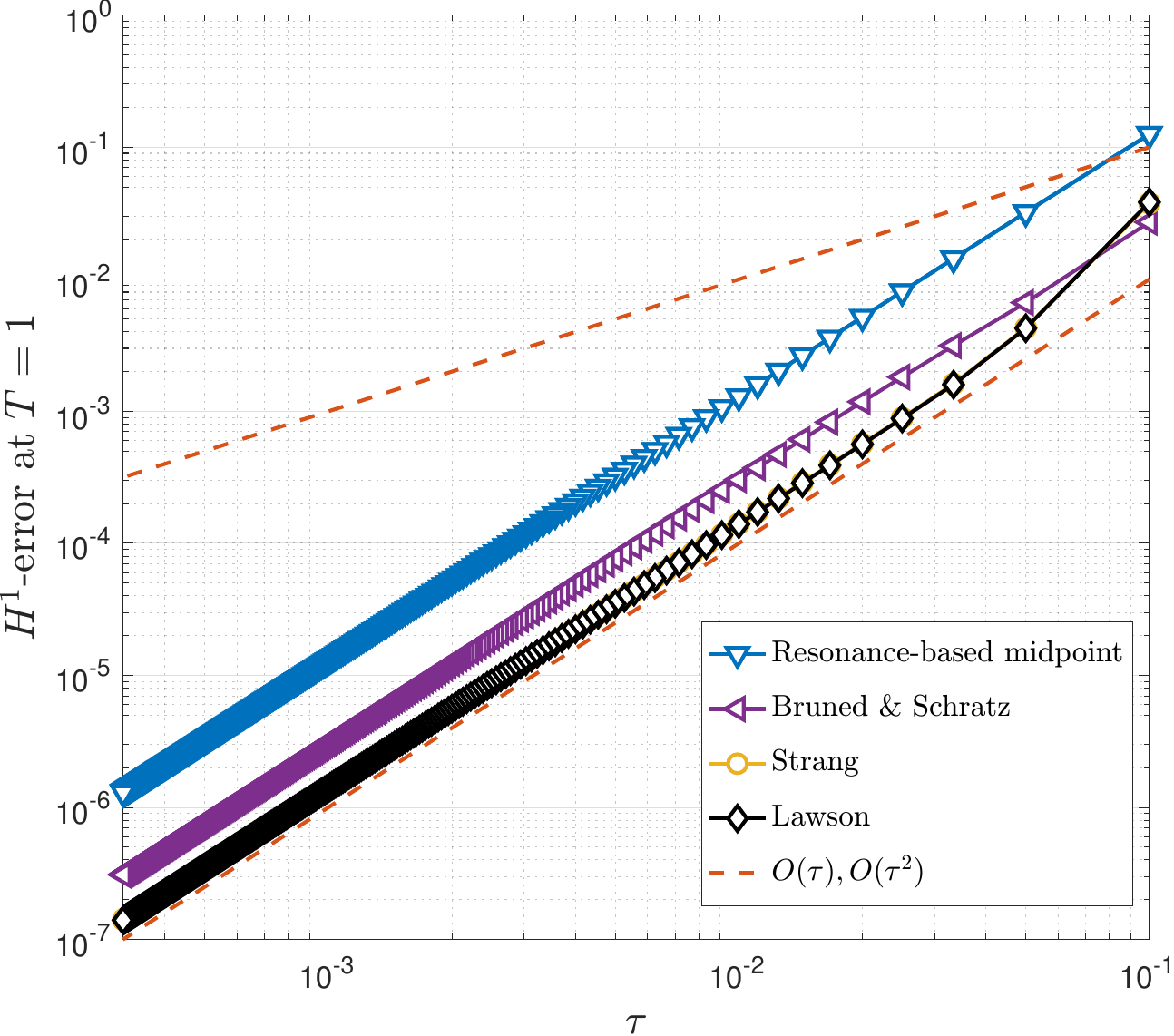}
		\caption{Order plot measured in $H^1$, initial data $u_0\in C^\infty$, as per \eqref{eqn:def_smooth_initial_data_numerics}.}
		\label{fig:time_step_vs_error_Cinf_data_NLSE}
\end{figure}

\ \\
\vspace{-1.5cm}\ \\
\subsubsection{Structure preservation properties}
In the next instance we shall look at how well our proposed method is able to preserve conservation laws from the NLSE. By Theorem~\ref{thm:L2_pres_condition_direct_flow} we expect the resonance-based midpoint rule to preserve the $L^2$-norm \eqref{eqn:quadratic_first_integral_nlse} of the solution to machine accuracy - and the same is expected for both `Strang' and `Lawson'. We can observe this expected behaviour both for low-regularity solutions, in Figure~\ref{fig:nlseL2preservationlowregularity}, and for smooth solutions, in Figure~\ref{fig:nlseL2preservationsmooth}. For both experiments we took $M=2048$ and $\tau=0.02$. From numerical experiments it is apparent that the existing low-regularity integrator `Bruned \& Schratz' \cite{bruned_schratz_2022} exhibits a clear drift in the error of the $L^2$-norm and is therefore unable to preserve this first integral over long times, thus justifying our novel constructions.

\begin{figure}[h!]
 \centering
     \includegraphics[width=0.91\linewidth]{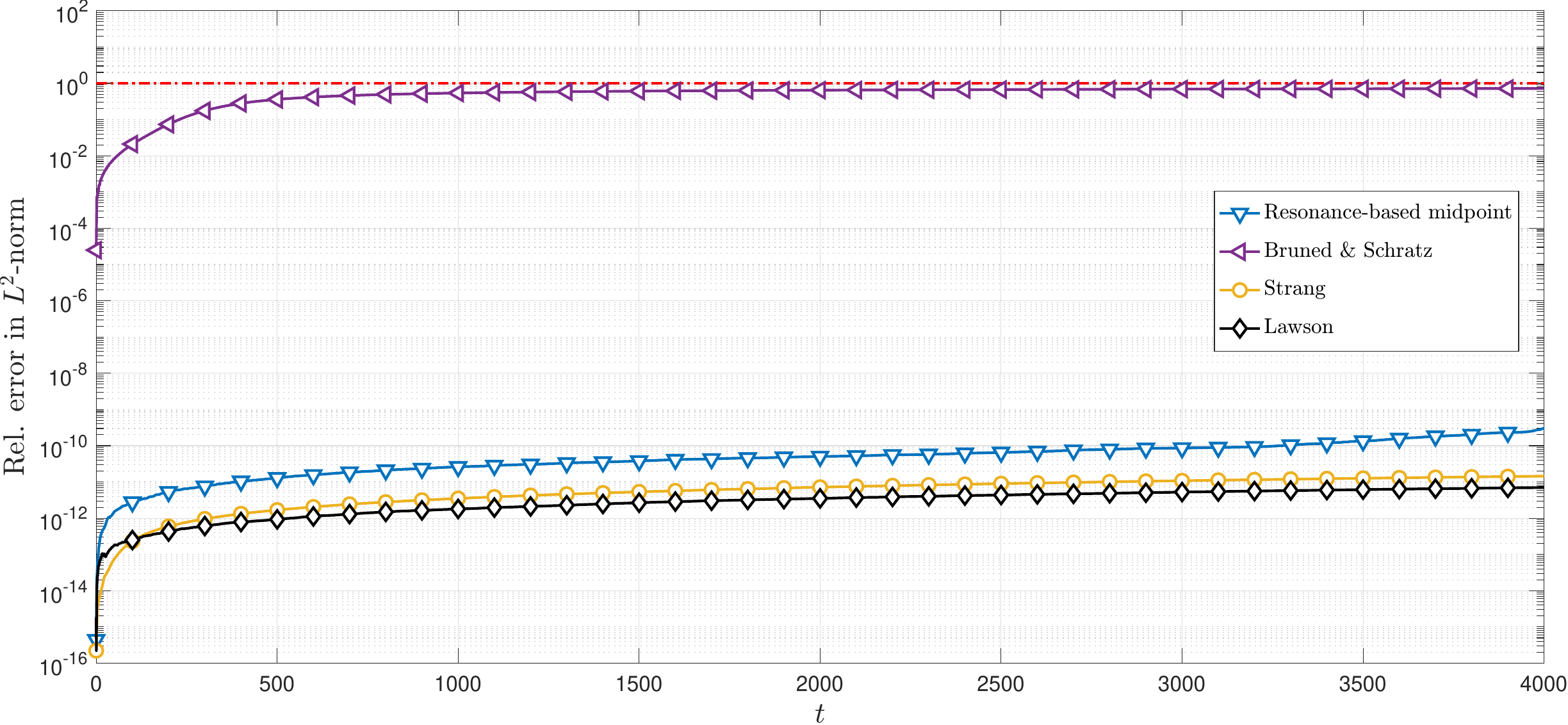}
	\caption{Relative error in the $L^2$-norm with $u_0\in H^2$, as per \eqref{eqn:law_for_random_initial_conditions} with $\vartheta=2$.}
	\label{fig:nlseL2preservationlowregularity}
 \end{figure}
\begin{figure}[h!]
	\centering
	\includegraphics[width=0.9\linewidth]{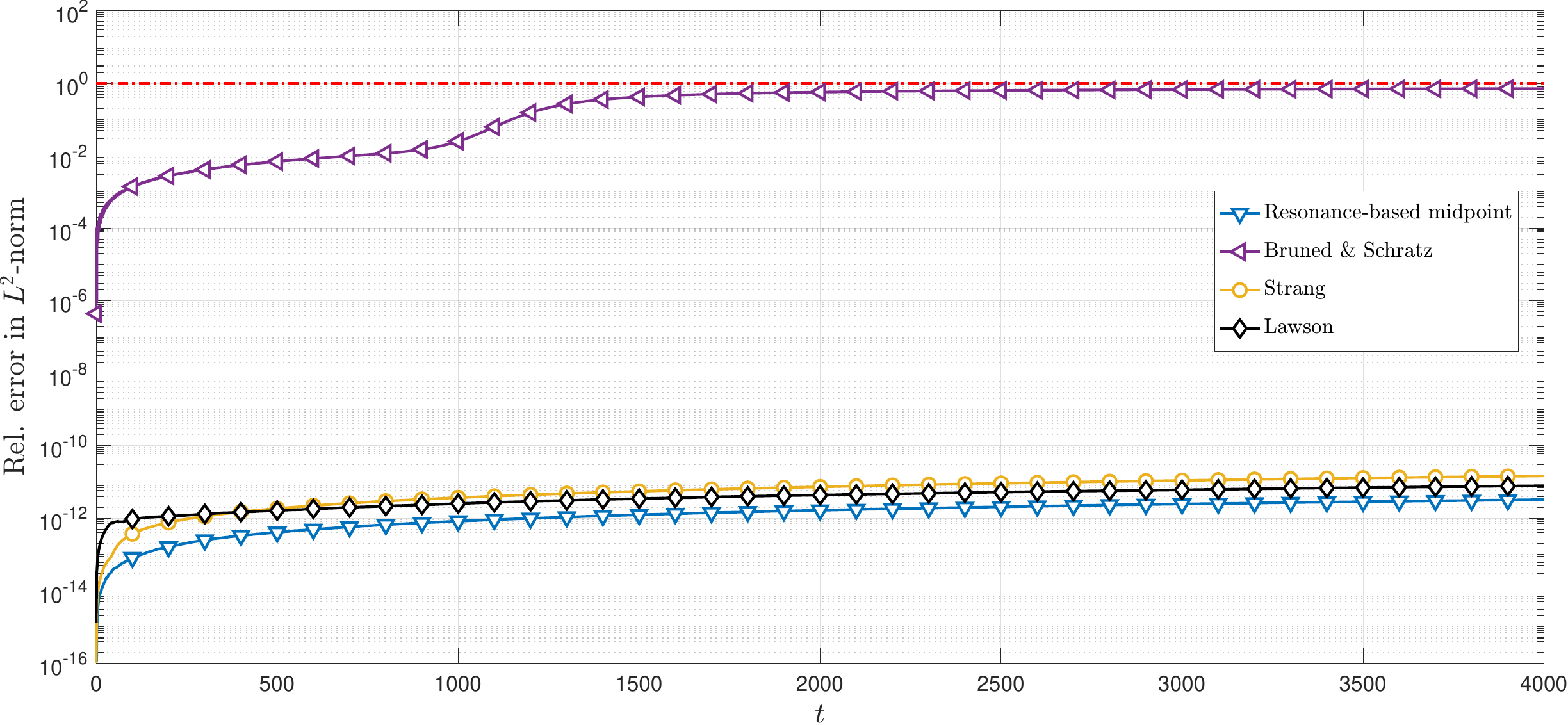}
	\caption{Relative error in the $L^2$-norm with $u_0\in C^\infty$, as per \eqref{eqn:def_smooth_initial_data_numerics}.}
	\label{fig:nlseL2preservationsmooth}
 \end{figure}

Finally, we can consider the long-time error in the Hamiltonian 
\begin{align*}
    \mathcal{H}^{[NLSE]}(u)=\int_{\mathbb{T}}|\partial_x u|^2+\frac{\mu}{2}|u|^4\dd x.
\end{align*}

{There is no theoretical guarantee for the Hamiltonian to be preserved over long times in our scheme.} However, for ODEs it is known that symplectic integrators are able to approximately preserve the Hamiltonian over exponentially long times \cite[Theorem~IX.8.1]{hairer2013geometric}. A comparable long-time preservation of the Hamiltonian was shown for splitting methods at non-resonant time steps and subject to a CFL condition of the form $\tau\lesssim M^{-2}$ in \cite{faou2012geometric}. In our numerical experiments we observe that our symplectic integrator appears to be able to achieve a similar feat: at a large number of tested time steps the energy is approximately preserved over very long time intervals. We suspect that this energy conservation may break down at isolated resonant time steps of similar nature to those found in \cite{faou2012geometric} but the analysis of this long-time behaviour is subject of future work.

\begin{figure}[h!]
	\centering
 \begin{subfigure}{1\textwidth}
 \centering
	\includegraphics[width=0.9\linewidth]{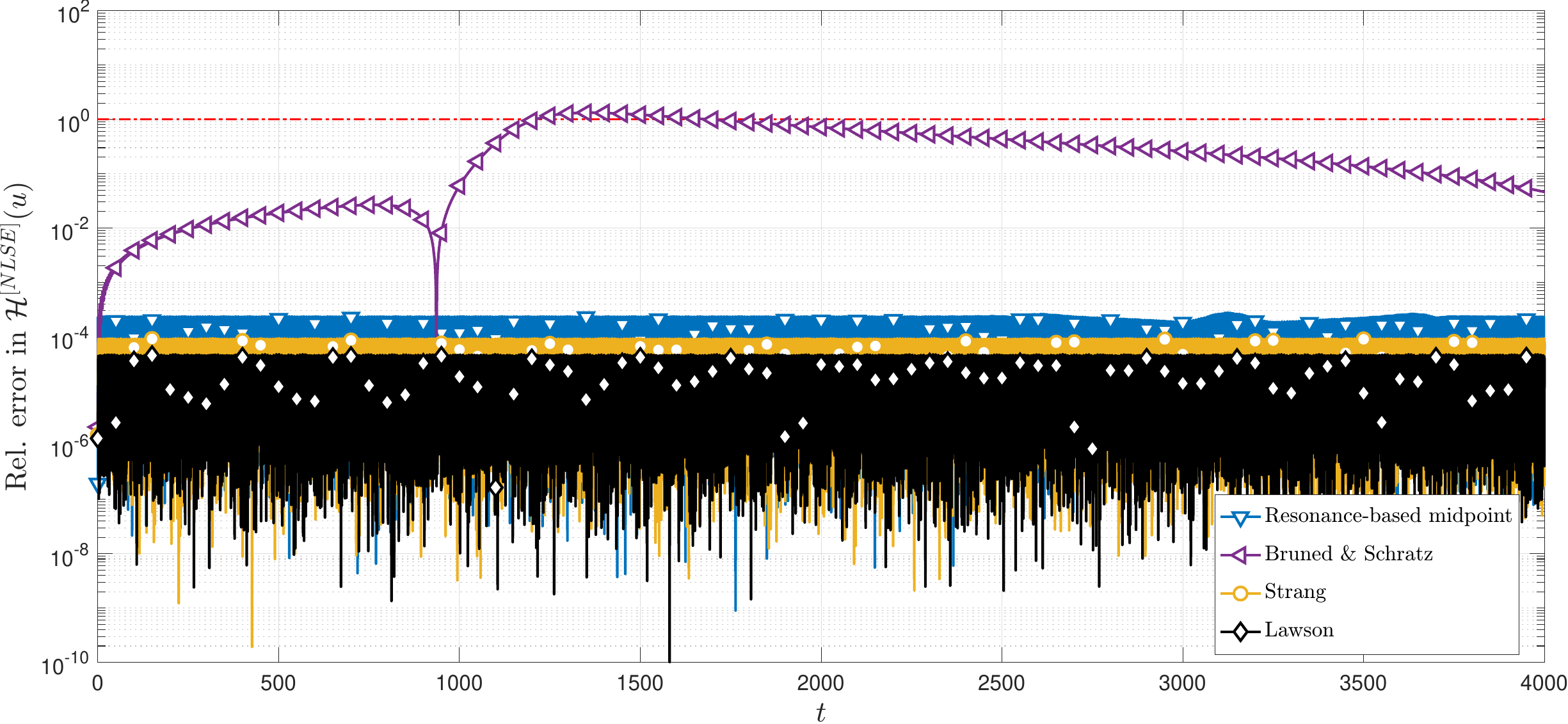}
	\caption{$M=64$.}
	\label{fig:nlse_hamiltonianpreservationsmooth_Msmall}
 \end{subfigure}
 \begin{subfigure}{1\textwidth}
 \centering
	\centering
	\includegraphics[width=0.9\linewidth]{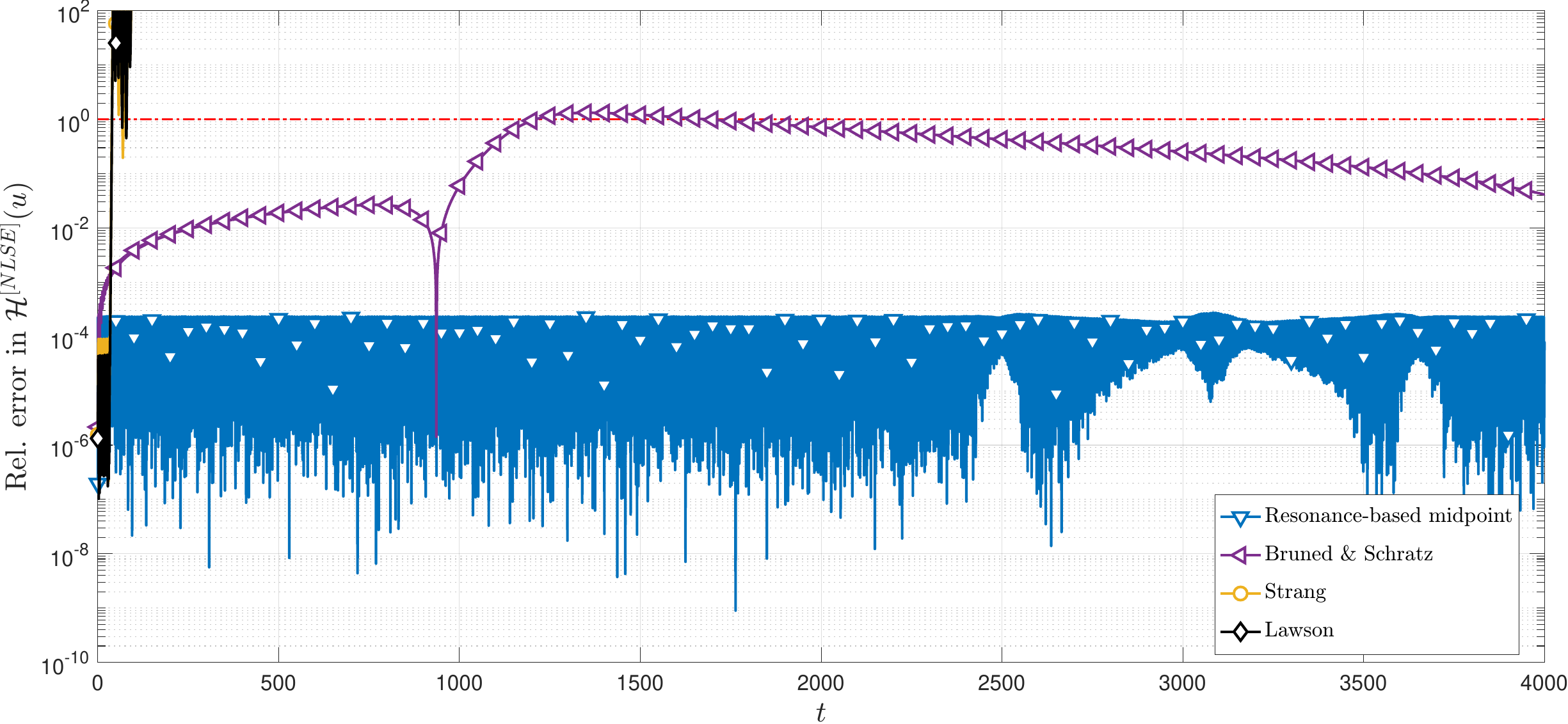}
 \caption{$M=256$.}
 \label{fig:nlse_hamiltonianpreservationsmooth_Mlarge}
 \end{subfigure}
	\caption{Error in the Hamiltonian of the numerical solution, for $u_0\in C^\infty(\mathbb{T})$ as per \eqref{eqn:def_smooth_initial_data_numerics}.}
	\label{fig:nlse_hamiltonianpreservationsmooth}
\end{figure}
In Figure~\ref{fig:nlse_hamiltonianpreservationsmooth} we exhibit the long-time behaviour for $u_0\in C^{\infty}$ for a specific choice of $\tau=0.02$ which is representative of the behaviour we observed at non-resonant time steps. It turns out that for $\tau$ sufficiently small (cf. Figure~\ref{fig:nlse_hamiltonianpreservationsmooth_Msmall}{)}, the Strang splitting and Lawson method, are able to preserve the Hamiltonian approximately over very long times - in keeping with theoretical expectations from the ODE setting. Equally our symplectic resonance-based midpoint rule can preserve the Hamiltonian over comparable times in the same regime. As expected, for larger $M$ (and same $\tau=0.02$) the CFL condition required for long-time energy preservation is no longer satisfied for the Strang splitting and thus the long-time behaviour breaks down, cf. Figure~\ref{fig:nlse_hamiltonianpreservationsmooth_Mlarge}. Perhaps somewhat surprising is that our new method appears to be able to continue to preserve the energy well even for larger values of $M$, which suggests that even in the smooth case our new method might be able to compete with prior work on approximate energy preservation, and in some cases even outperform previous state-of-the-art in this sense.

Finally, we note that similarly favourable long-time approximate energy preservation is observed for our method even for rough initial data. Figure~\ref{fig:nlse_hamiltonianpreservationrough} indicates that our method is able to approximately preserve the energy over long times even in low-regularity regimes where the energy conservation of the reference methods breaks down completely.

\begin{figure}[h!]
	\centering
	\begin{subfigure}{1\textwidth}
 \centering
	\includegraphics[width=0.9\linewidth]{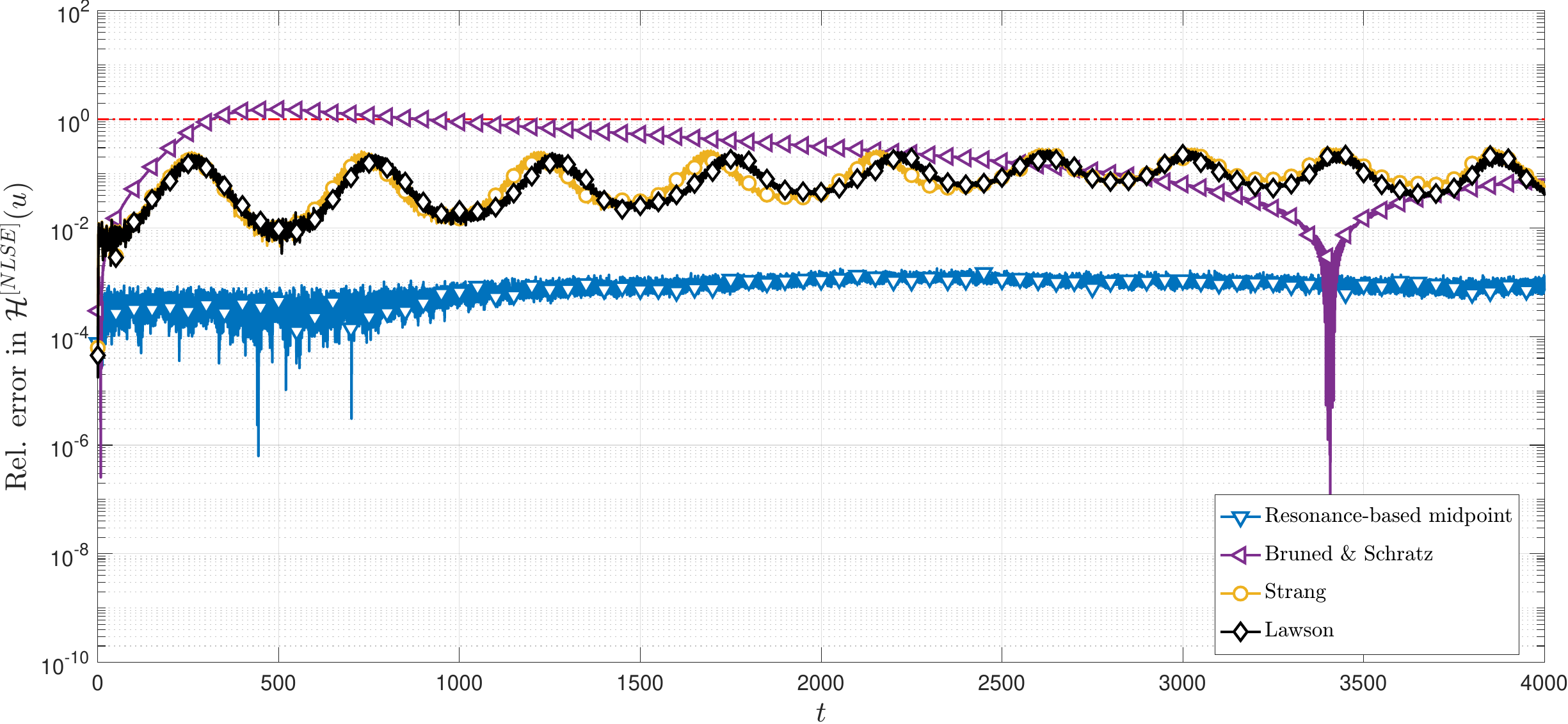}
	\caption{$M=64$.}
 \end{subfigure}
 \begin{subfigure}{1\textwidth}
 \centering
	\centering
	\includegraphics[width=0.9\linewidth]{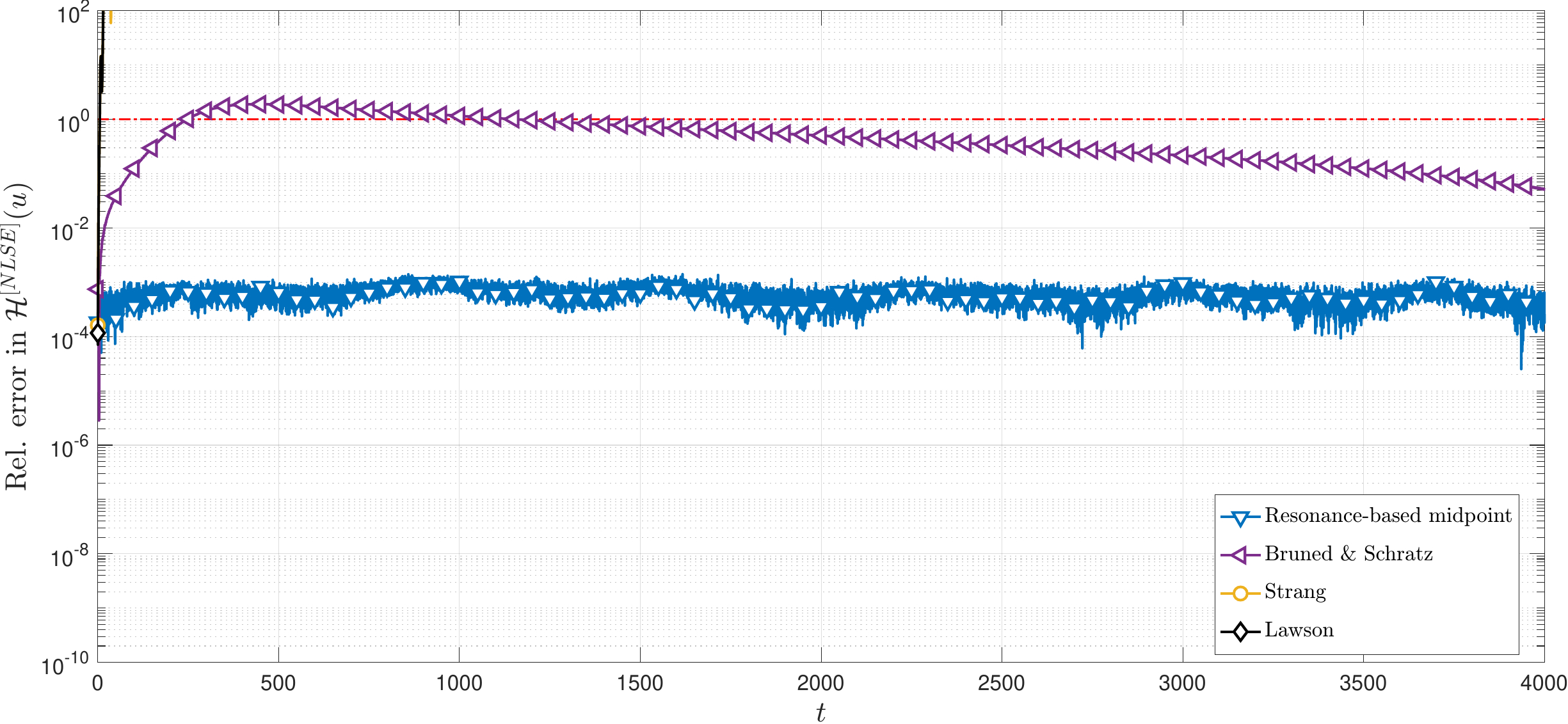}
 \caption{$M=256$.}
 \end{subfigure}
	\caption{Error in the Hamiltonian of the numerical solution, for $u_0\in H^2$, as per \eqref{eqn:law_for_random_initial_conditions} with $\vartheta=2$.}
	\label{fig:nlse_hamiltonianpreservationrough}
\end{figure}
\newpage

\subsection{KdV equation}

We will now perform similar numerical studies for the KdV case. We note that here a fair comparison with the literature is slightly more challenging, because existent structure preserving algorithms are not quite as competitive as for the NLSE. Nevertheless we will evaluate our resonance-based midpoint rule \eqref{eqn:resonance_based_midpoint_rule_u_KdV} against the following reference methods:
\begin{itemize}
    \item the second order low-regularity integrator introduced by \cite[Section~5.2]{bruned_schratz_2022}, denoted by `Bruned \& Schratz';
    \item the classical Strang splitting, which is a symplectic scheme for the KdV (cf. \cite{holden2011operator,holden2013operator}), denoted by `Strang'. Here we assume that the Burger's type nonlinearity is solved exactly (in our case with an auxiliary Runge--Kutta scheme of time step $\tau10^{-4}$) which is in itself a very expensive process and raises questions of the practical suitability of the Strang splitting in the KdV case;
    \item the second order momentum-preserving Lawson method introduced in \cite[Example~3.2]{celledoni2008symmetric}, denoted by `Lawson', which we adapted to the KdV nonlinearity. We note the method was originally designed and studied for the NLSE, but in its functional form can be adapted easily to the KdV equation. {Our reason for comparison against this method is that it indeed provides one of the most competitive structure preserving algorithms for the KdV equation currently available.}
\end{itemize}

In all of the numerical experiments reference solutions were again computed with `Bruned \& Schratz' with a reference time-step of $\tau_{ref}=10^{-6}$ and $M=2^{14}$ Fourier modes.

\subsubsection{Convergence properties}\label{sec:convergence_properties}

\begin{figure}[h!]
\centering
	\begin{subfigure}{0.495\textwidth}
		\centering
		\includegraphics[width=0.92\textwidth]{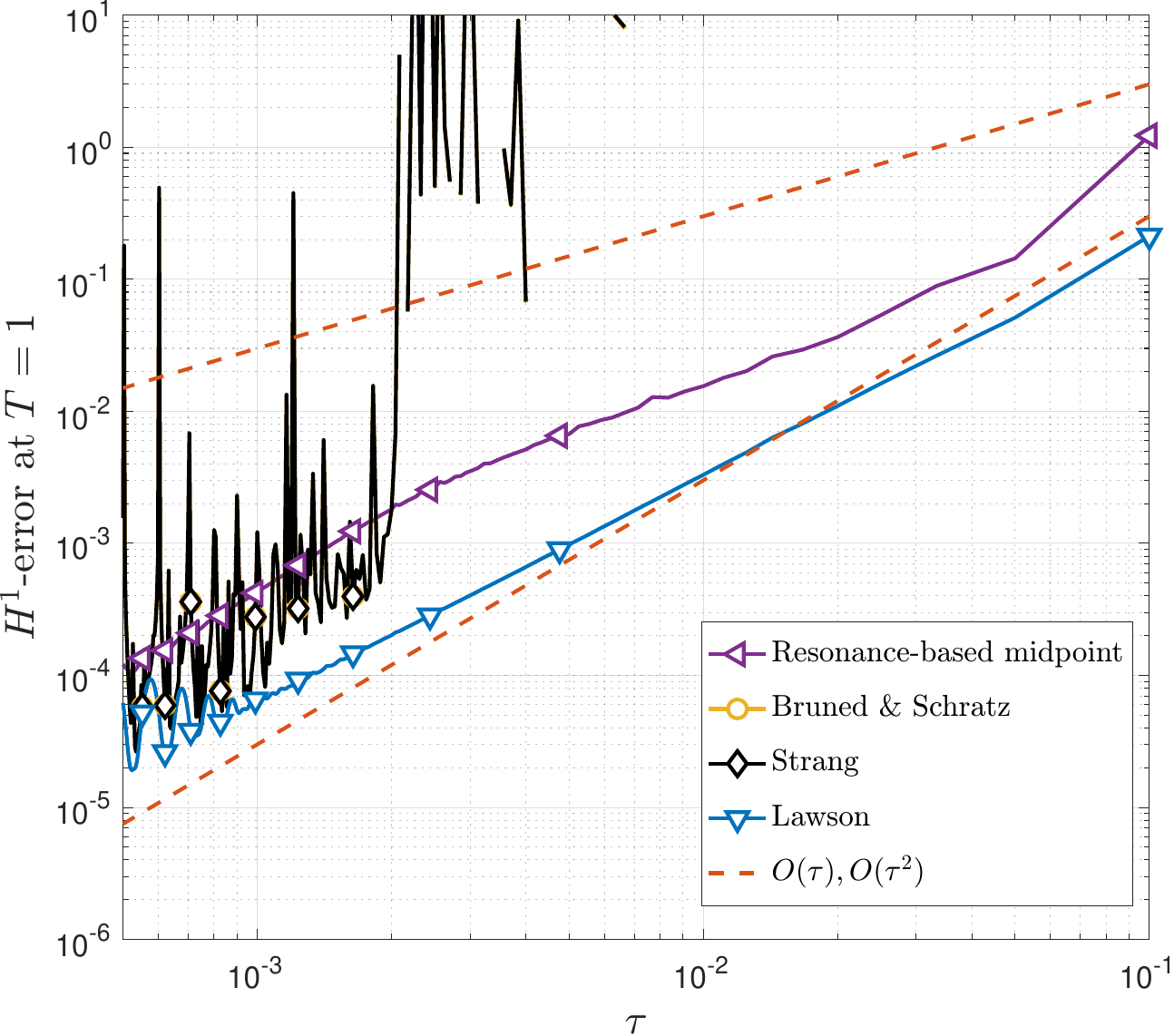}
  \caption{Initial data $u_0\in H^3$, as per \eqref{eqn:law_for_random_initial_conditions} with $\vartheta=3$.}
		\label{fig:time_step_vs_error_H2_data_kdv}
	\end{subfigure}
	\begin{subfigure}{0.495\textwidth}
		\centering
		\includegraphics[width=0.92\textwidth]{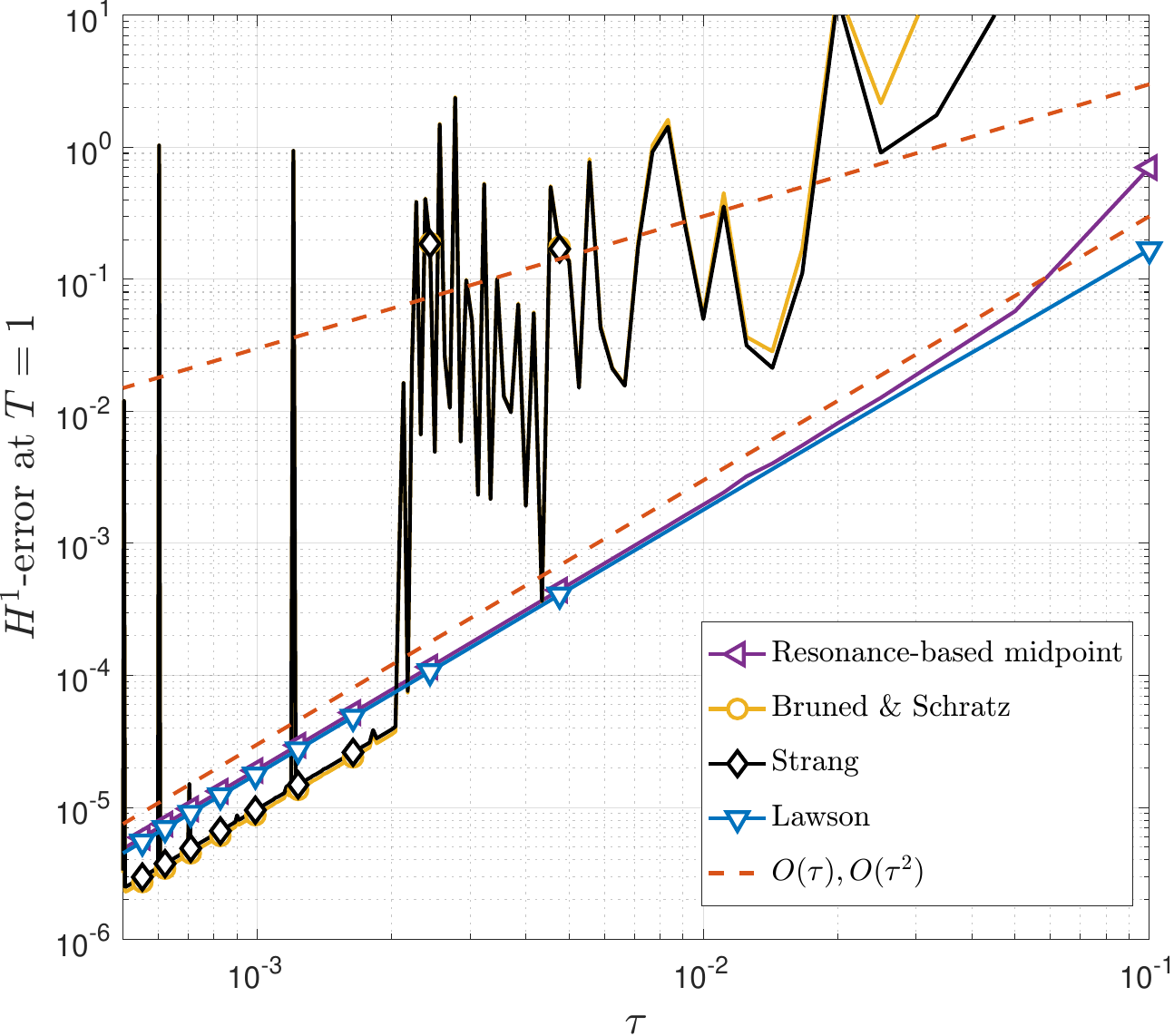}
  \caption{Initial data $u_0\in H^5$, as per \eqref{eqn:law_for_random_initial_conditions} with $\vartheta=5$.}
		\label{fig:time_step_vs_error_H2_data_kdv}
	\end{subfigure}\\
	\begin{subfigure}{0.495\textwidth}
		\centering
		\includegraphics[width=0.92\textwidth]{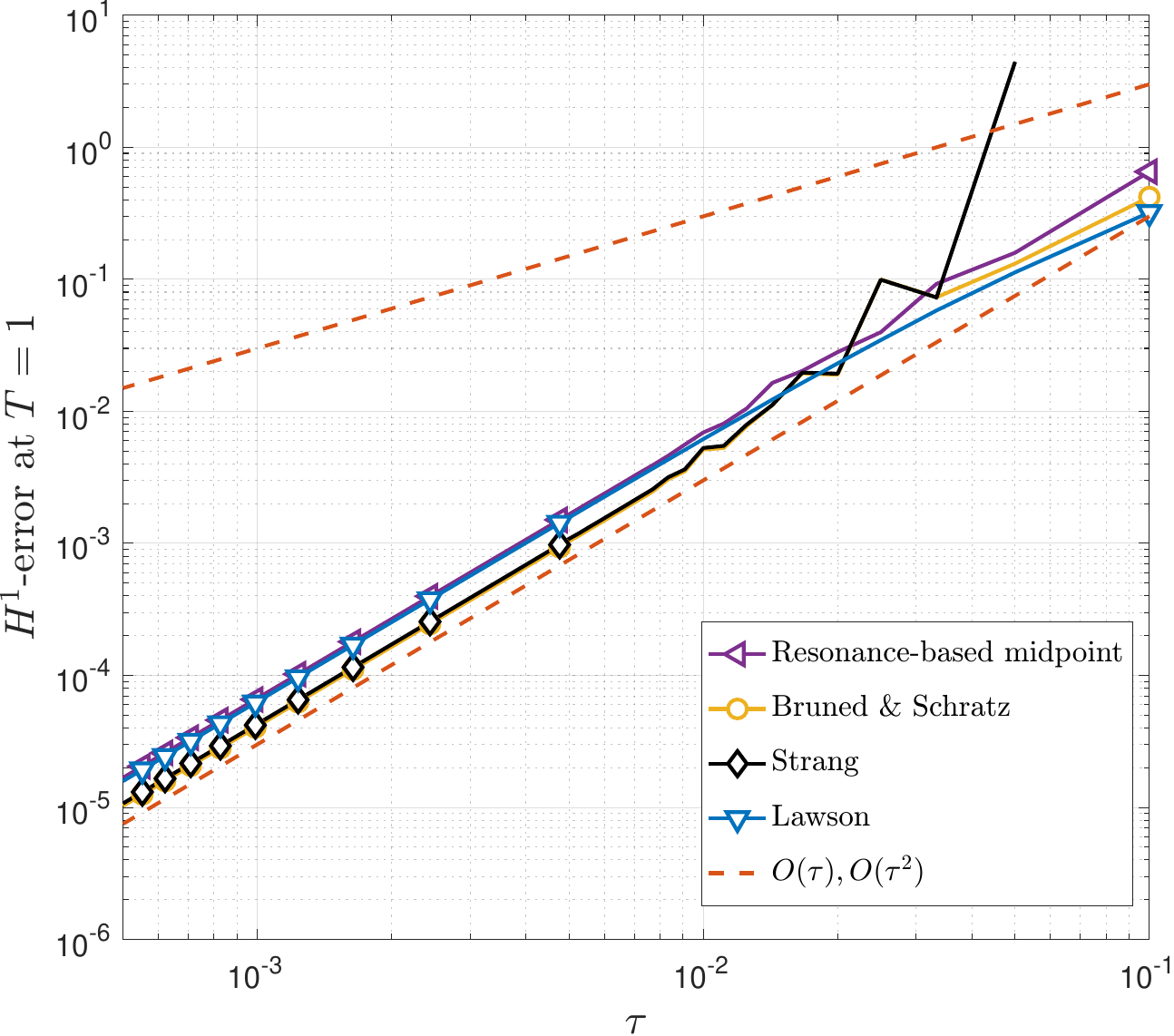}
		\caption{Initial data $u_0\in C^\infty$, as per \eqref{eqn:def_smooth_initial_data_numerics}.}
		\label{fig:time_step_vs_error_H2_data_kdv}
	\end{subfigure}
	\caption{Order plot measured in $H^1$, $M=64$.}
	\label{fig:convergence_properties_kdv}
\end{figure}

To begin with let us note that in practical experiments we found that Strang splitting and the aforementiond Lawson method both suffer from a CFL condition of the form $\tau\lesssim M^{-1}$, required to ensure stability and convergence of the methods. To provide a fair comparison we thus start our numerical discussion with a spectral discretisation with a relatively small number of Fourier modes, $M=64$, which can be seen in Figure~\ref{fig:convergence_properties_kdv}. Here we plot the convergence graphs of the methods for solutions of various levels of regularity, $u_0\in H^3,u_0\in H^5$ and $u_0\in C^\infty$ according to \eqref{eqn:law_for_random_initial_conditions} \& \eqref{eqn:def_smooth_initial_data_numerics} but rescaled such that $\|u_0\|_{L^2}=0.1$. This rescaling manifests itself essentially just as a rescaling of the nonlinearity which allows us to look at structure preservation properties over times comparable to the NLSE in Section~\ref{sec:numerical_experiments_structure_preservation_kdv}. Had we chosen the scaling $\|u_0\|_{L^2}=1$ the qualitative nature of our results would not change, but the time interval over which we see preservation of structure would be much shorter - indeed the resonance-based midpoint rule was found to still preserve the momentum and Hamiltonian over longer times than both reference methods in the case of $O(1)$ initial data (strong nonlinearity).

\begin{figure}[h!]
\centering
	\begin{subfigure}{0.495\textwidth}
		\centering
		\includegraphics[width=0.92\textwidth]{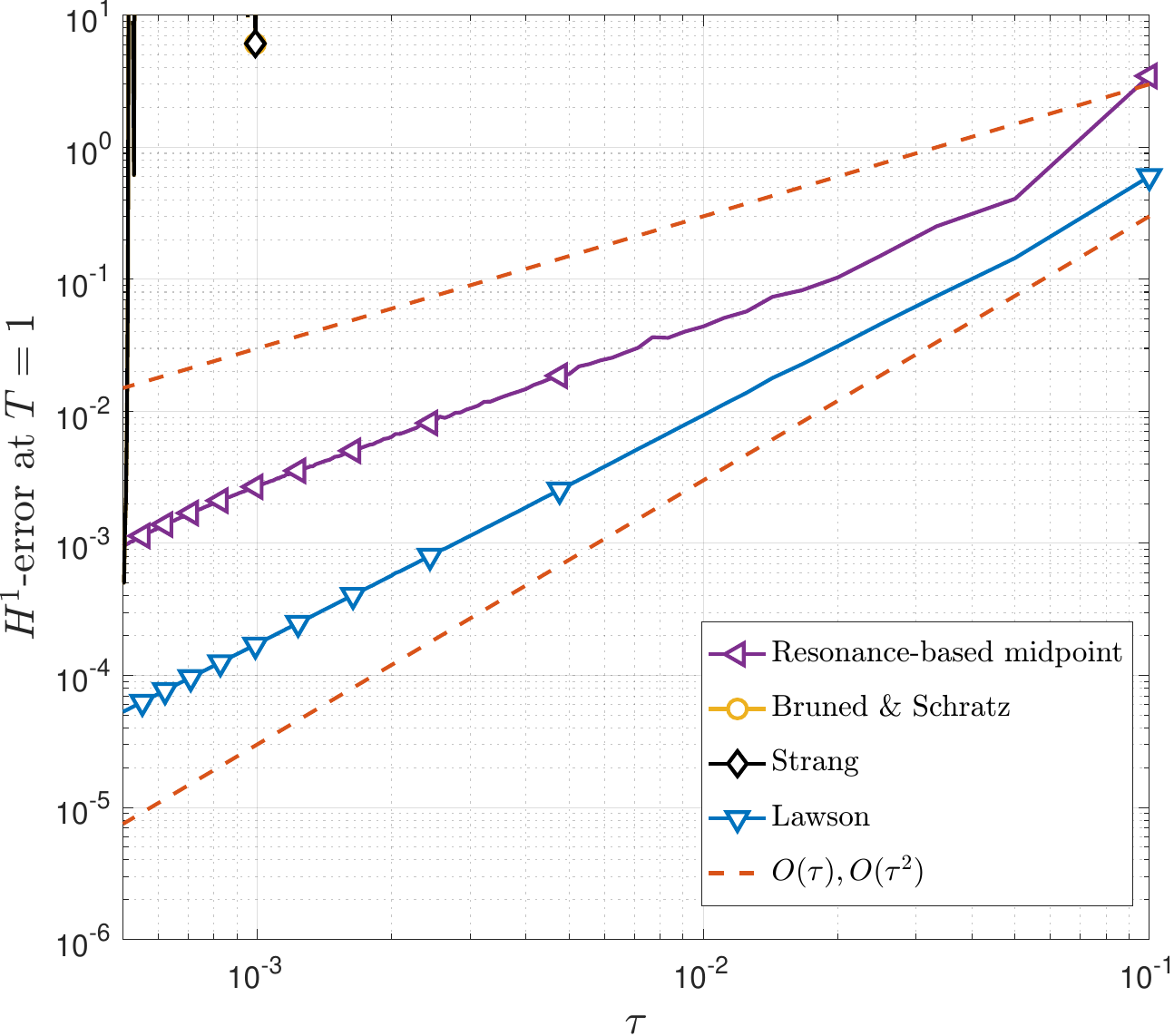}
  \caption{Initial data $u_0\in H^3$, as per \eqref{eqn:law_for_random_initial_conditions} with $\vartheta=3$.}
	\end{subfigure}
	\begin{subfigure}{0.495\textwidth}
		\centering
		\includegraphics[width=0.92\textwidth]{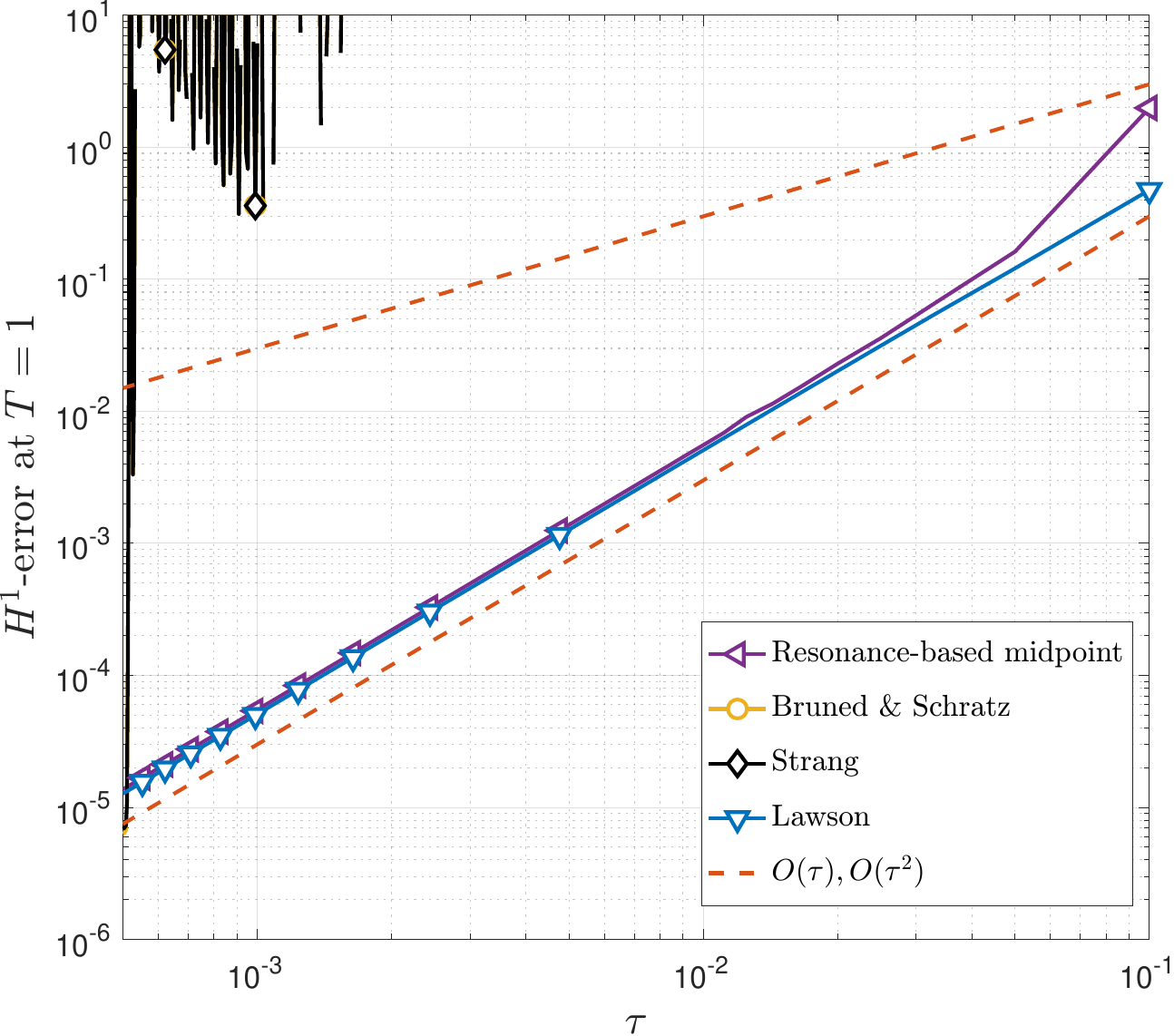}
  \caption{Initial data $u_0\in H^5$, as per \eqref{eqn:law_for_random_initial_conditions} with $\vartheta=5$.}
	\end{subfigure}\\
	\begin{subfigure}{0.495\textwidth}
		\centering
		\includegraphics[width=0.92\textwidth]{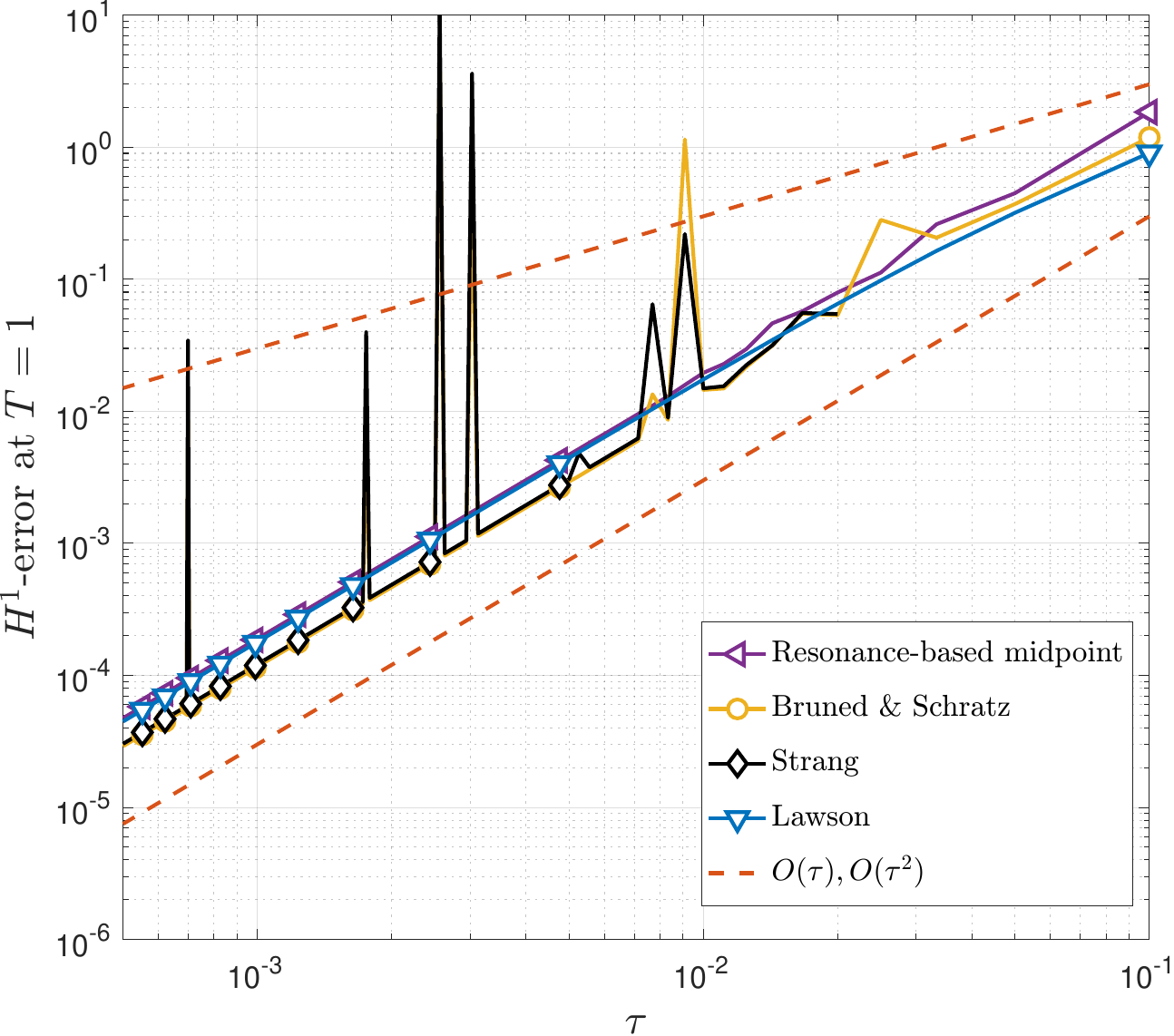}
		\caption{Initial data $u_0\in C^\infty$, as per \eqref{eqn:def_smooth_initial_data_numerics}.}
	\end{subfigure}
	\caption{Order plot measured in $H^1$, $M=128$ and $\|u_0\|_{L^2}=0.1$.}
	\label{fig:convergence_properties_kdv_M128}
\end{figure}

We see that our method converges at the predicted rates from Theorems~\ref{thm:global_error_in_H1_first_order} \& \ref{thm:global_error_in_H1_second_order} and indeed for solutions in $H^3$ appears to converge even faster than $\mathcal{O}(\tau)$. In this sense the method is able to outperform prior work including the resonance-based scheme introduced in \cite{bruned_schratz_2022}. The classical integrators `Strang' and `Lawson' perform poorly for low-regularity data, but converge as expected for smooth solutions.

However, as soon as we introduce more Fourier modes (cf. Figure~\ref{fig:convergence_properties_kdv_M128} with $M=128$) the behaviour of both classical integrators significantly worsens (indicative of the CFL requirement) whereas our resonance midpoint rule exhibits the same favourable convergence behaviour.

\newpage\ \newpage\subsubsection{Structure preservation properties}\label{sec:numerical_experiments_structure_preservation_kdv}
Having verified the convergence properties of our proposed numerical scheme, we now study its structure preservation properties for the case $M=64, \tau=0.02$. In the first instance, in {Figure~\ref{fig:momentumpreservationkdvsmooth}}, we look at the momentum \eqref{eqn:quadratic_first_integral_kdv} for smooth solutions which we observe to be preserved nearly exactly in our resonance-based midpoint rule, as well as in `Strang' and (according to the analysis in \cite[Proposition~3.1]{celledoni2008symmetric}) the `Lawson' method.
\begin{figure}[h!]
	\centering
	\includegraphics[width=0.9\linewidth]{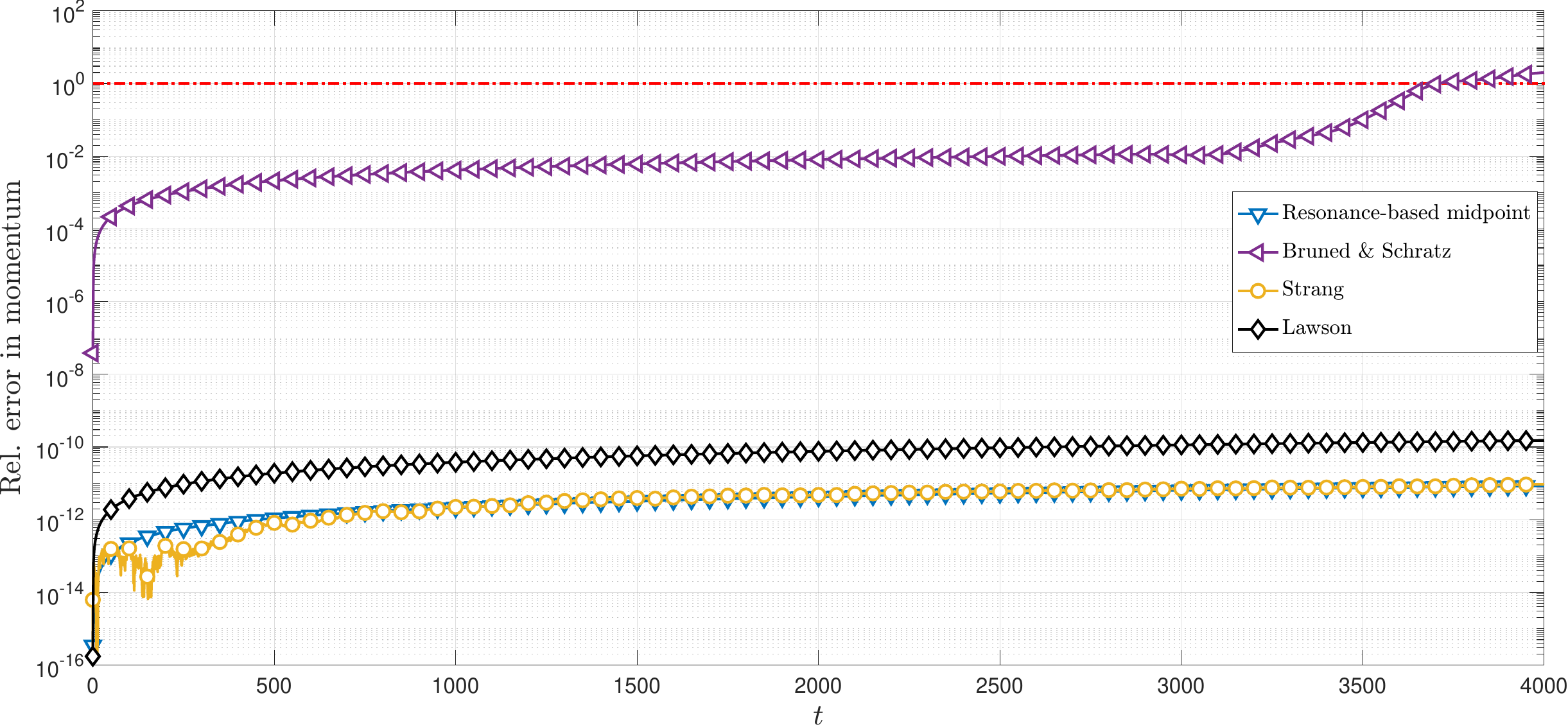}
	\caption{Relative error in the momentum with $u_0\in C^\infty$, as per \eqref{eqn:def_smooth_initial_data_numerics}, $\|u_0\|_{L^2}=0.1$, and $M=64$.}
	\label{fig:momentumpreservationkdvsmooth}
\end{figure}

\newpage However, the picture drastically changes when we move to low-regularity solutions in Figure~\ref{fig:momentumpreservationkdv}, where we observe a breakdown in the long-term preservation of this quantity in both the Lawson method and Strang splitting, whilst in our resonance-based midpoint rule the momentum remains preserved nearly exactly. The slightly larger error observed here is accounted for by the spatial discretisation error given the small number of Fourier modes in our spatial discretisation (which we chose to ensure the classical methods would provide competitive results as well), but the important observation is that this error does not grow in $t$.

\begin{figure}[h!]
	\centering
 \begin{subfigure}{1\textwidth}
 \centering
     \includegraphics[width=0.9\linewidth]{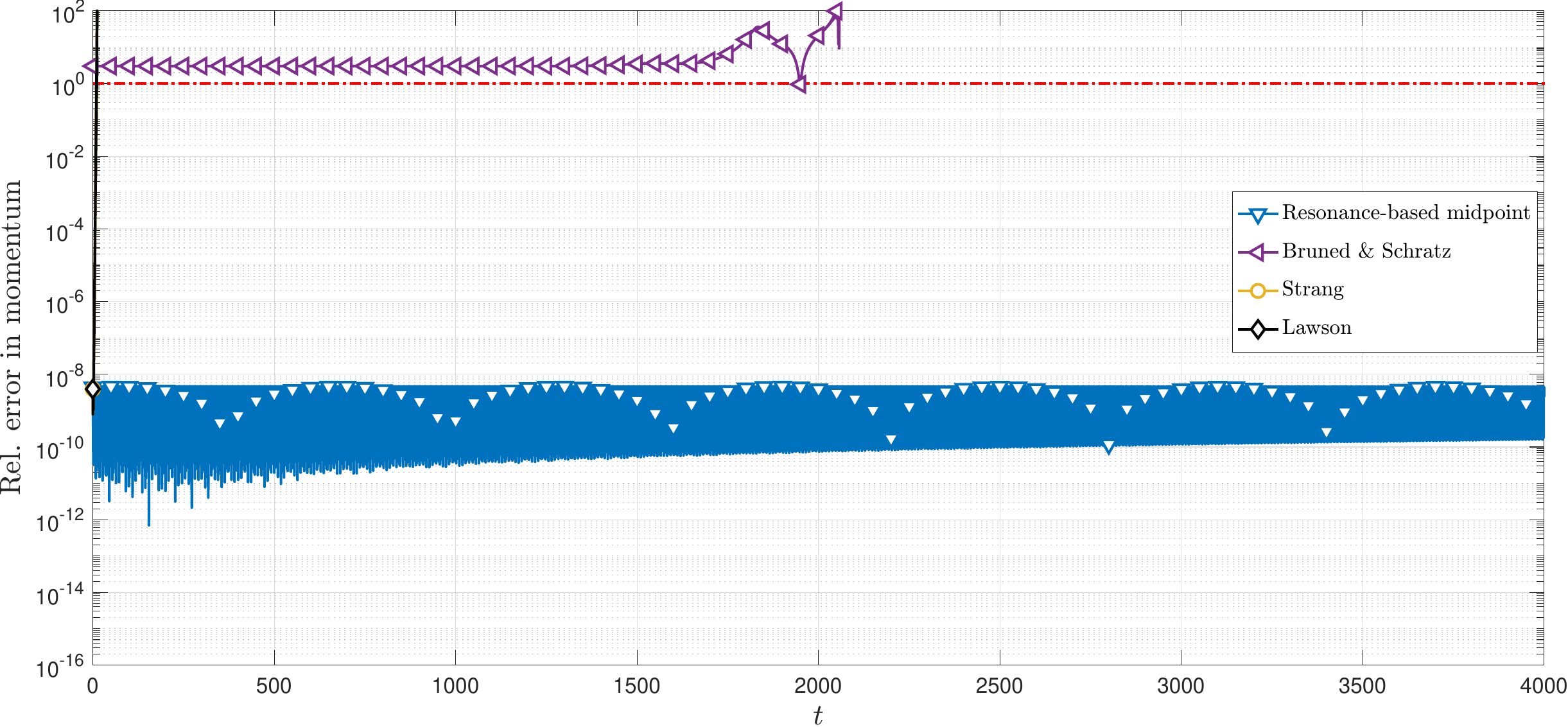}
	\caption{Full time interval $t\in[0,4000]$.}
	\label{fig:momentumpreservationkdvlowregularity}
 \end{subfigure}
 \begin{subfigure}{1\textwidth}
	\centering
	\includegraphics[width=0.9\linewidth]{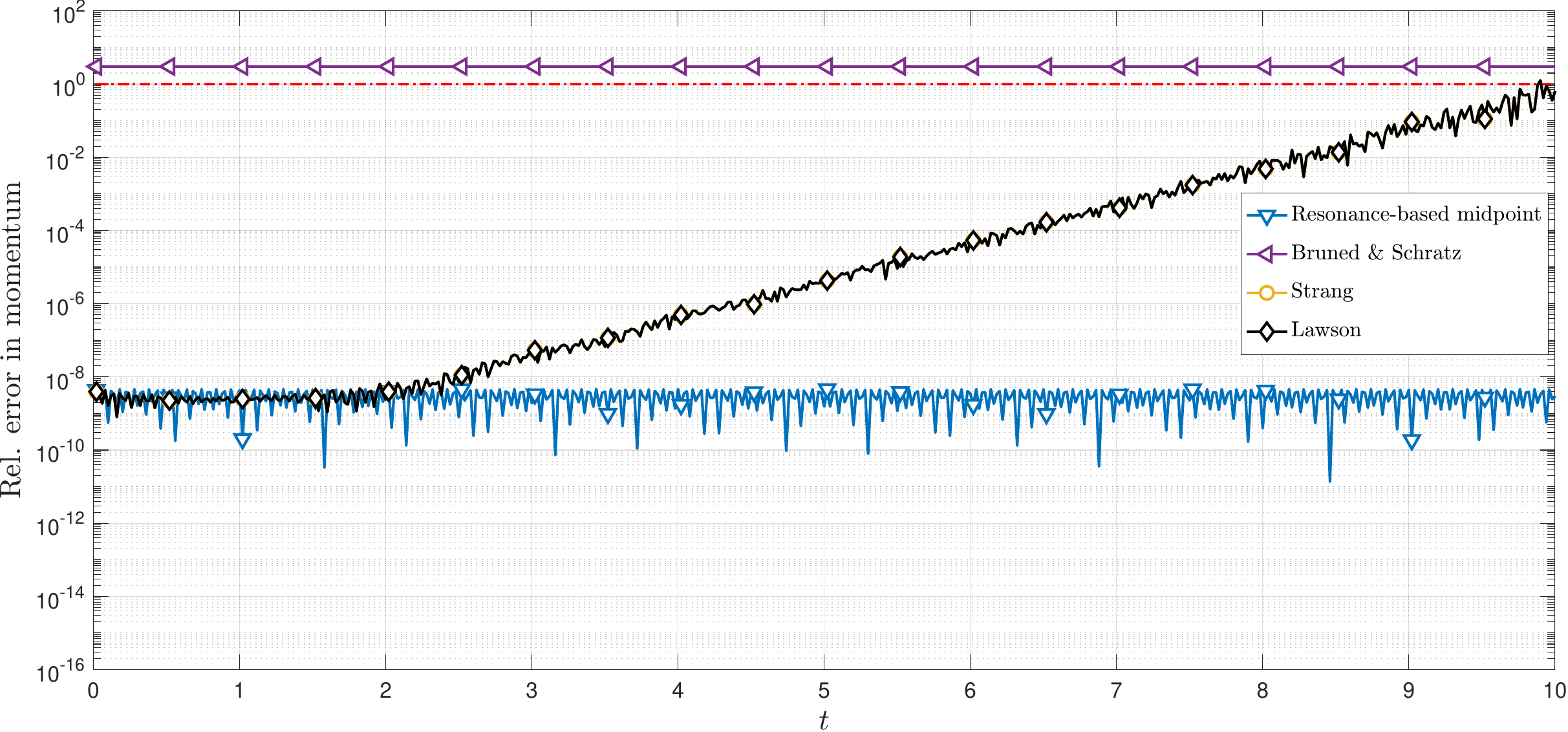}
	\caption{Zoom into the initial time interval $t\in[0,10]$.}
	\label{fig:momentumpreservationkdvlowregularity_zoom}
 \end{subfigure}
	\caption{Relative error in the momentum with $u_0\in H^3$, as per \eqref{eqn:law_for_random_initial_conditions} with $\vartheta=3$, $\|u_0\|_{L^2}=0.1$, and $M=64$.}
 \label{fig:momentumpreservationkdv}
\end{figure}

\newpage \ \vspace{-0.75cm}\\ Finally, we can again consider the Hamiltonian which is preserved under the exact flow of the KdV equation,
\begin{align*}
\mathcal{H}^{[KdV]}(u)=-\frac{1}{2}\int_{\mathbb{T}}3u_x^2+u^3\dd x.
\end{align*}
Like for the NLSE there are very few theoretical guarantees on the long-time preservation of this quantity under symplectic integrators for PDEs. Nevertheless, we found in numerical experiments that for a large number of time steps our method is able to preserve the Hamiltonian well over long times both in the smooth and rough regimes, while the Lawson method and the Strang splitting exhibit similar breakdown of the preservation properties for rough data as for the momentum. A representative example of this behaviour is shown in Figures~\ref{fig:hampreservationsmooth} \& \ref{fig:hampreservationlowregularity}.

\begin{figure}[h!]
	\centering
	\includegraphics[width=0.8\linewidth]{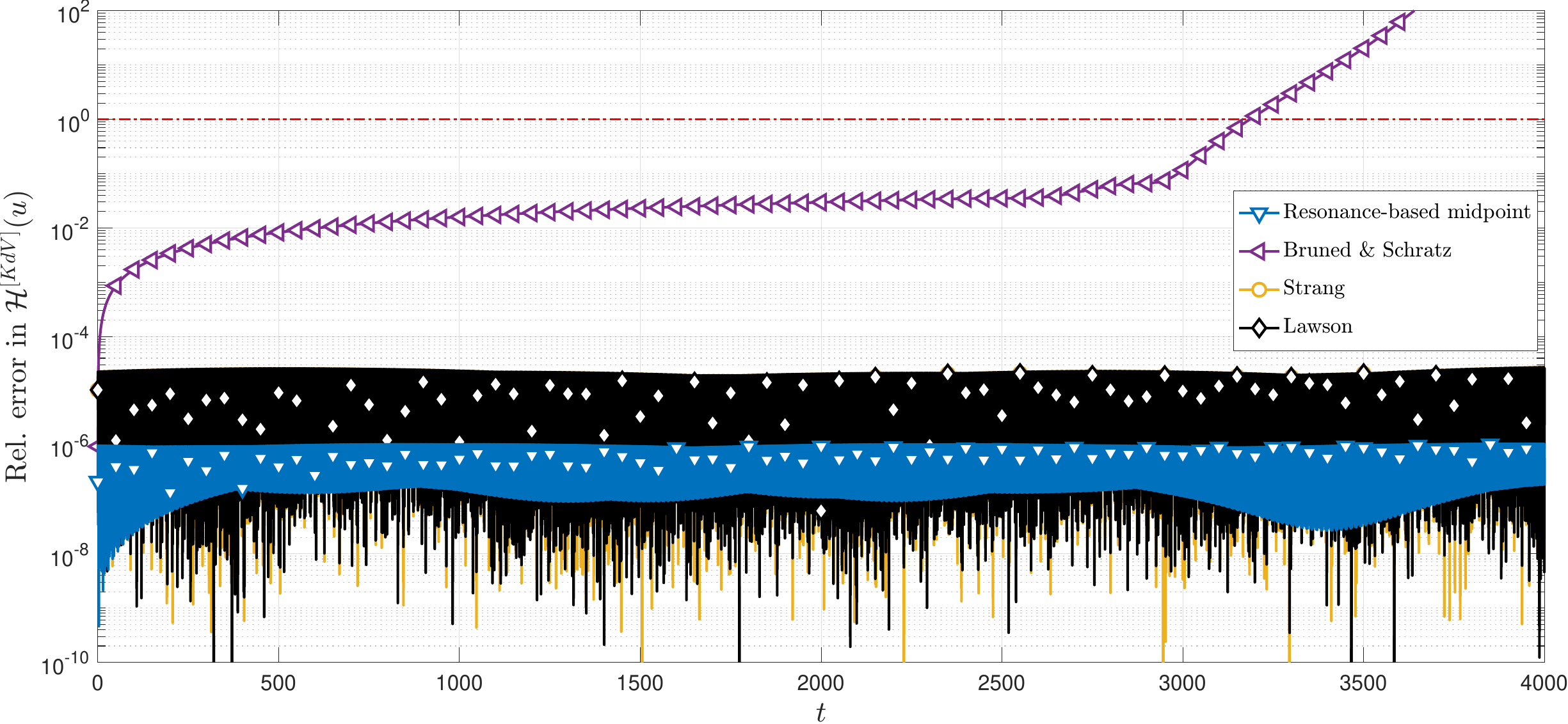}\vspace{-0.3cm}
	\caption{Relative error in the Hamiltonian with $u_0\in C^\infty(\mathbb{T})$ as per \eqref{eqn:def_smooth_initial_data_numerics}, $\|u_0\|_{L^2}=0.1$ and $M=64$.}\vspace{-0.1cm}
	\label{fig:hampreservationsmooth}
\end{figure}

\begin{figure}[h!]
	\centering
	\begin{subfigure}{1\textwidth}
	\centering
	\includegraphics[width=0.8\linewidth]{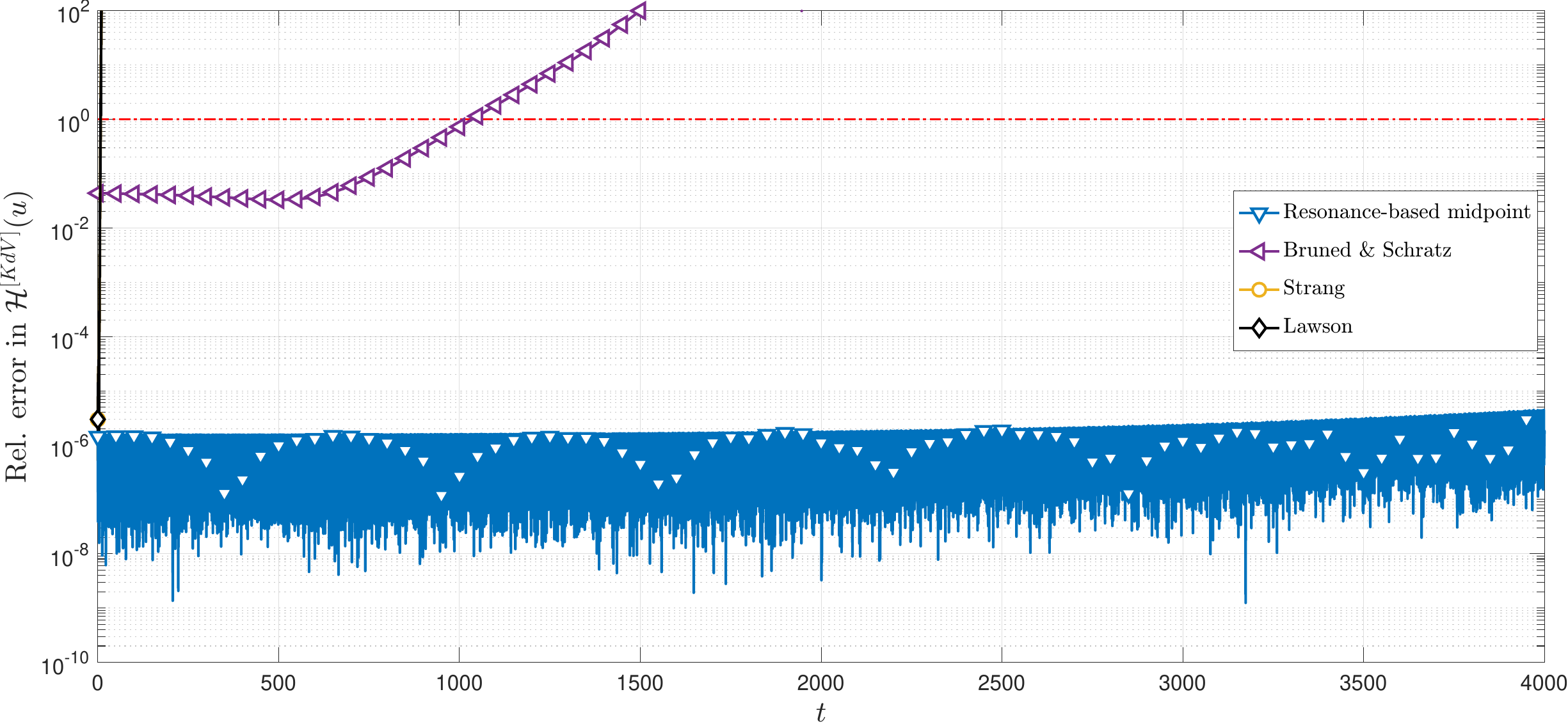}\vspace{-0.1cm}
	\caption{Full time interval $t\in[0,4000]$.}
	\end{subfigure}
	\begin{subfigure}{1\textwidth}
	\centering
	\includegraphics[width=0.8\linewidth]{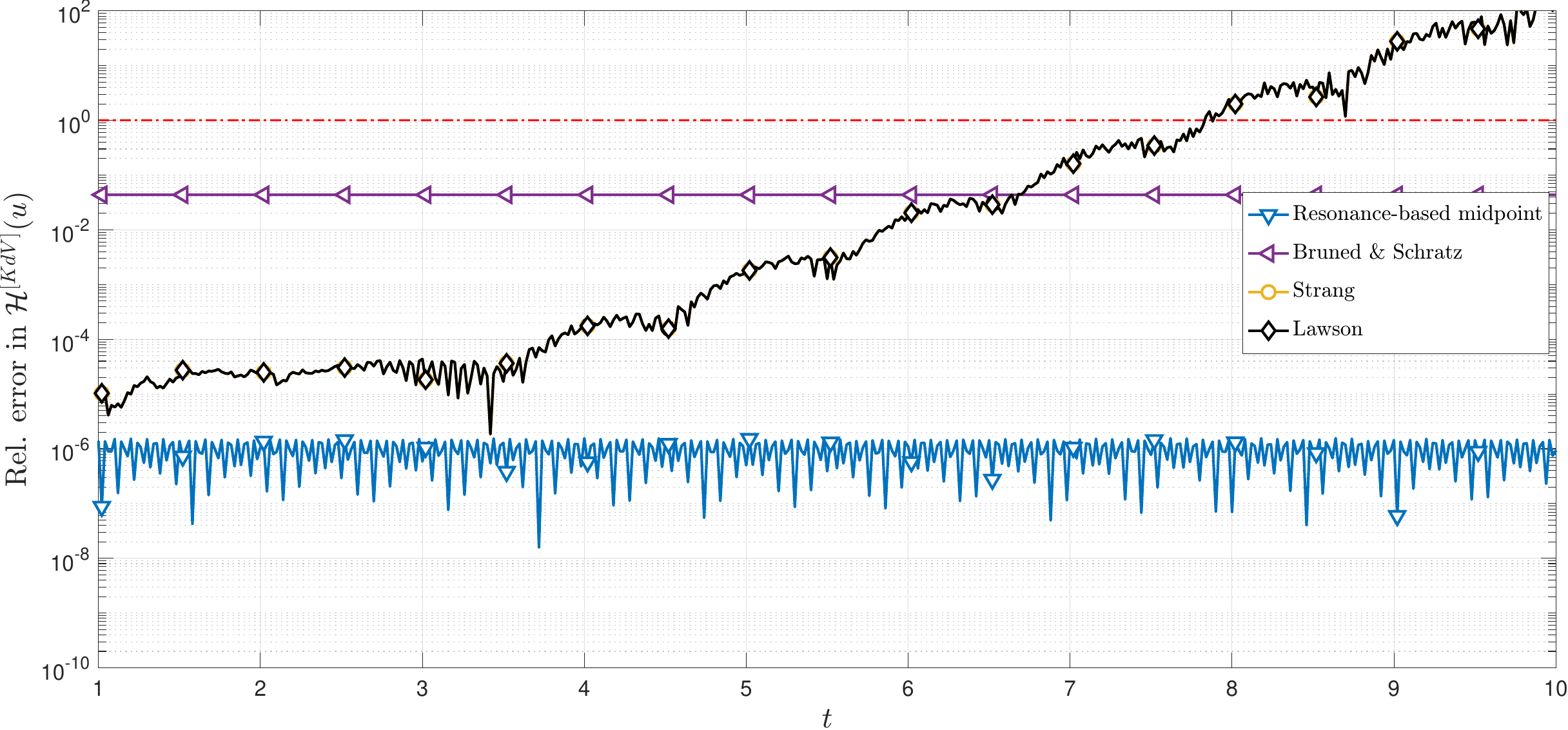}\vspace{-0.1cm}
	\caption{Zoom into the initial time interval $t\in[0,10]$.}
	\label{fig:hampreservationlowregularity_sub}
	\end{subfigure}\vspace{-0.3cm}
	\caption{Relative error in the Hamiltonian with $u_0\in H^3$ (\eqref{eqn:law_for_random_initial_conditions} with $\vartheta=3$), $\|u_0\|_{L^2}=0.1$ and $M=64$.}
	\label{fig:hampreservationlowregularity}
\end{figure}
\section{Conclusions}\label{sec:conclusions}
In this work we {present} a novel point of view for the construction of low-regularity integrators for {a class of dispersive nonlinear partial differential equations}. This novel point of view allows us to design a class of methods which incorporates a much larger number of degrees of freedom at similar low-regularity convergence properties as compared to prior work. We call this class of methods RK resonance-based schemes, and provide a number of examples of novel symplectic low-regularity integrators for {two canonical models of infinite-dimensional dispersive nonlinear Hamiltonian systems, the Korteweg de Vries and the nonlinear Schr\"odinger equations}. A particularly interesting example is the resonance-based midpoint rule, which we study in further detail by providing a rigorous convergence analysis under low regularity assumptions and confirming the favourable properties of this novel scheme as compared to state-of-the-art methods both from low-regularity integration and geometric numerical integration in numerical experiments.

We believe this work takes a significant step towards reconciling ideas from geometric numerical integration with constructions of resonance-based schemes in that, to the best of our knowledge, this work constitutes the first time the notion of symplecticity could be captured in low-regularity integrators for dispersive nonlinear equations. {As indicated in Section~\ref{sec:RK_res_schemes_for_general_eqns} \& \ref{sec:symplectic_kernel_approximations} we outline a foundation for the extension of these ideas to} a wider class of dispersive infinite-dimensional Hamiltoninan systems. {We note that a particular challenge is posed by the construction of symplectic kernel approximations as introduced in Section~\ref{sec:symplectic_kernel_approximations_NLSE} for different types of dispersive nonlinear equations. However,} based on the promising developments in related recent work \cite{shen2019geometric} which was able to construct symplectic exponential integrators for a large class of semilinear Poisson systems, and seeing as resonance-based schemes are nowadays available for a wide range of equations \cite{bruned_schratz_2022}, we expect that an extension of our construction is indeed feasible {and this will be subject to future work}.  {As outlined in Section~\ref{sec:general_idea_convergence_analysis}, an} additional problem to be studied in future research is the design of a structured {local error analysis} of the presented methods {in the low-regularity regime} (cf. \cite{luan2013exponential}), {which might require the development of a designated approach similar to B-series potentially combining ideas from \cite{bruned_schratz_2022} and \cite{luan2013exponential}.}

\section*{Acknowledgements}
The authors would like to thank Valeria Banica (Sorbonne Universit\'e), Yvonne Alama Bronsard (Sorbonne Universit\'e), Yvain Bruned (Université de Lorraine), Erwan Faou {(INRIA Bretagne Atlantique \& Universit\'e de Rennes I)}, Felice Iandoli (Università della Calabria) and Brynjulf Owren (Norwegian University of Science and Technology) for several interesting and helpful discussions. {We are particularly grateful to Buyang Li (The Hong Kong Polytechnic University) and Shu Ma (City University of Hong Kong) for interesting discussions that helped reduce the regularity assumptions required in Theorems~\ref{thm:L2_pres_condition_direct_flow} \& \ref{thm:preservation_of_symplectic_form}. Furthermore,} we would also like to express our gratitude to the constructive comments and feedback received from anonymous reviewers of an earlier version of this manuscript.} {Both authors gratefully acknowledge funding from the European Research Council (ERC) under the European
Union’s Horizon 2020 research and innovation programme (grant agreement No.\ 850941). GM additionally gratefully acknowledges funding from the European Union’s Horizon Europe research and innovation programme under the Marie Skłodowska--Curie grant agreement No.\ 101064261.}
\bibliographystyle{siam}     
\nocite{*}  
\bibliography{biblio1}
\addcontentsline{toc}{section}{References}
\begin{appendix}
\section{Proof of Proposition~\ref{prop:local_error_non-symplectic_second_order_scheme}}\label{app:proof_local_error_non-symplectic_second_order_scheme}
We recall the statement of Proposition~\ref{prop:local_error_non-symplectic_second_order_scheme}:
\begin{proposition}  {Fix $R>0,s>1/2,\gamma\in[0,2]$, and let} us denote by {$\tau\mapsto\phi_{\tau}(u(t_n))\in H^{s+2}$} the solution to \eqref{eqn:Cauchy_problem_NLS} with initial condition {$u(t_n)\in H^{s+2}$} and denote by $\Phi_{\tau}$ the time-stepping scheme \eqref{eqn:NLSE_non-symplectic_second_order_integrator}{. Then} there is a $\tau_R>0$ such that for all $\tau\in [0,\tau_R)$ we have
whenever $\sup_{t\in[0,\tau]}\|{\phi_{t}}(u(t_n))\|_{H^{s+2}}<R$ then
\begin{align*}
	\left\|{\phi_\tau}(u(t_n))-\Phi_{\tau}(u(t_n))\right\|_{H^s}\leq c_R\tau^{3} 
	\end{align*}
for some constant $c_R>0$ depending on $R>0,s$.
\end{proposition}
The proof relies on the following lemma concerning the solution of the implicit equations which can be proved analogously to Theorem~\ref{thm:sln_of_implicit_eqns_NLSE}, where we define here the map $\mathcal{S}$ by
\begin{align*}
\widehat{\left(\mathcal{S}(w)\right)}_k&=e^{-ik^2\tau}\hat{u}^n_k+(-i\mu)e^{-ik^2\tau}\sum_{k+k_1=k_2+k_3}\int_0^\tau \left[e^{-2iskk_1}+e^{2isk_2k_3}-1\right]\\
    &\quad\quad\quad\quad\quad\quad\quad\quad\quad\quad\quad\quad\quad\left[\overline{\hat{u}_{k_1}^n}\hat{u}_{k_2}^n\hat{u}_{k_3}^n+\frac{s}{\tau}\left(e^{i\tau(k_2^2+k_3^2-k_1^2)}\overline{\hat{w}_{k_1}}\hat{w}_{k_2}\hat{w}_{k_3}-\overline{\hat{u}_{k_1}^n}\hat{u}_{k_2}^n\hat{u}_{k_3}^n\right)\right]\dd s.
\end{align*}

    \begin{lemma}\label{lem:appendix_sln_of_implicit_eqns_NLSE}
	Let $R>0$ and $s>1/2$. Then there is a constant $\tau_R$ such that for all $\tau\in(0,\tau_R)$ and any $u^n\in B_R(H^{s})=\{u\in H^s\,\vert\, \|u\|_{H^s}<R\}$ we have that $u^{n+1}$, the exact solution of \eqref{eqn:resonance_based_midpoint_rule_u_NLSE}, is given by the following limit in $H^s$:
	\begin{align}\label{eqn:appendix_limit_expression_for_u^{k+1}}
		u^{n+1}=\lim_{j\rightarrow \infty}\mathcal{S}^{(j)}(e^{i\tau{\partial_x^2}}u^n),\quad \text{where\ }\mathcal{S}^{(j)}(u)=\underbrace{\mathcal{S}\circ\cdots\circ\mathcal{S}}_{j\text{-times}}(u).
	\end{align}
Moreover, we have the estimate
\begin{align}\label{eqn:appendix_estimate_implicit_equation_forward_distance}
	\left\|u^{n+1}-e^{i\tau{\partial_x^2}}u^{n}\right\|_{H^s}\leq \tau \tilde{C}_R,
\end{align}
for some $\tilde{C}_R$ which depends only on $R$ (and $s$).
\end{lemma}

\begin{proof}[Proof of Proposition~\ref{prop:local_error_non-symplectic_second_order_scheme}] By \eqref{eqn:twisted_NLS} and \eqref{eqn:appendix_estimate_implicit_equation_forward_distance} if follows that under the same assumptions of Lemma~\ref{lem:appendix_sln_of_implicit_eqns_NLSE} we have
\begin{align}\label{eqn:app1_first_estimate}
    \left\|\Phi_{\tau}(u(t_n))-\phi_{\tau}(u(t_n))\right\|_{H^s}\leq \tau C_R, \ \forall \tau\in(0,\tau_R),
\end{align}
for some constants $\tau_R,C_R>0$ depending only on $s,R>0$. Thus we have 
\begin{align}\label{eqn:app1_auxiliary_estimate1}
     \left\|\Phi_{\tau}(u(t_n))-e^{i\tau\partial_x^2}u(t_n)-\mathcal{S}(\phi_{\tau}(u(t_n)))\right\|_{H^s}\leq \tau^2 C_R,\ \forall\tau\in(0,\tau_R),
\end{align}
where $\tau_R,C_R>0$ are again two constants which depend only on $s,R>0$, but which may be of different value than above. Moreover we have, writing $v(t)=\exp(-it\partial_x^2)\phi_t(u(t_n))$,
\begin{align*}
    &\left\|\Phi_{\tau}(u(t_n))-e^{i\tau\partial_x^2}u(t_n)-\mathcal{S}(\phi_{\tau}(u(t_n)))\right\|_{H^s}^2\\
    &\quad=|\mu|^2\sum_{k\in\mathbb{Z}}\langle k\rangle^{2s}\left|\sum_{k+k_1=k_2+k_3}\int_0^\tau\left[e^{-2iskk_k+2isk_2k_3}\overline{\hat{v}_{k_1}(s)}\hat{v}_{k_2}(s)\hat{v}_{k_3}(s)\right]\dd s\right.\\
    &\quad\quad\quad\quad-\left.\int_0^\tau \left[e^{-2iskk_1}+e^{2isk_2k_3}-1\right]\left[\overline{\hat{v}_{k_1}(0)}\hat{v}_{k_2}(0)\hat{v}_{k_3}(0)\right.\right.\\
    &\quad\quad\quad\quad\quad\quad\quad\left.\left.+\frac{s}{\tau}\left(e^{i\tau(k_2^2+k_3^2-k_1^2)}\overline{\hat{v}_{k_1}(\tau)}\hat{v}_{k_2}(\tau)\hat{v}_{k_3}(\tau)-\overline{\hat{v}_{k_1}(0)}\hat{v}_{k_2}(\tau)\hat{v}_{k_3}(\tau)\right)\right]\dd s\right|^2.
\end{align*}
Combining \eqref{eqn:local_error_second_order_kernel_approx_NLSE} and \eqref{eqn:twisted_NLS} we find
\begin{align}\label{eqn:app1_auxiliary_estimate2}
    \left\|\Phi_{\tau}(u(t_n))-e^{i\tau\partial_x^2}u(t_n)-\mathcal{S}(\phi_{\tau}(u(t_n)))\right\|_{H^s}^2\leq \tau^3 C_R,\ \forall\tau\in(0,\tau_R),
\end{align}
where $\tau_R,C_R>0$ are again two constants which depend only on $s,R>0$. \eqref{eqn:app1_auxiliary_estimate1} \& \eqref{eqn:app1_auxiliary_estimate2} together imply that 
\begin{align}\label{eqn:app1_second_estimate}
    \left\|\Phi_{\tau}(u(t_n))-\phi_{\tau}(u(t_n))\right\|_{H^s}\leq \tau^2 C_R, \ \forall \tau\in(0,\tau_R),
\end{align}
thus improving on \eqref{eqn:app1_first_estimate} by a factor of $\tau$. Iterating this process once again but replacing \eqref{eqn:app1_first_estimate} by \eqref{eqn:app1_second_estimate} yields the desired result.
\end{proof}

\section{Proof of Lemma~\ref{lem:stability_lemma_1}}\label{app:proof_of_stability_lemma_1}
For completeness we recall the statement of the lemma.
\begin{lemma}Let us introduce the notation
	\begin{align*}
		\mathcal{G}(t_n,\tau,\tilde{v}):=\frac{1}{6}\e^{(t_n+\tau)\partial_x^3}\left(\e^{-(t_n+\tau)\partial_x^3}\partial_x^{-1}\tilde{v}\right)^2-\frac{1}{6}\e^{t_n\partial_x^3}\left(\e^{-t_n\partial_x^3}\partial_x^{-1}\tilde{v}\right)^2
	\end{align*}
	Then, for $l=1,2,3$, there is a continuous function $M_l:\mathbb{R}_{\geq0}\times\mathbb{R}_{\geq0}\rightarrow \mathbb{R}_{\geq0}$ such that
	\begin{align*}
		\| \mathcal{G}(t_n,\tau,f)-\mathcal{G}(t_n,\tau,g)\|_{H^l}\leq {\tau^{\frac{1}{2}}} M_l\left(\|f\|_{H^l},\|g\|_{H^l}\right)\|f-g\|_{H^l}.
	\end{align*}
\end{lemma}

\begin{proof}
	As highlighted in {Section}~\ref{sec:properties_of_implicit_equations} the cases $l=1,2$ are treated in \cite[Eq.~(38) \& Lemma~2.4]{hofmanova2017exponential}. It remains to prove the case $l=3$. For this we proceed similarly to \cite[Lemma~2.4]{hofmanova2017exponential} and write
{	\begin{align*}
&\| \mathcal{G}(t_n,\tau,f)-\mathcal{G}(t_n,\tau,g)\|_{H^3}^2\\
&=\frac{1}{36}\| \partial_x^3\e^{-(t_n+\tau)\partial_x^3}\left[\mathcal{G}(t_n,\tau,f)-\mathcal{G}(t_n,\tau,g)\right]\|_{L^2}^2\\
&=\!\frac{1}{36}\!\underbrace{\left\langle\!\! \partial_x^3\!\!\left[\!\left(\e^{-\tau\partial_x^3}\partial_x^{-1}\tilde{f}\right)^2\!\!\!-\!\left(\!\e^{-\tau\partial_x^3}\partial_x^{-1}\tilde{g}\right)^2\right]\!\!-\!\partial_x^3\e^{-\tau\partial_x^3}\!\!\left[\left(\partial_x^{-1}\tilde{f}\right)^2\!\!\!-\!\!\left(\partial_x^{-1}\tilde{g}\right)^2\right]\!,\!\partial_x^3\!\!\left[\!\left(\e^{-\tau\partial_x^3}\partial_x^{-1}\!\!\tilde{f}\right)^2\!\!\!-\!\!\left(\e^{-\tau\partial_x^3}\partial_x^{-1}\tilde{g}\right)^2\right]\!\right\rangle}_{=:A_1}\\
&\quad-\frac{1}{36}\underbrace{\left\langle \!\partial_x^3\!\left[\left(\e^{-\tau\partial_x^3}\partial_x^{-1}\!\tilde{f}\right)^2\!\!\!-\!\left(\e^{-\tau\partial_x^3}\partial_x^{-1}\!\tilde{g}\right)^2\!\right]\!\!-\!\partial_x^3\e^{-\tau\partial_x^3}\!\!\left[\left(\partial_x^{-1}\!\tilde{f}\right)^2\!\!\!-\!\left(\partial_x^{-1}\tilde{g}\right)^2\right],\partial_x^3\e^{-\tau\partial_x^3}\!\!\left[\!\left(\!\partial_x^{-1}\tilde{f}\right)^2\!\!\!-\!\left(\partial_x^{-1}\tilde{g}\right)^2\!\right]\right\rangle,}_{=:A_2}
	\end{align*}}
where we denoted by $\tilde{f}=\e^{-t_n\partial_x^3}f$ and $\tilde{g}=\e^{-t_n\partial_x^3}g$. Now let us estimate $A_2$ first. Letting $F:=\partial_x^3\e^{-\tau\partial_x^3}\left[\left(\partial_x^{-1}\tilde{f}\right)^2-\left(\partial_x^{-1}\tilde{g}\right)^2\right]$, we have
\begin{align*}
	|A_2|&=\left|\sum_{m=a+b}(-im^3)\overline{\hat{F}_{-m}}\left(\e^{i\tau (a^3+b^3)}-\e^{i\tau(a+b)^3}\right)\frac{1}{ab}\left(\hat{\tilde{f}}_a\hat{\tilde{f}}_b-\hat{\tilde{g}}_a\hat{\tilde{g}}_b\right)\right|\\
	&\leq \sum_{m=a+b}\frac{|a+b|^3}{|a||b|}|\overline{\hat{F}_{-m}}|\left|1-\e^{i\tau3ab(a+b)}\right|\left|\hat{\tilde{f}}_a\hat{\tilde{f}}_b-\hat{\tilde{g}}_a\hat{\tilde{g}}_b\right|\\
&\leq 3\tau \sum_{m=a+b}|a+b|^4|\overline{\hat{F}_{-m}}|\left|\hat{\tilde{f}}_a\hat{\tilde{f}}_b-\hat{\tilde{g}}_a\hat{\tilde{g}}_b\right|\\
&\leq 3\tau \sum_{m=a+b}\left|(-im)\overline{\hat{F}_{-m}}\right||a+b|^3\left|(\hat{\tilde{f}}_a-\hat{\tilde{g}}_a)\hat{\tilde{f}}_b+\hat{\tilde{g}}_a(\hat{\tilde{f}}_b-\hat{\tilde{g}}_b)\right|\\
&\leq 3\tau \sum_{m=a+b}\left|(-im)\overline{\hat{F}_{-m}}\right|(|a|^3+3|a|^2|b|+3|a||b|^2+|b|^3)\left[\left|(\hat{\tilde{f}}_a-\hat{\tilde{g}}_a)\hat{\tilde{f}}_b\right|+\left|\hat{\tilde{g}}_a(\hat{\tilde{f}}_b-\hat{\tilde{g}}_b)\right|\right]
\end{align*}
Thus by Cauchy--Schwarz we have, for some $c_1>0$,
\begin{align}\label{eqn:first_estimate_A_2_appendix}
	|A_2|\leq c_1\tau \|\partial_x F\|_{L^2}\left(\sum_{j=0}^{3}\left(\|\partial_x^{j}\tilde{g}\ast \partial_x^{3-j}(\tilde{f}-\tilde{g})\|_{L^2}+\|\partial_x^{j}\tilde{f}\ast \partial_x^{3-j}(\tilde{f}-\tilde{g})\|_{L^2}\right)\right).
\end{align}
We can now make use of the following observation:
\begin{claim} Let $j\in\{0,1,2,3\}$ then there is a constant $c_2>0$ such that for any functions $f,g\in H^3$ we have
	\begin{align}\label{eqn:convolution_bound_appendix}
		\|\partial_x^{j}f\ast\partial_x^{3-j}g\|_{L^2}\leq c_2\|f\|_{H^3}\|g\|_{H^3}.
	\end{align}
\end{claim}
\begin{subproof}[Proof of Claim]
Let us take without loss of generality $j>0$, then for $f^{(j)}(k):=|k|^{j}|\hat{f}_k|, k\in\mathbb{Z}\setminus\{0\}.$ we have by Young's inequality
\begin{align*}
	\|\partial_x^{j}f\ast\partial_x^{3-j}g\|_{L^2}&=\|f^{(j)}\ast g^{(3-j)}\|_{l^2}\\
	&\leq \|f^{(j)}\|_{l^1}\|g^{(3-j)}\|_{l^2}.
\end{align*}
Now using the Cauchy--Schwarz inequality we have $\|f^{(j)}\|_{l^1}\leq c_2\|f^{(3)}\|_{l^2}$ for some constant $c_2>0$ independent of $f$ which immediately implies the bound \eqref{eqn:convolution_bound_appendix}.
\end{subproof}
Thus, combining \eqref{eqn:first_estimate_A_2_appendix} and \eqref{eqn:convolution_bound_appendix} we conclude:
{\begin{align*}
	|A_2|\leq \tau c_3\|\partial_x F\|_{L^2}\|f-g\|_{H^3}\leq \tau \tilde{c}_3\|f-g\|_{H^3}^2,
\end{align*}
where $c_3,\tilde{c}_3>0$ are constants} which depends continuously on $\|f\|_{H^3},\|g\|_{H^3}$. We can now introduce $\tilde{F}:=\partial_x^3\left[\left(\e^{-\tau\partial_x^3}\partial_x^{-1}\tilde{f}\right)^2-\left(e^{-\tau\partial_x^3}\partial_x^{-1}\tilde{g}\right)^2\right]$ and follow exactly the same estimates to show
{\begin{align*}
	|A_1|\leq \tau c_4\|f-g\|_{H^3}^2,
\end{align*}}
where $c_4>0$ is a constant which depends continuously on $\|f\|_{H^3},\|g\|_{H^3}$. Therefore the result follows.
\end{proof}
\section{Proof of Lemma~\ref{lem:auxilliary_convolution_estimate}}\label{app:proof_of_auxilliary_convolution_estimate}
For completeness we recall the statement of the Lemma.
\begin{lemma}
	For any $j,l\in\mathbb{N}$ such that $j+l\geq 1$ there is a constant $c>0$ such that for all $f,g\in H^{j+l}$ and any $F\in L^2$ whose Fourier coefficients satisfy
	\begin{align*}
	\hat{F}_m\leq \sum_{m=a+b}|m|^l |\hat{f}_a||\hat{g}_b|, \quad \forall m\in\mathbb{Z},
	\end{align*}
	we have
	\begin{align*}
	\|F\|_{H^j}\leq c \|f\|_{H^{j+l}}\|g\|_{H^{j+l}}.
	\end{align*}
\end{lemma}
\begin{proof} The proof follows the arguments from \cite[Appendix~A]{maierhofer2022convergence} closely. We begin by noting that for any $r\geq 0$ there is a constant $C_r>0$ such that for all $a,b\in\mathbb{Z}$ we have
	\begin{align*}
	(|a|+|b|)^r\leq C_r(|a|^r+|b|^r).
	\end{align*}
	Thus we can estimate
	\begin{align}\nonumber
	\|F\|_{H^j}^2&=\sum_{m\in\mathbb{Z}\setminus\{0\}}|m|^{2j}|\hat{F}_m|^2\leq \sum_{m\in\mathbb{Z}\setminus\{0\}}\left(\sum_{m=a+b}|m|^{l+j} |\hat{f}_a||\hat{g}_b|\right)^2\\\nonumber
	&\leq\sum_{a\in\mathbb{Z}\setminus\{0\}}\left(\sum_{b\in\mathbb{Z}\setminus\{0\}}|a+b|^{l+j} |\hat{f}_a||\hat{g}_b|\right)^2\\\label{eqn:appendix2_intermediate_estimate_convolution_inequality}
	&\leq C \sum_{a\in\mathbb{Z}\setminus\{0\}}\left(\sum_{b\in\mathbb{Z}\setminus\{0\}}|a|^{l+j} |\hat{f}_a||\hat{g}_b|\right)^2+\sum_{a\in\mathbb{Z}\setminus\{0\}}\left(\sum_{b\in\mathbb{Z}\setminus\{0\}}|b|^{l+j} |\hat{f}_a||\hat{g}_b|\right)^2.
	\end{align}
	Now by the discrete Minkowski inequality we have
	\begin{align*}
	\sum_{a\in\mathbb{Z}\setminus\{0\}}\left(\sum_{b\in\mathbb{Z}\setminus\{0\}}|a|^{l+j} |\hat{f}_a||\hat{g}_b|\right)^2&\leq \left(\sum_{b\in\mathbb{Z}\setminus\{0\}}\left(\sum_{a\in\mathbb{Z}\setminus\{0\}}|a|^{2(l+j)}|\hat{f}_a|^2|\hat{g}_b|^2\right)^{\frac{1}{2}}\right)^2\\
	&\leq \|f\|_{H^{l+j}}^2\left(\sum_{b\in\mathbb{Z}\setminus\{0\}}|\hat{g}_b|\right)^2\leq \left(\sum_{b\in\mathbb{Z}\setminus\{0\}}|b|^{-2(j+l)}\right)\|f\|_{H^{l+j}}\|g\|_{H^{l+j}}
	\end{align*}
	where in the final line we used the Cauchy--Schwarz inequality. The final sum converges because $j+l\geq1$, thus applying a similar estimate to the second term in \eqref{eqn:appendix2_intermediate_estimate_convolution_inequality} concludes the proof.
\end{proof}
\end{appendix}
\end{document}